\documentclass[12pt]{amsart}
\usepackage{amsmath,amssymb,amsthm,tikz,float}

\newtheorem{theorem}{Theorem}
\newtheorem{proposition}[theorem]{Proposition}
\newtheorem{corollary}[theorem]{Corollary}
\newtheorem{lemma}[theorem]{Lemma}

\begin{document}
\title[Conformal invariance in 2-d critical bond percolation]{Conformal invariance of the exploration path in 2-d critical bond percolation in the square lattice}
\author{Jonathan Tsai, \ \ S. C. P. Yam, \ \ Wang Zhou \\}


\address{Department of Mathematics, the University of Hong Kong}
\email{jtsai@maths.hku.hk}

\address{Department of Statistics,
the Chinese University of Hong Kong,
Shatin, New Territories, Hong Kong}
\email{scpyam@sta.cuhk.edu.hk}

\address{Department of Statistics and Applied Probability, National University of
 Singapore, Singapore 117546}
\email{stazw@nus.edu.sg}


\begin{abstract}
In this paper we present the proof of the convergence of the critical bond percolation exploration process on the square lattice to the trace of SLE$_{6}$. This is an important conjecture in mathematical physics and probability. The case of critical site percolation on the hexagonal lattice was established in the seminal work of Smirnov in \cite{smirnov01a,smirnov01b} via proving Cardy's formula.  Our proof uses a series of transformations and conditioning to construct a pair of paths: the $+\partial$CBP and the $-\partial$CBP. The convergence in the site percolation case on the hexagonal lattice allows us to obtain certain estimates on the scaling limit of the $+\partial$CBP and the $-\partial$CBP. By considering a path which is the concatenation of $+\partial$CBPs and $-\partial$CBPs in an alternating manner, we can prove the convergence in the case of bond percolation on the square lattice.\end{abstract}

\maketitle

\section{Introduction}

Percolation theory, going back as far as Broadbent and
Hammersley \cite{BH57},  describes the flow of fluid in a porous medium  with stochastically blocked channels.
In terms of mathematics, it consists in removing each edge (or each vertex) in a lattice with a given probability $p$. In these days, it
has become part of the mainstream in probability and statistical physics. One can refer to Grimmett's book~\cite{Grim} for more background.
Traditionally, the study of percolation was concerned with the critical probability that is with respect to the question of whether or not there exists an infinite
open cluster -- bond percolation on the square lattice
and site percolation on the hexagonal lattice are critical for $p = 1/2$. This tradition is due to many reasons. One originates from physics: at the critical probability, a phase transition occurs. Phase
transitions are among the most striking phenomena in physics. A small change in an environmental parameter,
such as the temperature or the external magnetic field, can induce huge changes in the macroscopic properties of a system.
Another one is from mathematics: the celebrated `conformal invariance' conjecture of Aizenman and
Langlands, Pouliotthe and Saint-Aubin~\cite{LPS94} states that the probabilities of some macroscopic events have conformally invariant limits at criticality which
turn out to be very helpful in understanding discrete systems.
This conjecture was expressed in another form by Schramm~\cite{sch00}. In his seminal paper, Schramm introduced the percolation exploration
path which separates macroscopic open clusters from closed ones and conjectured that this path converges to his conformally invariant
Schramm-Loewner evolution (SLE) curve as the mesh size of the lattice goes to zero.

For critical site percolation on the hexagonal lattice, Smirnov \cite{smirnov01a,smirnov01b} proved the conformal invariance of the scaling limit of
crossing probabilities given by Cardy's formula. Later on, a detailed proof of the convergence of the critical site percolation
exploration path to SLE$_6$ was provided by Camia and Newman \cite{CN07}.
This allows one to use the SLE machinery \cite{LSW01a,LSW01b} to
obtain new interesting properties of critical site percolation, such as the values of some critical
exponents which portray the limiting behavior of the probabilities of certain
exceptional events (arm exponents) \cite{LSW02, SW01}. For a review one can refer to \cite{werner07}.

Usually a slight move in one part may affect the whole situation. 
But there is still no proof of convergence of the critical percolation
exploration path on general lattices, especially the square lattice, to SLE$_6$. The reason is that the proofs in the site percolation on the hexagonal lattice case depend heavily on the particular properties of the hexagonal lattice.

However much progress has been made in recent years, thanks to SLE,
in understanding the geometric and topological properties of (the scaling limit of)
large discrete systems.
Besides the percolation exploration path on the triangular lattice, many random self-avoiding lattice paths from the statistical
physics literature are proved to have  SLE as scaling limits, such as
loop erased random walks and uniform spanning tree Peano paths
\cite{LSW04}, the harmonic explorer's path~\cite{SS05}, the level lines of the discrete Gaussian free field~\cite{SS09},
the interfaces of the FK Ising model~\cite{smirnov10}.

In this paper we shall prove the convergence of the exploration path of the critical bond percolation on the square lattice (which is an interface between open and closed edges after certain boundary conditions have been applied -- see Figure \ref{figg1} and \ref{figg2}) to the trace of SLE$_{6}$ and as a consequence the conformal invariance of the scaling limit is established. Let $D$ be a domain in $\mathbb{C}$ and $a,b\in\partial D$. First we consider the following metric on curves from $a$ to $b$ in $D$:
\begin{equation}\rho( \nu_{1},\nu_{2})=\inf_{\sigma}\sup_{t\in[0,1]} |\nu_{1}(t)-\nu_{2}(\sigma(t))|, \label{metric} \end{equation}
\pagebreak
\begin{figure}[hp]
 \begin{center}
\scalebox{1}{\includegraphics{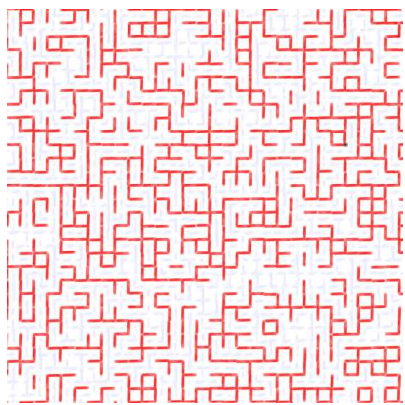}}
 \end{center}
\caption{Bond percolation on the square lattice. The open edges are marked in red.} \label{figg1}
\end{figure}
\begin{figure}[hp]
 \begin{center}
\scalebox{1}{\includegraphics{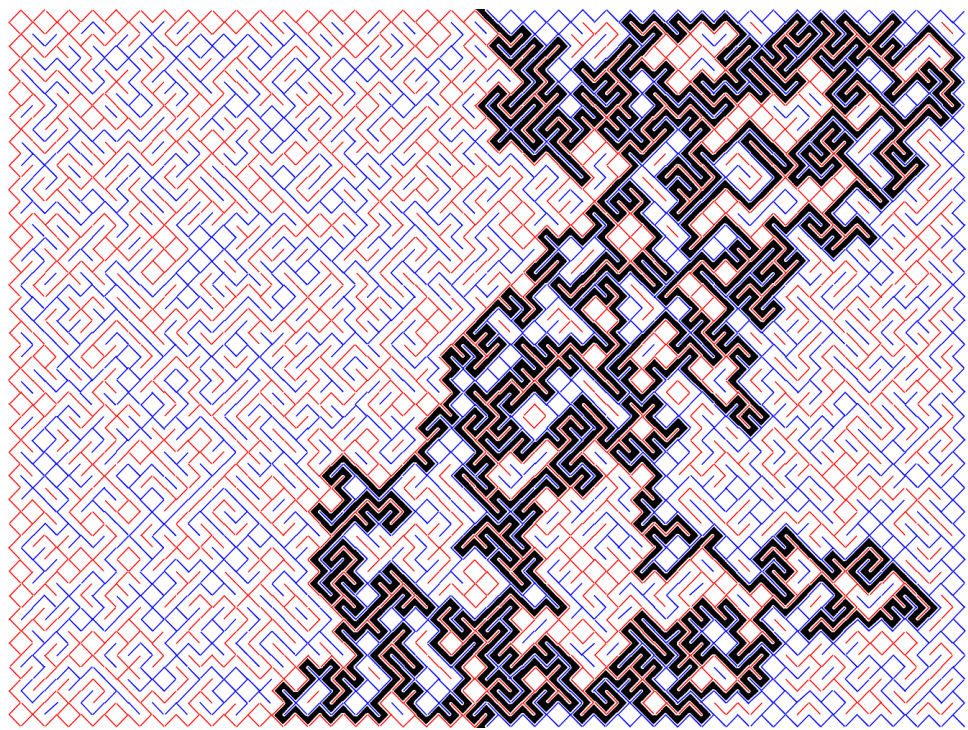}}
 \end{center}
\caption{The bond percolation exploration process lies in the ``corridor'' between the red edges and blue edges marked in black.} \label{figg2}
\end{figure}

\noindent
where $\nu_{1}[0,1]$, $\nu_{2}[0,1]$ are any two curves from $a$ to $b$ in $D$ and the infimum is taken over all reparamaterizations $t\mapsto\sigma(t)$. Here $\sigma:[0,1]\rightarrow [0,1]$ is a continuous surjective non-decreasing function.

\begin{theorem}\label{main}
The critical bond percolation exploration path on the square lattice converges in distribution in the metric given by (\ref{metric}) to the trace of SLE$_{6}$ as the mesh size of the lattice tends to zero.
\end{theorem}

The idea of the proof is as follows: By replacing the hexagonal sites in the hexagonal lattice with rectangular sites and then shifting each row left and right in an alternating manner, we can change the hexagonal lattice into a rectangular lattice. The site percolation exploration path on the hexagonal lattice then induces a pair of paths on the rectangular lattice: the $+$BP and $-$BP. The $-$BP is the reflection of the $+$BP across the $y$-axis. By construction of the lattice modification, the $+$BP and $-$BP lie in a $2\delta$ neighbourhood of the site percolation exploration path and hence, in particular, both converge to SLE$_{6}$ in the scaling limit. We call a vertex of the $+$BP path (respectively the $-$BP path) where the path has 3 choices for the next vertex \emph{free}. The free vertices are precisely the vertices where the next step is not blocked by vertices that the path has previously visited (or the boundary).

We then condition the $+$BP (respectively the  $-$BP) not to go in the same direction for two consecutive edges at the  free vertices only and call the conditioned path the $+$CBP (respectively the $-$CBP). We then condition further such that the $+$CBP (respectively the $-$CBP) does not go in the same direction for two consecutive edges at all vertices and call the conditioned path the $+\partial$CBP (respectively the $-\partial$CBP).

  It turns out that the bond percolation exploration path on the square lattice squashed to a path on the rectangular lattice can be constructed by alternating pastings of $+\partial$CBP and $-\partial$CBP paths; moreover, the $+\partial$CBP (respectively the $-\partial$CBP) can be coupled with a sequence of $+$CBP paths (respectively $-$CBP paths) such that their Loewner driving functions are close.  This means that it is sufficient to study the $+$CBP.

Indeed, we shall show that the driving function of the $+$CBP converges subsequentially to an \emph{$\epsilon$-semimartingale}: the sum of a local martingale and a finite $(1+\epsilon)$-variation process for all $\epsilon>0$ sufficiently small.  The idea of this part of the proof is as follows: For simplicity, let us consider the upper half-plane $\mathbb{H}=\{z\in\mathbb{C}:\mathrm{Im}[z]>0\}$ in the complex plane and the $+$CBP path from $0$ to $\infty$ on the lattice of mesh-size $\delta$. For clarity purposes, we hide the dependence on $\delta$ in the following notation. Let $\gamma:[0,\infty)\rightarrow \overline{\mathbb{H}}$ be the path parameterized by the half-plane capacity and let $g_{t}$ be the associated conformal mappings $g_{t}:\mathbb{H}\setminus\gamma(0,t]\rightarrow \mathbb{H}$ that are hydrodynamically normalized. Then using a result in \cite{tsai09}, we can write the Loewner driving function of the path $\gamma$ as
\begin{equation} \xi_{t}=\frac{1}{2}\big[a_{1}(t)+b_{1}(t)+\sum_{k=2}^{N(t)}L_{k}(a_{k}(t)-b_{k}(t))\big],\label{wweq6}\end{equation}
where for each $k$, $a_{k}(t)>b_{k}(t)$ are the preimages of the $k$th vertex of the path $\gamma$ under the conformal mapping $g_{t}^{-1}$; $L_{k}$ is $-1$ if the curve turns left at the $k$th step and +1 if the curve turns right at the $k$th step; and $N(t)$ is the number of the vertices on the path $\gamma[0,t]$.  Let $t_{0},t_{1},t_{2},t_{3},\ldots$ denote the times at which the curve $\gamma$ is at each vertex of the path. We also choose $0=m_{0}<m_{1}<m_{2}<\ldots$ random steps adapted to the process.

Then if we let $M_{n}\triangleq \xi_{t_{m_{n}}}$, we have
\begin{equation}
M_{n}-M_{n-1}
 =R_{n-1}(t_{m_{n}})-R_{n-1}(t_{m_{n-1}})+ \frac{1}{2}\sum_{k=m_{n-1}+1}^{m_{n} } L_{k} \big( a_{k}(t_{m_{n}})-b_{k}(t_{m_{n}}) \big),
\label{qqeq2}\end{equation}
where
\begin{equation}R_{n-1}(t)=\frac{1}{2}\big[ a_{1}(t)+b_{1}(t)+\sum_{k=2}^{m_{n-1}} L_{k}(a_{k}(t)-b_{k}(t)) \big] .\label{qqeq7}\end{equation}
Since $a_{m_n}(t_{m_{n}})=b_{m_n}(t_{m_{n}})$ we can write
\[ a_{k}(t_{m_{n}})-b_{k}(t_{m_{n}})=\sum_{j=k}^{m_{n}-1} \Delta_{j,n} , \]
where
\[\Delta_{j,n}= \left[ (a_{j}(t_{m_{n}})-a_{j+1}(t_{m_{n}}))-(b_{j}(t_{m_{n}})-b_{j+1}(t_{m_{n}})\right] . \]
Then, we can telescope the sum in (\ref{qqeq2}) and take conditional expectations to get
\begin{eqnarray}
\mathbb{E}\left[M_{n}-M_{n-1} | \mathcal{F}_{m_{n-1}}\right]
 &=&\mathbb{E}\left[R_{n-1}(t_{m_{n}})-R_{n-1}(t_{m_{n-1}})|\mathcal{F}_{m_{n-1}}\right] \nonumber \\
 &&\quad \quad + \frac{1}{2}\sum_{j=m_{n-1}+1}^{m_{n}-1 }\mathbb{E}\big[  \Delta_{j,n}\sum_{k=m_{n-1}+1}^{j} L_{k}|\mathcal{F}_{m_{n-1}}\big].
\label{qqeq3} \end{eqnarray}
Using the convergence of the $+$BP path to SLE$_{6}$ and by a particular choice of $\{m_{n}\}$, we deduce that we can decompose for sufficiently small mesh-size $\delta$,
\[\mathbb{E}\big[  \Delta_{j,n}\sum_{k=m_{n-1}+1}^{j} L_{k}|\mathcal{F}_{m_{n-1}}\big]\approx\mathbb{E}\big[  \Delta_{j,n}\big]\mathbb{E}\big[\sum_{k=m_{n-1}+1}^{j} L_{k}|\mathcal{F}_{m_{n-1}}\big].\]
From the definition of $(L_{k})$, using a symmetry argument, one should be able to show that
\[\mathbb{E}\big[\sum_{k=m_{n-1}+1}^{j} L_{k}|\mathcal{F}_{m_{n-1}}\big]\approx 0.\]
(at least sufficiently far from the boundary). This would imply that
\[\mathbb{E}\left[M_{n}-M_{n-1} | \mathcal{F}_{m_{n-1}}\right]
\approx \mathbb{E}\left[R_{n-1}(t_{m_{n}})-R_{n-1}(t_{m_{n-1}})|\mathcal{F}_{m_{n-1}}\right].\]
Hence
\[M_{n}-\sum_{k=1}^{n} [R_{k-1}(t_{m_{k}})-R_{k-1}(t_{m_{k-1}})]\] is almost a martingale. By telescoping the sum in (\ref{qqeq7}), we can show that
\[\big|\sum_{k=1}^{n} [R_{k-1}(t_{m_{k}})-R_{k-1}(t_{m_{k-1}})]\big|\leq \mathcal{W}^{\delta}|A(t_{m_{k}})-A(t_{m_{k-1}})|+|B(t_{m_{k}})-B(t_{m_{k-1}})|\]
where $A$ and $B$ are finite variation processes and
\[\mathcal{W}^{\delta}=\max_{j=2,\ldots m_{n-1}}\big|\sum_{k=2}^{j}L_{k}\big|\]

Since the moments of $\mathcal{W}^{\delta}$ are all bounded, we can use a stronger version of the Kolmogorov-Centsov continuity theorem (contained in the Appendix) to show that
\[\sum_{k=1}^{n} R_{k-1}(t_{m_{k}})-R_{k-1}(t_{m_{k-1}})\]
is a finite $(1+\epsilon)$-variation process for all $\epsilon>0$ sufficiently small as $\delta\searrow 0$. Hence we should be able to embed $M_{n}$ into a
continuous time $\epsilon$-semimartingale $M_{t}$ so that $\xi_{t}$ should converge to $M_{t}$ as the mesh size $\delta\searrow 0$. Then the locality property of the scaling limit and an infinite divisibility argument can be used to show that we must have $M_{t}=\sqrt{6}B_{t}$ where $B_{t}$ is standard 1-dimensional Brownian motion.

Hence we have the driving term convergence of the bond percolation exploration process to SLE$_{6}$. We can then get the convergence of the path to the trace of SLE$_{6}$ either by considering the 4 and 5-arm percolation estimates (as in \cite{CN07}) or using the recent result of Sheffield and Sun \cite{SS10} and repeating a similar argument.

This paper is organized as follows:
\begin{enumerate}
\item[Section 2:] We present the notation to be used in this paper.
\item[Section 3:] We introduce the bond percolation exploration path on the square lattice.
\item[Section 4:] We discuss the lattice modification and the restriction procedure that will allow us to define the $\pm$CBP and $\pm\partial$CBP from the site percolation exploration path on the hexagonal lattice.
\item[Section 5:] We derive the formula for the driving function on lattices (i.e. the formula (\ref{wweq6}) above).
\item[Section 6:] We use the convergence of the site percolation exploration path to SLE$_{6}$ in order to obtain certain useful estimates.
\item[Section 7:] We obtain certain estimates for the driving function of the $+$CBP.
\item[Section 8:] We apply the estimates obtained in Section 5 and Section 6 to obtain the subsequential driving term convergence to an $\epsilon$-semimartingale of the $+$CBP. By coupling the $+$CBP with the $+\partial$CBP, we obtain the subsequential driving term convergence of the $+\partial$CBP to an $\epsilon$-semimartingale as well.
\item[Section 9:] Using the convergence obtained in Section 7, we deduce convergence of the driving term of the bond percolation exploration path on the square lattice to an  $\epsilon$-semimartingale. The locality property then implies, via an infinite divisibility argument, that the semimartingale must in fact be $\sqrt{6}B_{t}$.
\item[Section 10:] We discuss how we can obtain the curve convergence from the driving term convergence obtained in Section 9; hence proving Theorem \ref{main}.
\item[Appendix:] We prove a stronger version of the Kolmogorov-Centsov continuity theorem.
\end{enumerate}

\section*{Acknowledgments}
Jonathan Tsai and Phillip Yam acknowledge the financial support from The Hong Kong RGC GRF 502408. Phillip Yam also expresses his sincere gratitude to the hospitality of Hausdorff Center for Mathematics of the University of Bonn and Mathematisches Forschungsinstitut Oberwolfach (MFO) during the preparation of the present work.  Wang Zhou was partially supported by a grant R-155-000-116-112 at the National University of Singapore.

\section{Notation}\label{notation}
We consider ordered triples of the form $\mathbf{D}=(D,a,b)$ where $D\subsetneq \mathbb{C}$ is a simply-connected domain and $a,b,\in\partial D$ with $a\neq b$ such that $a,b$ correspond to unique prime-ends of $D$. We say such a triple is admissible. Let $\mathcal{D}$ be the set of all such triples. By the Riemann mapping theorem, for any $\mathbf{D}\in\mathcal{D}$ we can find a conformal map $\phi_{\mathbf{D}}$ of $\mathbb{H}$ onto $D$ with $\phi_{\mathbf{D}}(0)=a$ and $\phi_{\mathbf{D}}(\infty)=b$.

For a given lattice $\mathbb{L}$, we define $\mathcal{D}^{\mathbb{L}}$ such that for any $(D,a,b)\in\mathcal{D}^{\mathbb{L}}$, the boundary of $D$ is the union of vertices and edges of the lattice, and $a,b\in\partial D$ are vertices of the lattice such that there is a path on the lattice from $a$ to $b$ contained in $D$. 
  We say that a path $\Gamma$, from $a$ to $b$ in $D$ is a \emph{non-crossing path} if  it is the limit of a sequence of simple paths from $a$ to $b$ in $D$.

Consider $\mathbf{D}=(D,a,b)\in\mathcal{D}^{\mathbb{L}}$. Let $\nu$ be a simple path from $a$ to $b$ on the lattice $\mathbb{L}$. Then $\phi_{\mathbf{D}}^{-1}(\nu)$ is a path in $\mathbb{H}$ from $0$ to $\infty$. We define $\gamma:[0,\infty)\mapsto \mathbb{H}$ be the curve $\phi_{\mathbf{D}}^{-1}(\nu)$ such that $\gamma$ is parameterized by half-plane capacity (see \cite{lawlerbook}). Let $Z_{0}, Z_{1},\ldots, Z_{n}$ be the images under $\phi_{\mathbf{D}}^{-1}$ of the vertices of the path $\nu$. Then $Z_{k}$ is a point on the curve $\gamma(t)$. We denote the time corresponding to $Z_{k}$ by $t_{k}$ i.e. $\gamma(t_{k})=Z_{k}$.
\pagebreak
\begin{figure}[hp]
 \begin{center}
\scalebox{0.5}{\includegraphics{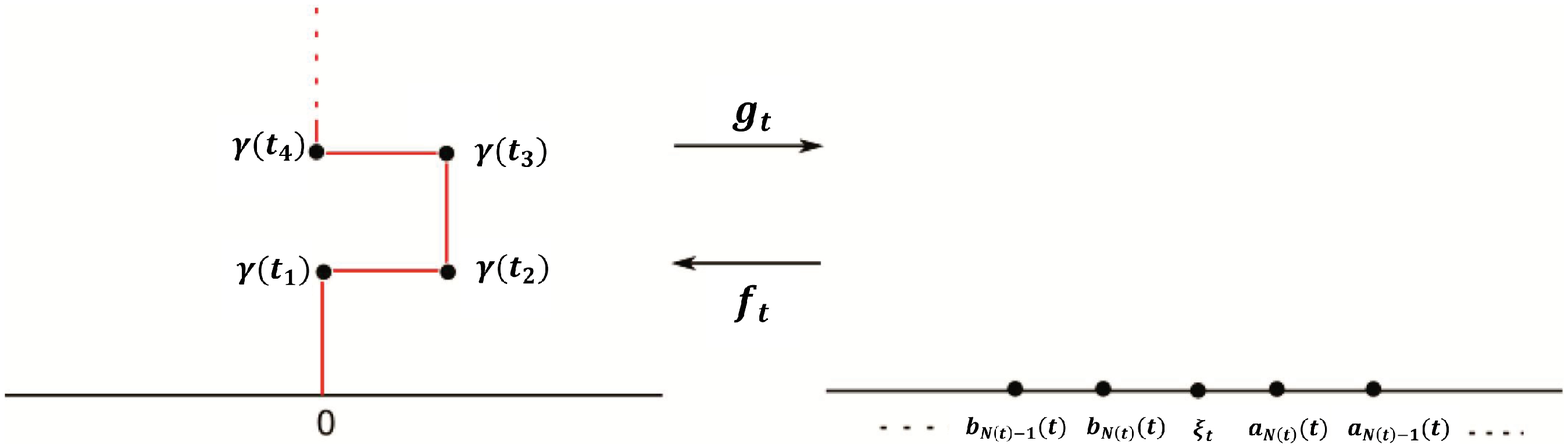}}
 \end{center}
\caption{} \label{nota1}
\end{figure}

Now suppose that $f_{t}$ is the conformal map of $\mathbb{H}$ onto $H_{t}=\mathbb{H}\setminus\gamma[0,t]$ satisfying the hydrodynamic normalization:
\begin{equation}f_{t}(z)=z-\frac{2t}{z}+O\big(\frac{1}{z^{2}}\big) \text{ as } z\rightarrow\infty.\label{zeq6}\end{equation}
The function $f_{t}$ satisfies the chordal Loewner differential equation \cite{lawlerbook}
\[\dot{f}_{t}(z)=-f_{t}'(z)\frac{2}{z-\xi(t)},\]
where $\xi(t)=f_{t}^{-1}(\gamma(t))$ is the chordal driving function. The inverse function $g_{t}=f_{t}^{-1}$ satisfies
\[\dot{g}_{t}(z)=\frac{2}{g_{t}(z)-\xi(t)}.\]

We define $N(t)$ to be the largest $k$ such that $t_{k}\leq t$. Then for $1\leq k\leq N(t)$, we define
$a_{k}(t)$ and $b_{k}(t)$ to be the two preimages of $\phi_{\mathbf{D}}^{-1}(Z_{k})$ under $f_{t}$ such that $b_{k}(t)\leq a_{k}(t)$ (see Figure \ref{nota1}).
Then since $a_{k}(t),b_{k}(t)$ are the images of $\phi_{\mathbf{D}}^{-1}(Z_{k})$ under $g_{t}$, they satisfy
\begin{equation}\dot{a}_{k}(t)=\frac{2}{a_{k}(t)-\xi(t)} \text{ and } \dot{b}_{k}(t)=\frac{2}{b_{k}(t)-\xi(t)} \label{zeq2}\end{equation}
for $t\in(t_{k},\infty)$.

\section{The bond percolation exploration path on the square lattice}\label{sect2}
We consider critical bond percolation on the square lattice $L$: between every two adjacent vertices, we add an edge between the vertices with probability $1/2$. Let $E$ be the collection of such edges. Consider the dual lattice of the square lattice by considering the vertices positioned at the centre of each square on the square lattice. Between two adjacent vertices on the dual lattice $L^{*}$, we add an edge to the dual lattice if there is no edge in $E$ separating the two vertices. Let $E^{*}$ be the collection of such edges on the dual lattice. We now rotate the original lattice and its dual lattice by $\pi/4$ radians anticlockwise about the origin. Note that the lattice which is the union of $L$ and $L^{*}$  is also a square lattice (but of smaller size). We denote this lattice by $\mathbb{L}^{Sq}$ -- by scaling we can assume that the mesh-size (i.e. the side-length of each square on the lattice) is 1. Another way of constructing $E$ and $E^{*}$ is as follows: For each square in $\mathbb{L}^{Sq}$, we add a (diagonal) edge between a pair of the diagonal vertices or we add a (diagonal) edge between the alternate pair of diagonal vertices each with probability 1/2. Then $E$ is the collection of the diagonal edges in $\mathbb{L}^{Sq}$ that join vertices of $L$ and $E^{*}$ is the collection of diagonal edges in $\mathbb{L}^{Sq}$ that join vertices of $L^{*}$.

Now consider a simply-connected domain $D\subsetneq\mathbb{C}$ such that the boundary of $D$ is on the lattice $\mathbb{L}^{Sq}$ and consider $a,b\in\partial D\cap\mathbb{L}^{Sq}$. We apply the following boundary conditions: in the squares in $\mathbb{L}^{Sq}$ that are on the boundary to the left of $a$ up to $b$, we join the vertices of $L$;  in the squares in $\mathbb{L}^{Sq}$ that are on the boundary to the right of $a$ up to $b$, we join the vertices of $L^{*}$. For the interior squares, we join the edges using the above method of constructing $E$ and $E^{*}$.  See Figure \ref{fig1}.

Then there is a continuous path $\Gamma$ from $a$ to $b$ on the dual lattice of $\mathbb{L}^{Sq}$ that does not cross the edges of $E$ and $E^{*}$ such that the edges to the right of $\Gamma$ are in $E^{*}$ and the edges to the left of $\Gamma$ are in $E$ . We call the path $\Gamma$ the \emph{bond percolation exploration path from $a$ to $b$ on the square lattice} (abbreviated \emph{SqP}). Similarly, we can define the percolation exploration path on the square lattice of mesh size $\delta$, $\delta \mathbb{L}^{Sq}$, for some $\delta>0$. See Figure \ref{fig2}. The SqP is not a simple path since it can intersect itself at the corner of the squares; however, it is a non-crossing path.

Then at every vertex of the SqP, the path turns left or right each with probability $1/2$ except when turning in one of the directions will result in the path being blocked (i.e the path can no longer reach the end point $b$) -- in this case the path is forced to go in the alternate direction. We shall use this as the construction of the bond percolation exploration path. We say that a vertex $V$ of a SqP is $\emph{free}$ if there are two possible choices for the next vertex; otherwise, we say that $V$ is \emph{non-free}.

The SqP satisfies the \emph{locality property}. This means that for any domain $D$ such that $0\in\partial D$ and $D\cap\mathbb{H}\not=\emptyset$, we can couple an SqP in $(D,0,b)$ with an SqP in $(\mathbb{H},0,\infty)$ up to first exit of $D\cap\mathbb{H}$.
\pagebreak
\begin{figure}[hp]
 \begin{center}
\scalebox{0.5}{\includegraphics{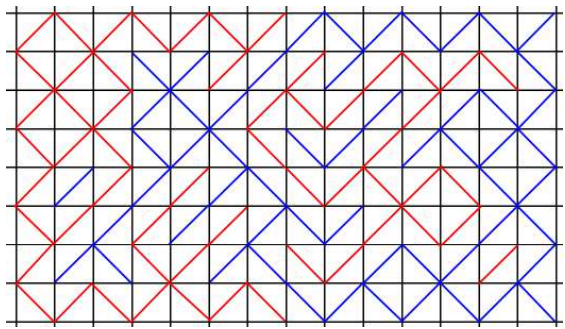}}
 \end{center}
\caption{Critical bond percolation on the square lattice in a rectangle with boundary conditions. The red edges form $E$ and the blue edges form $E^{*}$.} \label{fig1}
\end{figure}
\begin{figure}[hp]
 \begin{center}
\scalebox{0.5}{\includegraphics{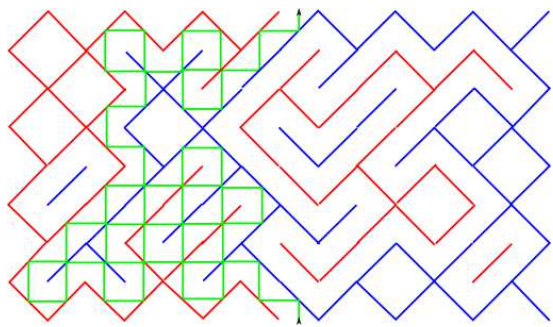}}
 \end{center}
\caption{The green path is the bond percolation exploration path.} \label{fig2}
\end{figure}
\section{Modification of the hexagonal lattice}\label{sect3}

We abbreviate the percolation exploration path on the hexagonal lattice as \emph{HexP}. Consider the following modification of the hexagonal lattice of mesh size $\delta>0$: for each hexagonal site on the lattice, we replace it with a rectangular site such that the rectangles tessellate the plane (see Figure \ref{fig3}). Each rectangle contains six vertices on its boundary: 4 at each corner and 2 on the top and bottom edges of the rectangle. We call this lattice the \emph{brick-wall lattice} of mesh size $\delta$.
\pagebreak
\begin{figure}[hp]
 \begin{center}
\scalebox{0.3}{\includegraphics{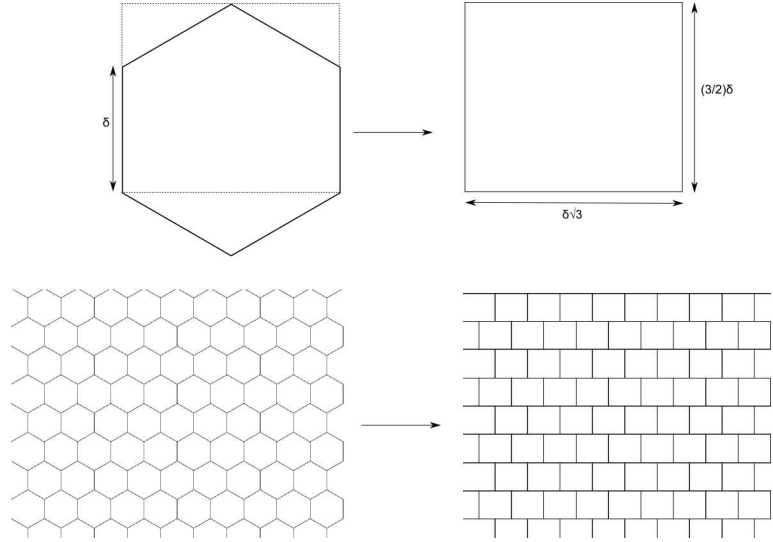}}
 \end{center}
\caption{The modification from the hexagonal lattice to the brick-wall lattice.} \label{fig3}
\end{figure}
\begin{figure}[hp]
 \begin{center}
\scalebox{0.5}{\includegraphics{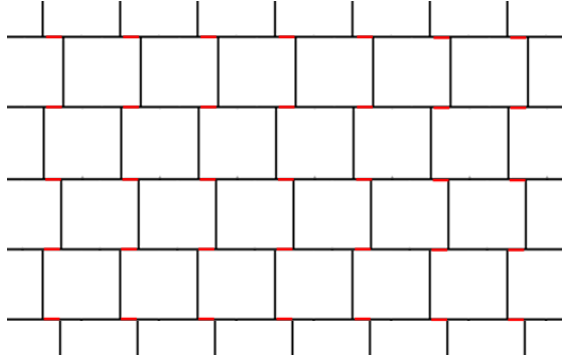}}
 \end{center}
\caption{The $\epsilon$-brick-wall lattice. The red edges have Euclidean length $\epsilon\delta$.} \label{fig4}
\end{figure}
 It is clear that the brick wall lattice is topologically equivalent to the hexagonal lattice. Let $w=\sqrt{3}$ denote the (horizontal) width of each rectangle of this lattice (i.e. the length of the base) when $\delta=1$.

We label the rows of the lattice by the integers such that the row containing the real-line is labeled as 0. For sufficiently small, fixed $\epsilon>0$, we shall now modify the brick wall lattice in $\mathbb{C}$ in the following way (see Figure \ref{fig4}):
\begin{enumerate}
\item For $k=2n$ for some $n\in\mathbb{Z}$, shift the $k$th row left by $(w/2-\epsilon)\delta/2$;
\item  For $k=2n+1$ for some $n\in\mathbb{Z}$, shift the $k$th row right by $(w/2-\epsilon)\delta/2$;
\end{enumerate}
 We call the resulting lattice (which is still topologically equivalent to the hexagonal lattice) the \emph{$\epsilon$-brick-wall lattice}.
For $\epsilon<0$ with $-\epsilon$ sufficiently small, we define the \emph{$\epsilon$-brick-wall lattice} to be the reflection of the \emph{$(-\epsilon)$-brick-wall lattice} across the $y$-axis. Note that as $\epsilon\searrow 0$ or $\epsilon\nearrow 0$ the $\epsilon$-brick-wall lattice tends to a rectangular lattice which we call the \emph{shifted brick wall lattice}.

Then we can find a function $\Phi_{\epsilon}$ which satisfies:
\begin{enumerate}
\item $\Phi_{\epsilon}$ maps the vertices of the hexagonal lattice 1--1 and onto the vertices of the $\epsilon$-brick-wall lattice.
\item For any path $\Gamma$ on the hexagonal lattice, $\Phi_{\epsilon}(\Gamma)$ is a path on the $\epsilon$-brick-wall lattice such that $\Phi_{\epsilon}(\Gamma)$ is contained in a $3\delta$-neighbourhood of $\Gamma$.
\end{enumerate}

We now suppose that $\nu$ is HexP in some domain $D$ from $a$ to $b$ on the lattice of mesh-size $\delta$. Let $D^{*}=\Phi_{\epsilon}(D)$ and $a^{*}=\Phi_{\epsilon}(a)$, $b^{*}=\Phi_{\epsilon}(b)$ denote the corresponding domain and fixed boundary points on the $\epsilon$-brick wall lattice. Consider the path $\nu_{\epsilon}=\Phi_{\epsilon}(\nu)$ on the $\epsilon$-brick wall lattice. We call $\nu_{\epsilon}$ the \emph{hexagonal lattice percolation exploration path on the $\epsilon$-brick wall lattice} . We abbreviate it as \emph{$\epsilon$-BP}.

\begin{figure}[hp]
 \begin{center}
\scalebox{0.7}{\includegraphics{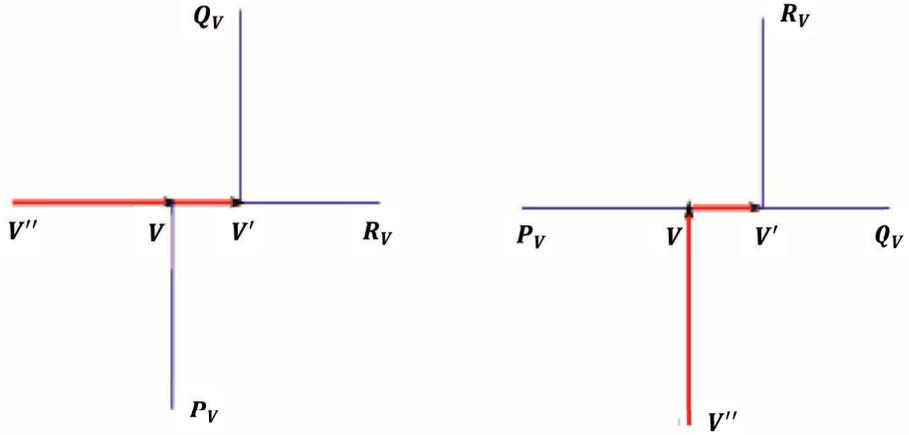}}
 \end{center}
\caption{The above situations (and their rotations and reflections) illustrate all the possible cases for $V, V', V'' , P_{V},Q_{V},R_{V}$.} \label{fig5}
\end{figure}

Now take any boundary vertices $a^{*}$, $b^{*}$ of a domain $D^{*}$ on the $\epsilon$-brick-wall lattice and take any simple path $\pi_{a^{*}\rightarrow V}$ on the lattice starting from $a^{*}$ and ending at an interior vertex $V$ of $D^{*}$ such that the final edge of the path is not of Euclidean length $\epsilon\delta$. Let $V'$ denote the vertex on the lattice connected to $V$ by an edge of Euclidean length $\epsilon\delta$; and let $V''$ be the second last vertex of the path.  Let $P_{V}$, not equal to either $V'$ or $V''$  be the remaining neighboring vertex of $V$. Also let $Q_{V}, R_{V}$  be the neighboring vertices of $V'$ not equal to $V$ (see Figure \ref{fig5}) such that the edge from $V'$ to $R_{V}$ is parallel to the edge from $V''$ to $V$. In other words, the edge from $Q_V$ to $V'$ is perpendicular to the edge from $V''$ to $V$. We say that the path $\pi_{a^{*}\rightarrow V}$ \emph{leaves an unblocked path to $b^{*}$} if it satisfies the following conditions.

\begin{enumerate}
\item We can continue the path $\pi_{a^{*}\rightarrow V}$ from $a^{*}$ to $V$ to a simple path from  $a^{*}$ to $b^{*}$ in $D^{*}$ such that the next vertex after $V$ is $P_{V}$;
\item We can continue the path $\pi_{a^{*}\rightarrow V}$ to a simple path from $a^{*}$ to $b^{*}$ in $D^{*}$ such that the next vertex after $V$ is $V'$ and the next vertex after $V'$ is $Q_{V}$;
\item We can continue the path $\pi_{a^{*}\rightarrow V}$ to a simple path from $a^{*}$ to $b^{*}$ in $D^{*}$ such that the next vertex after $V$ is $V'$ and the next vertex after $V'$ is $R_{V}$.
\end{enumerate}

Suppose that $\nu_{\epsilon}$ is an $\epsilon$-BP from $a^{*}$ to $b^{*}$ in $D^{*}$. Let $X_{0}^{\epsilon},X_{1}^{\epsilon},\ldots$ denote the vertices of $\nu_{\epsilon}$.  We say that $X_{j}^{\epsilon}$ is an \emph{unblocked} vertex of $\nu_{\epsilon}$ if the subpath of $\nu_{\epsilon}$ from $a^{*}$ to $X_{j}^{\epsilon}$ leaves an unblocked path to $b^{*}$.
 \begin{lemma}\label{lem1}
Conditioned on the event that $X_{j}^{\epsilon}$ is an unblocked vertex of an $\epsilon$-BP, we have
\[\mathbb{P}[X_{j+1}^{\epsilon}=P_{X_{j}^{\epsilon}}|X_{j}^{\epsilon} \text{ is unblocked}]=\frac{1}{2},\]
\begin{eqnarray*}
&&\mathbb{P}[X_{j+1}^{\epsilon}=(X_{j}^{\epsilon})', \ X_{j+2}^{\epsilon}
=Q_{X_{j}^{\epsilon}}|X_{j}^{\epsilon} \text{ is unblocked}]
\\&=&\mathbb{P}[X_{j+1}^{\epsilon}=(X_{j}^{\epsilon})', \  X_{j+2}^{\epsilon}=R_{X_{j}^{\epsilon}}|X_{j}^{\epsilon} \text{ is unblocked}]=\frac{1}{4}.
\end{eqnarray*}
\end{lemma}

\begin{proof}
We compare with the corresponding probabilities of the HexP.
\end{proof}

We now define the paths $\nu^{+}$ and $\nu^{-}$ by
\[\nu^{+}=\lim_{\epsilon\rightarrow 0+}\Phi_{\epsilon}(\nu) ,\]
and
\[\nu^{-}=\lim_{\epsilon\rightarrow 0-}\Phi_{\epsilon}(\nu) .\]
$\nu^{+}$ and $\nu^{-}$ are non-crossing paths on the shifted brick-wall lattice. We call $\nu^{+}$ the \emph{right percolation exploration process on the shifted brick-wall lattice} (abbreviated as +BP) and $\nu^{-}$ the \emph{left percolation exploration process on the shifted brick-wall lattice} (abbreviated as $-$BP). Then we can couple an $\epsilon$-BP with $\nu^{+}$ and $\nu^{-}$.
\begin{lemma}\label{lem2}
There is a coupling of a HexP, $\epsilon$-BP, $(-\epsilon)$-BP,  +BP, $-$BP such that for sufficiently small $\epsilon>0$, each of these paths is contained within a 2$\delta$ neighbourhood of any other.
\end{lemma}
\begin{proof}
This directly follows from the above definition.
\end{proof}
The aim is to modify the $+$BP and $-$BP to make them closer to the SqP. The SqP has only two possibilities for the next vertex at each vertex of the path and each edge of the path is perpendicular to the path. Hence, we need to condition the $+$BP and $-$BP not to go straight at each vertex. We do this in two steps. Firstly we look at the unblocked vertices of the $+$BP  and $-$BP, and we want to prevent it from going straight. At the unblocked vertices, we condition the $\epsilon$-BP such that the $+$BP does not go straight to create a new path -- the $+$CBP -- as follows.

Let $X_{0},X_{1},X_{2},\ldots$ denote the vertices of the HexP, $\nu$, and let $X_{0}^{+},X_{1}^{+}, X_{2}^{+},\ldots$ be the vertices of $+BP$, $\nu^{+}$.
We define a function $\phi^{+}:\mathbb{N}\rightarrow \mathbb{N}$ recursively by $\phi^{+}(0)=0$ and
\[\phi^{+}(n)=\inf\big\{j\geq \phi^{+}(n-1): \lim_{\epsilon\rightarrow 0+}\Phi_{\epsilon}(X_{j})=X_{n}^{+}    \big \}. \]
Then
\[ \big|\Phi_{\epsilon}(X_{\phi^{+}(n)})-X_{n}^{+} \big|\leq\epsilon\delta\]
and so
\[X_{n}^{+}=\lim_{\epsilon\rightarrow 0^{+}}X^{\epsilon}_{\phi^{+}(n)}.\]
We say that $X^{+}_{n}$ is an \emph{unblocked} vertex of the $+$BP if $X^{\epsilon}_{\phi^{+}(n)}$ is an unblocked vertex of the $\epsilon$-BP. This definition is independent of the choice of $\epsilon>0$ for sufficiently small $\epsilon$.
For each $n=0,1,2,\ldots$, we consider the vertex $X_{n}^{+}$ and take $V=X^{\epsilon}_{\phi^{+}(n)}$; we use the previous notation to define an event
\begin{eqnarray*}A_{n}^{+}&\triangleq&\{X_{n}^{+} \text{ is unblocked; if } X^{\epsilon}_{\phi^{+}(n)+1}=V', \text{ then } X^{\epsilon}_{\phi^{+}(n)+2}=Q_{V}\}\\&&\cup\{X_{n}^{+} \text{ is not unblocked}\}.\end{eqnarray*}
Similarly, we can define the events $A_{n}^{-}$.

We define the \emph{conditioned right percolation exploration path on the shifted brick wall lattice} (abbreviated as $+$CBP) to be the path whose transition probabilities at the $n$th step is the transition probability of the $+$BP at the $n$th step conditioned on $(A_{k}^{+})_{k=1}^{n}$. Similarly, we define the \emph{conditioned left percolation exploration path on the shifted brick wall lattice} (abbreviated as $-$CBP) to be the path whose law up to the $n$th step  is the law of a $-$BP up to the $n$th step conditioned on $(A_{k}^{-})_{k=1}^{n}$.

 We say that a vertex of the +CBP or $-$CBP, $\widetilde{X}_{j}$, is a \emph{free vertex} if the next vertex $\widetilde{X}_{j+1}$ has exactly two possible values (with positive probability) such that the edges $[\widetilde{X}_{j-1},\widetilde{X}_{j}]$ and $[\widetilde{X}_{j},\widetilde{X}_{j+1}]$ are perpendicular; otherwise, we say that it is a \emph{non-free} vertex.

\begin{lemma}\label{lem3}
Let $\widetilde{X}_{0},\widetilde{X}_{1},\ldots$ denote the vertices of a $+$CBP or a $-$CBP. Conditioned on the event that $\widetilde{X}_{j}$ is a free vertex, $\widetilde{X}_{j+1}$ has two possible values with probability $1/2$ and $\widetilde{X}_{j+1}-\widetilde{X}_{j}$ is independent of $\widetilde{X}_{0},\ldots,\widetilde{X}_{j}$.
\end{lemma}
\begin{proof}
Follows directly from Lemmas \ref{lem1} and \ref{lem2}, and the definition of the $+$CBP and $-$CBP.
\end{proof}

So by construction, at the free vertices of the +CBP or $-$CBP,  the path does not go in the same direction for two consecutive edges. We now restrict the $+$CBP further by restricting to paths that do not go in the same direction for two consecutive edges at all vertices.
More precisely, we condition the $+$CBP or $-$CBP  not to go straight at each non-free vertex. This gives us the curve which is ``almost" the SqP except for the fact that the topology on the shifted brick-wall lattice induced by the topology on the $(\pm\epsilon)$-brick-wall lattice is not the same as the standard topology on the shifted brick-wall lattice. This is explained in further detail in Section \ref{section8}.
We call this restricted path the \emph{boundary conditioned right percolation exploration path on the shifted brick wall lattice} (abbreviated $+\partial$CBP). Similarly, we define the \emph{boundary conditioned left percolation exploration path on the shifted brick wall lattice} (abbreviated $-\partial$CBP). Then at every vertex of the $+\partial$CBP or $-\partial$CBP, the next edge is perpendicular to the previous edge. We say that a vertex $V$ of a $+\partial$CBP or $-\partial$CBP is $\emph{free}$ if there are exactly two possible values for the next vertex; otherwise, we say that $V$ is \emph{non-free}.

Finally, we remark that the $-$BP, $-$CBP,  and $-\partial$CBP are identically distributed to the reflection across the $y$-axis of the $+$BP, $+$CBP,  and $+\partial$CBP. In particular the Loewner driving functions of $-$BP, $-$CBP,  and $-\partial$CBP is $-1$ times the driving functions of +BP, +CBP,  and $+\partial$CBP respectively.

\begin{figure}[hp]
\scalebox{0.3}{\includegraphics{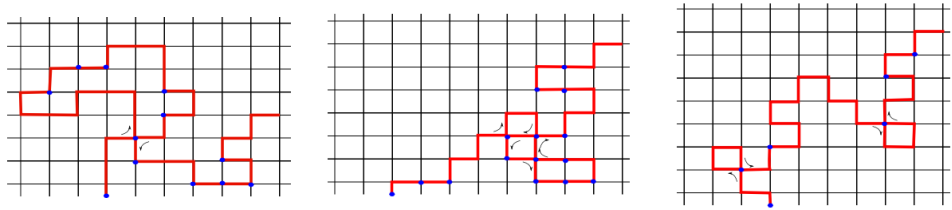}}
\caption{From left to right, a sample path of a $+$BP, $+$CBP and $+\partial$CBP. The blue vertices denote the blocked/non-free vertices and the arrows indicate the direction of the path.} \label{fig6}
\end{figure}

\section{The driving function of paths on lattices}
We suppose that $\mathbb{L}$ is a rectangular lattice (including the shifted brick-wall lattice and the square lattice of mesh size $\delta>0$). Suppose that $\mathbf{D}=(D,a,b)\in\mathcal{D}^{\mathbb{L}}$ for some $\delta>0$. Since the boundary of $D$ is on the lattice, we can apply the Schwarz-Christoffel formula \cite{Ne} to the function $\phi_{\mathbf{D}}$ to get that $\phi_{\mathbf{D}}$ satisfies
\begin{equation}\phi_{\mathbf{D}}'(z)^{2}=K\prod_{j=1}^{M}(z-r_{j,0})^{\rho_{j}}\label{zeq3}\end{equation}
for some $r_{j,0}\in \mathbb{R}$, $K\in\mathbb{C}$, $\rho_{j}\in\mathbb{R}$ and $M\in\mathbb{N}$ (see Figure \ref{nota2}). Note that when $(D,a,b)=(\mathbb{H},0,\infty)$, we have $M=0$.

 Let $\nu$ be a path on the $\mathbb{L}$. We assume for the while that $\nu$ is simple. At each point $Z_{k}$, the path changes the direction by $\pm \pi/2$ radians. Let $D_{n}=\phi_{\mathbf{D}}(H_{t_{n}})$. Since the boundary of $D_{n}$ is also on the lattice, we can apply the Schwarz-Christoffel formula to $\phi_{\mathbf{D}}\circ f_{t}$ to get
\begin{eqnarray}&&\phi_{\mathbf{D}}'(f_{t}(z))^{2}f_{t}'(z)^{2} \nonumber \\ &=&K\frac{(z-\xi_{t})^{2}}{(z-a_{1}(t))(z-b_{1}(t))}\Big(\prod_{k=2}^{N(t)}\big(\frac{z-b_{k}(t)}{z-a_{k}(t)}\big)^{L_{k}}\Big)
 \Big(\prod_{j=1}^{M}\big(z-r_{j}(t)\big)^{\rho_{j}}\Big)\label{zeq4},
 \end{eqnarray}
where $r_{j}(t)=f_{t}^{-1}(r_{j,0})$ and $L_{k}=-1,0, \text{ or } 1$ depending on whether the path at the $k$th step goes left, straight or right respectively. Note that $r_{j}(0)=r_{j,0}$.
We call $(L_{k})$ the \emph{turning sequence} of the path.

Combining (\ref{zeq3}) and (\ref{zeq4}), and eliminating the constant $K$, we get
\begin{eqnarray}
&&f_{t}'(z)^{2}\prod_{j=1}^{M}(f_{t}(z)-r_{j})^{\rho_{j}}\nonumber\\
&=&\frac{(z-\xi_{t})^{2}}{(z-a_{1}(t))(z-b_{1}(t))}\Big(\prod_{k=2}^{N(t)}\big(\frac{z-b_{k}(t)}{z-a_{k}(t)}\big)^{L_{k}}\Big)
 \big(\prod_{j=1}^{M}(z-r_{j}(t))^{\rho_{j}}\big)\label{zeq5}.\end{eqnarray}

\begin{figure}[hp]
\scalebox{0.55}{\includegraphics{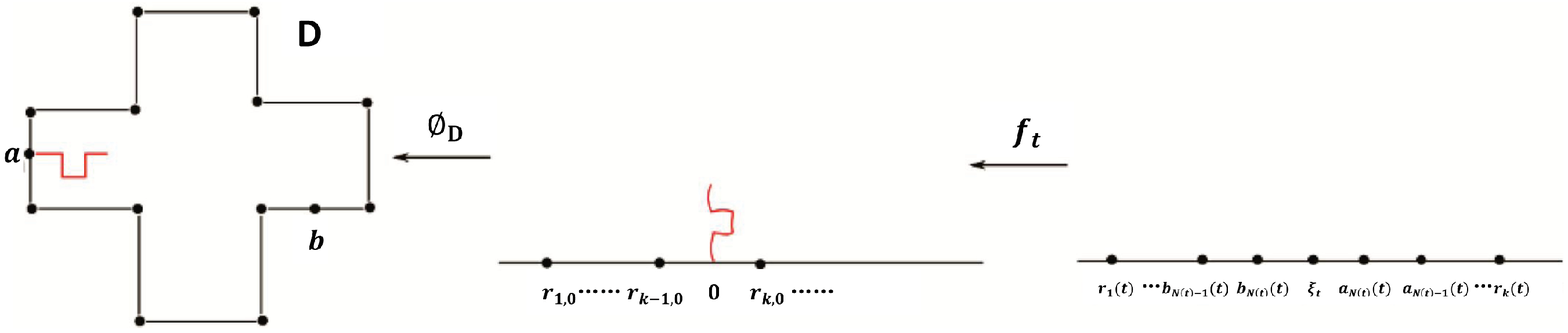}}
\caption{} \label{nota2}
\end{figure}
By taking the limit of rectilinear paths, we can extend (\ref{zeq5}) to non-crossing paths as well. In this case the points $a_{k}(t)$ and $b_{k}(t)$ for $k=1,2,\ldots$ may coincide with each other or with $\xi_{t}$ when the path makes loops.
\ \\

Now consider the expansion of both sides of (\ref{zeq5}) as $z\rightarrow\infty$. Using the fact that $f_{t}$ is hydrodynamically normalized, (see (\ref{zeq6})), we get
\begin{eqnarray*} \mathrm{LHS}&=&1+\frac{\sum_{j=1}^{M}\rho_{j}r_{j}}{z}+O\left(\frac{1}{z^2}\right)\text{
as } z\rightarrow \infty \\
\mathrm{RHS}&=&1+\frac{2\xi_{t}-a_{1}(t)-b_{1}(t)+\left(\sum_{k=2}^{N(t)}L_{j}(b_{k}(t)-a_{k}(t))\right)+\left(\sum_{j=1}^{M}\rho_{j}r_{j}(t)\right)}{z}\\
&&\quad \quad +O\big(\frac{1}{z^2}\big).
\end{eqnarray*}
Comparing the coefficients of the $1/z$ term, we deduce that
\begin{equation} \xi_{t}=\frac{1}{2}\left(a_{1}(t)+b_{1}(t)\right)+\frac{1}{2}\big(\sum_{j=1}^{M}\rho_{j}(r_{j}-r_{j}(t))\big)
+\frac{1}{2}\big(\sum_{k=2}^{N(t)}L_{k}(a_{k}(t)-b_{k}(t))\big).\label{weq6}\end{equation}
Then we have
\begin{equation}\xi_{t}=R_{t}(t)\label{zeq7}\end{equation}
where, for $s\leq t$, $R_{s}(t)$ is defined by
\begin{equation}R_{s}(t)=\frac{1}{2}\left(a_{1}(t)+b_{1}(t)\right)+\frac{1}{2}\big(\sum_{j=1}^{M}\rho_{j}(r_{j}-r_{j}(t))\big)
+\frac{1}{2}\big(\sum_{k=2}^{N(s)}L_{k}(a_{k}(t)-b_{k}(t))\big).\label{zeq11}\end{equation}

\section{Convergence of the +BP to SLE$_{6}$}
Let $\nu^{\delta} $ be the right percolation exploration path on the shifted brick-wall lattice with mesh-size $\delta$ (the +BP). Denote the vertices of $\nu^{\delta} $ by $Z^{\delta} _{0},Z^{\delta} _{1},Z^{\delta} _{2},\ldots$. Let $\gamma^{\delta} (t)$ be the curve $\nu^{\delta} $ parameterized  by the  half-plane capacity and let $\xi^{\delta} _{t}$ be the Loewner driving function of $\nu^{\delta} $. Then from Lemma \ref{lem2} and the results of Smirnov \cite{smirnov01a,smirnov01b}, Camia and Newman \cite{CN07}, we have the following theorem.

\begin{theorem}\label{th1}
For any $\mathbf{D}=(D,a,b)\in\mathcal{D}$ and $T>0$. The sequence of curves $(\gamma _{t}^{\delta})$ converges in distribution in the metric given in (\ref{metric}) to the trace of SLE$_{6}$ on $[0,T]$ as $\delta\searrow 0$.
\end{theorem}
\begin{corollary}\label{couple}For any fixed sufficiently small $\delta>0$, there exists a filtered probability space $(\Omega,\mathcal{F}_{t}^{\delta},\mu)$ on which $\gamma^{\delta}(t)$ and the trace of SLE$_{6}$, $\Gamma(t)$,  are defined and an increasing function $\varsigma_{\delta}(t)$  such that
\begin{enumerate}
\item $\Gamma(t)$ and $\gamma^{\delta}(\varsigma_{\delta}^{-1}(t))$ are adapted to $\mathcal{F}_{t}^{\delta}$.
\item The Loewner driving function of $\Gamma(t)$ is $\sqrt{6}$ times an $\mathcal{F}_{t}^{\delta}$-Brownian motion $B_{t}$.
\item Almost surely, $\gamma^{\delta}(t)$ lies in a C$\delta^{\alpha}$ neighbourhood of $\Gamma(t)$ for some universal constants $C,\alpha>0$.
\item For $t\in[0,T]$, we have
\[|\varsigma_{\delta}(t)-t|<C\delta^{\frac{\alpha}{2}}.\]
\end{enumerate}
\end{corollary}
\begin{proof}
For any sequence $\delta_{k}\searrow 0$ as $k\rightarrow \infty$, by Theorem \ref{th1}, the $+$BP on the lattice of mesh size $\delta_{k}$, $\nu^{\delta_{k}}$, converges to SLE$_{6}$ as $k\rightarrow\infty$. Hence, the Skorokhod-Dudley theorem \cite{RY} gives a coupling between $(\nu^{\delta_{k}})_{k \in \mathbb{N}}$ and SLE$_{6}$ such that $\nu^{\delta_{k}}$ converges almost surely to SLE$_{6}$ as $k\rightarrow\infty$. Then (3) simply follows from \cite{BCL} or \cite{MNW}.

To prove (1) and (2), we construct the above coupling in the following way: firstly, we consider a standard one dimensional Brownian motion $\mathbb{B}_{t}$ and associated natural filtration $\widehat{\mathcal{F}}_{t}$. Using the Donsker's invariance principle \cite{RY}, we can consider a simple random walk
\[Y_{n}=\sum_{k=1}^{n} \widehat{L}_{k}\]
(where $\widehat{L}_{k}$ are independent random variables taking the values $+1$ and $-1$ with probability 1/2 each) coupled to $\mathbb{B}_{t}$. Since the HexP (on the lattice of mesh size $\delta$) can be constructed from the sequence of random variables $(\widehat{L}_{k})$, we can define a +BP, $\nu^{*}$, on the filtration $\widehat{\mathcal{F}}_{t}$ using the coupling defined in Lemma \ref{lem2}. We can reparametrize by an increasing function  $ \sigma(t)$ such that
$Y_{n}$ is adapted to $\widehat{\mathcal{F}}_{\sigma(t_{n})}$. Note that $(Y_{n})$ completely determines $\nu^{*}$, and vice versa.

In accordance with the first paragraph above, we are clear that one can construct a SLE$_{6}$ on the filtration $\widehat{\mathcal{F}}_{\sigma(t)}$ with parametrization given by $\nu^{*}$. We now reparameterize by $t\mapsto \varsigma_{\delta}^{-1}(t)$ so that SLE$_{6}$ is parameterized by half-plane capacity; we also set $\mathcal{F}_{t}^{\delta}=\widehat{\mathcal{F}}_{\sigma(\varsigma_{\delta}^{-1}(t))}$. This completes the construction of the coupling and establishes (1) and (2).

Lemma 4.10 in \cite{LMR} states that if the curves are close, then so are their respective half-plane capacities, and this Lemma also gives a relationship between their closeness. Hence (4) follows from this Lemma, and (3).
\end{proof}
\begin{corollary}\label{cor1}
For any $T>0$ and $s,u,t\in(0,T]$ such that $s<u<t$. Let
$\alpha^{\delta} _{s}(t)$ and $\beta^{\delta} _{s}(t)$ be the two real-valued preimages of $\gamma^{\delta} (s)$ under $f^{\delta}_{t}$ with $\alpha^{\delta} _{s}(t)>\beta^{\delta} _{s}(t)$.
In accordance with Theorem \ref{th1} and Corollary \ref{couple}, we can define
\[V_{s}(t)\triangleq\lim_{\delta\searrow 0} \big(\alpha^{\delta} _{s}(t)-\beta^{\delta} _{s}(t)\big)\]
such that $V_{s}(t)-V_{s}(u)$ is independent of $\mathcal{F}^{\delta}_{u}$ and $V_{s}(s)=0$.
Moreover, we have
\[| \left(\alpha^{\delta} _{s}(t)-\beta^{\delta} _{s}(t)\right)- V_{s}(t)|\leq C\delta^{\frac{\alpha}{2}} \text{ almost surely},\]
for some constant $C>0$ which is independent of $s,t$ and $\delta$  (but could depend on $T$).
\end{corollary}
\begin{proof}
From Theorem \ref{th1} and Corollary \ref{couple}, $V_{s}(t)$ is well-defined, $V_{s}(t)- V_{s}(u)$ is independent of $\mathcal{F}^{\delta}_{u}$. Also, $V_{s}(s)=0$ since $\alpha^{\delta}_{s} (s)=\beta^{\delta}_{s} (s)$. Moreover, Corollary \ref{couple} implies that we can couple a SLE$_{6}$ trace, $\Gamma$, with $\nu^{\delta} $ such that $\nu^{\delta} $ is contained in a $C\delta^{\alpha}$-neighbourhood of $\Gamma$. Now, Lemma 4.8 in \cite{LMR} states that if the curves are close, then so are their respective preimages, and this Lemma also gives a relationship between their closeness.
In particular, using  (3) in Corollary \ref{couple}, this Lemma implies that
\[\left| \left(\alpha^{\delta}_{s} (t)-\beta^{\delta} _{s}(t)\right)- V_{s}(t)\right| \leq C\delta^{\frac{\alpha}{2}}\]
for some constant $C>0$. Uniform convergence implies that $C$ does not depend on $s,t$ or $\delta$.
\end{proof}

\section{Some useful estimates}
In this section, we fix a sufficiently small $\delta>0$, and for the sake of clarity, we here suppress the dependence of $\delta$ and $\mathbf{D}$ in all the notations in this section. 
We let $\mathcal{F}_{t}$ denote the filtration given in Corollary \ref{couple}. Using the notation in Section \ref{sect3}, we let
\[\mathfrak{A}^{+}_{n}\triangleq\big(A_{k}^{+}\big)_{k=0}^{n},\]
\[\mathfrak{A}^{+}_{\infty}\triangleq\big(A_{k}^{+}\big)_{k=0}^{\infty}.\]
Also, for a random variable $X$ and event $E$, we denote by $X|E$ the restriction of the random variable $X$ to the event $E$ whose distribution is the conditional distribution of $X$ given $E$ i.e.
\[F_{X|E}(x)\triangleq\mathbb{P}[X\leq x| E ].\]
We now need to use some concepts introduced in \cite{LLN}. Let $A$ be a compact $\mathbb{H}$-hull i.e. $A$ is compact and $\mathbb{H}\setminus A$ is simply-connected. We define
$\mathrm{hsiz}[A]$ to be the area of the union of all the discs tangent to $\mathbb{R}$ with centre at a point in $A$. Then there is the following relationship between $\mathrm{hsiz}$ and half-plane capacity (see Theorem 1 in \cite{LLN})
\begin{equation}\frac{1}{66}\mathrm{hsiz}[A]<\mathrm{hcap}[A]<\frac{7}{2\pi}\mathrm{hsiz}[A]\label{hsiz}\end{equation}

Now for some  $\eta>\delta$ (which we shall specify later), we define a sequence of stopping times as follows: $m_{0}=0$ and
\[m_{n}\triangleq\inf\left\{j\geq m_{n-1}: \mathrm{hsiz}[ f_{t_{m_{n-1}}}(\gamma[t_{m_{n-1}},t_{j}])]>\frac{\eta}{2}\right\}.\]
Then for $\eta$ sufficiently large compared with $\delta$, we have
\begin{equation}t_{m_{n}}-t_{m_{n-1}}<\frac{7}{2\pi}\eta\label{xsx1}.\end{equation}
The main idea here is the following: when the curve is winding a lot, the $\mathrm{hsiz}$ does not change since the winded parts would be contained in the union of discs in the definition of $\mathrm{hsiz}$. Hence the curve must `unwind'' at time $t_{m_{n}}$ for each $n$. This fact will be needed in the following lemma.

First we define, for $k>n$,
\[W_{k,n}\triangleq\sum_{j=n+1}^{k}L_{j}.\]
Note that $W_{k,n}$ is equal to a constant multiplied by the winding of the curve from the $(n+1)$-th step to the $k$th step. We need the following lemma regarding $W_{k,n}$.

\begin{lemma}\label{turn1}
For any $l>0$,
\begin{equation}\mathbb{P}[|W_{k,m_{n}}|\geq l|\mathcal{F}_{\varsigma_{\delta}(t_{m_{n}})},\mathfrak{A}_{k}^{+}]\leq C_{1}e^{-C_{2}l}.\label{brb1}\end{equation}
for constants $C_{1},C_{2}$ independent of $k$, $n$ and $l$.
In particular, $\mathbb{E}[\big|W_{k,m_{n}}\big|^{\beta}|\mathcal{F}_{\varsigma_{\delta}(s)},$ $\mathfrak{A}_{k}^{+}]$ is well-defined for any $s<t_{m_{n}}$ and $\beta>0$.
\end{lemma}
\begin{proof}
By the choice of $m_{n}$ and the definition of $\mathrm{hsiz}$, we must have $|W_{m_{n}-1,0}|<K$ for some constant $K>0$ depending only on the domain $D$. 

Since $W_{k,n}$ is the winding, this means that as $W_{k,n}$ gets large, the curve will spiral. This implies that the +CBP path passes through at least $W_{k,n}$ rectangles whose width is at least $\delta$ and whose length is uniformly bounded below. Hence, by independence in each disjoint rectangle, we have
\[\mathbb{P}[|W_{k,n}|\geq l|\mathcal{F}_{\varsigma_{\delta}(t_{n})},\mathfrak{A}_{k}^{+}]\leq C_{1}e^{-C_{2}l}\]
for some $C_{1},C_{2}>0$ not depending on $k$, $n$ or $l$.
\end{proof}
By Corollary \ref{cor1}, for any $n$ such that $t_{m_{n}}\leq T$ and $k$ such that $t_{k}\in[t_{m_{n-1}},t_{m_{n}}]$ we can find $V_{t_{k}}(t_{m_{n}})$ such that $V_{t_{k}}(t_{m_{n}})$ is independent of $\mathcal{F}_{\varsigma_{\delta}(t_{k})}$ and if we let
\begin{equation}U_{k}(t_{m_{n}})\triangleq (a_{k}(t_{m_{n}})-b_{k}(t_{m_{n}}))-V_{t_{k}}(t_{m_{n}}),\label{req1}\end{equation}
then $|U_{k}(t_{m_{n}})|\leq C\delta^{\alpha/2}$ almost surely for $C$ not depending on $n$, $k$ and $\delta$.

Using (\ref{weq6}), (\ref{zeq7}) and (\ref{zeq11}), we can write
\begin{eqnarray*}
&&\xi _{t_{m_{n}}}-\xi _{t_{m_{n-1}}}\\
&=&R_{t_{m_{n-1}}}(t_{m_{n}})-R_{t_{m_{n-1}}}(t_{m_{n-1}})+\frac{1}{2}\big[\sum_{k=m_{n-1}+1}^{m_{n}}L_{k}(a_{k}(t_{m_{n}})-b_{k}(t_{m_{n}}))\big].
\\&=&R_{t_{m_{n-1}}}(t_{m_{n}})-R_{t_{m_{n-1}}}(t_{m_{n-1}})+\frac{1}{2}\big[\sum_{k=m_{n-1}+1}^{m_{n}}L_{k}(V_{t_{k}}(t_{m_{n}})+U_{k}(t_{m_{n}}))\big].
\\&=&R_{t_{m_{n-1}}}(t_{m_{n}})-R_{t_{m_{n-1}}}(t_{m_{n-1}})+H_{n},
\end{eqnarray*}
where
\begin{equation}H_{n}\triangleq\frac{1}{2}\big[\sum_{k=m_{n-1}+1}^{m_{n}}L_{k}(a_{k}(t_{m_{n}})-b_{k}(t_{m_{n}}))\big]=\frac{1}{2}\big[\sum_{k=m_{n-1}+1}^{m_{n}}L_{k}(V_{t_{k}}(t_{m_{n}})+U_{k}(t_{m_{n}}))\big]\label{beq1}.\end{equation}

\begin{lemma}\label{mel3}
For any sufficiently small $\epsilon>0$,
\begin{equation}
\Big{|}\mathbb{E}\big[\sum_{k=m_{n-1}+1}^{m_{n}}L_{k}U_{k}(t_{m_{n}})|\mathcal{F}_{\varsigma_{\delta}(t_{m_{n-1}})},\mathfrak{A}_{m_{n}}^+\big]
\Big{|}\leq C\delta^{\frac{\alpha}{2}-\epsilon}\label{req3}\end{equation}
for some constant $C>0$ not depending on $n$ or $\delta$.
\end{lemma}
\begin{proof}
We write
\[\Upsilon_{k}(t_{m_{n}})\triangleq U_{k}(t_{m_{n}})-U_{k+1}(t_{m_{n}})\]
and hence
\[U_{j}(t_{m_{n}})=\sum_{k=j}^{m_{n}-1}\Upsilon_{k}(t_{m_{n}})\]
since $U_{m_{n}}(t_{m_{n}})=0$. By exchanging the order of summation, we get
\[\sum_{k=m_{n-1}+1}^{m_{n}}L_{k}U_{k}(t_{m_{n}})
=\sum_{k=m_{n-1}+1}^{m_{n}-1}\Upsilon_{k}(t_{m_{n}})W_{k,m_{n-1}}\]
Hence we get
 \begin{eqnarray*}
&&\mathbb{E}\big[\sum_{k=m_{n-1}+1}^{m_{n}}L_{k}U_{k}(t_{m_{n}})|\mathcal{F}_{\varsigma_{\delta}(t_{m_{n-1}})},\mathfrak{A}_{m_{n}}^{+}\big]\\
 &=& \mathbb{E}\big[\sum_{k=m_{n-1}+1}^{m_{n}-1}\Upsilon_{k}(t_{m_{n}})W_{k,m_{n-1}}|\mathcal{F}_{\varsigma_{\delta}(t_{m_{n-1}})},\mathfrak{A}^{+}_{m_{n}}\big]\\
 &=& \mathbb{E}\big[\sum_{k=m_{n-1}+1}^{m_{n}-1}\Upsilon_{k}(t_{m_{n}})\sum_{l=-\infty}^{\infty} l\mathbb{I}_{W_{k,m_{n-1}}=l}|\mathcal{F}_{\varsigma_{\delta}(t_{m_{n-1}})},\mathfrak{A}^{+}_{m_{n}}\big]
 \\&=&\mathbb{E}\Big[\sum_{k=m_{n-1}+1}^{m_{n}-1}\Upsilon_{k}(t_{m_{n}})\Big[\big(\sum_{l=-N}^{N} l\mathbb{I}_{W_{k,m_{n-1}}=l}\big)+\big(\sum_{|l|\geq N+1} l\mathbb{I}_{W_{k,m_{n-1}}=l}\big)\Big]
 |\mathcal{F}_{\varsigma_{\delta}(t_{m_{n-1}})},\mathfrak{A}^{+}_{m_{n}}\Big],
\end{eqnarray*}
for suitable choice of $N\in\mathbb{N}$. Thus
 \begin{eqnarray*}
&&\Big|\mathbb{E}\big[\sum_{k=m_{n-1}+1}^{m_{n}}L_{k}U_{k}(t_{m_{n}})|\mathcal{F}_{\varsigma_{\delta}(t_{m_{n-1}})},\mathfrak{A}_{m_{n}}^{+}\big]\Big|\\ &\leq&\mathbb{E}\Big[\big|\sum_{k=m_{n-1}+1}^{m_{n}-1}\Upsilon_{k}(t_{m_{n}})\big|\Big[\big|\sum_{l=-N}^{N} l\mathbb{I}_{W_{k,m_{n-1}}=l}\big|\\
&&\qquad +\big|\sum_{|l|\geq N+1} l\mathbb{I}_{W_{k,m_{n-1}}=l}\big|\Big]
 |\mathcal{F}_{\varsigma_{\delta}(t_{m_{n-1}})},\mathfrak{A}^{+}_{m_{n}}\Big].
\end{eqnarray*}
Also, by Corollary \ref{cor1}, we have
\[\big|\sum_{k=m_{n-1}+1}^{m_{n}-1}\Upsilon_{k}(t_{m_{n}})\big|=\big|U_{t_{m_{n-1}}}(t_{m_{n}})\big|\leq C_{3}\delta^{\frac{\alpha}{2}}\]
for some constant $C_{3}>0$ not depending on $n$ or $\delta$.
Hence, if we let $N$ be the smallest integer greater than $\delta^{-\epsilon}$ for some sufficiently small $\epsilon>0$, then using Lemma \ref{turn1}, we have
\[\big|\mathbb{E}\big[\sum_{k=m_{n-1}+1}^{m_{n}}L_{k}U_{k}(t_{m_{n}})|\mathcal{F}_{\varsigma_{\delta}(t_{m_{n-1}})},\mathfrak{A}^{+}_{m_{n}}\big]\big|\leq C_{4}( \delta^{\frac{\alpha}{2}-\epsilon}+ \delta^{\frac{\alpha}{2}}e^{-C_{2}\delta^{-\epsilon}})
\leq C \delta^{\frac{\alpha}{2}-\epsilon}\]
for some constants $C_{4}, C >0$ not depending on $n$ and $\delta$.
 \end{proof}
 \ \\
We now consider the other part of $H_{n}$,
\[\sum_{k=m_{n-1}+1}^{m_{n}}L_{k}V_{t_{k}}(t_{m_{n}}).\]
We can decompose $V_{t_{k}}(t)$ as
\begin{eqnarray*} V_{t_{k}}(t)&=&V_{\varsigma_{\delta}(t_{k})}(\varsigma_{\delta}(t))+(V_{t_{k}}(t)-V_{\varsigma_{\delta}(t_{k})}(\varsigma_{\delta}(t)))
\\&=&V^{(1)}_{t_{k}}(t)+ V^{(2)}_{t_{k}}(t),
\end{eqnarray*}
where
\begin{equation}V^{(1)}_{t_{k}}(t)\triangleq V_{\varsigma_{\delta}(t_{k})}(\varsigma_{\delta}(t))\label{yeq1}\end{equation}
and
\[V^{(2)}_{t_{k}}(t)\triangleq V_{t_k}(t)-V^{(1)}_{t_{k}}(t).\]
Noting that by (4) in Corollary \ref{couple}, we have
\[ |V^{(2)}_{t_{m_{n-1}}}(t_{m_{n}})|\leq C\delta^{\frac{\alpha}{2}}\]
for some constant $C>0$ not depending on $n$ or $\delta$ since $V_{t_{k}}(t)$ is a continuous increasing function and hence is Lipschitz with constant not depending on $\delta$; moreover, this Lipschitz constant is bounded since $V_{t_{k}}(t)$ is the difference of two coupled Bessel processes.
Using this estimate, and applying exactly the same method of proof as in Lemma \ref{mel3}, we can prove the following result.
\begin{lemma}\label{mel5}
For any sufficiently small $\epsilon>0$,
\begin{equation}
\Big{|}\mathbb{E}\big[\sum_{k=m_{n-1}+1}^{m_{n}}L_{k}V_{t_{k}}^{(2)}(t_{m_{n}})|\mathcal{F}_{\varsigma_{\delta}(t_{m_{n-1}})},\mathfrak{A}_{m_{n}}^+\big]
\Big{|}\leq C\delta^{\frac{\alpha}{2}-\epsilon}\label{req4}\end{equation}
for some constant $C>0$ not depending on $n$ or $\delta$.
\end{lemma}
We now consider
\[\sum_{k=m_{n-1}+1}^{m_{n}}L_{k}V^{(1)}_{t_{k}}(t_{m_{n}}).\]

Firstly, let $\mathcal{J}_{n}\triangleq\{m_{n-1}+1\leq k\leq m_{n}: \mathrm{dist}[Z_{k},\partial( D\setminus\nu[0,m_{n-1}])]\geq 4\delta\}$ where $\nu$ is the +BP. Then we define a sequence of stopping times by $S_{0}\triangleq m_{n-1}+1$, and for 
\[S_{2j-1}\triangleq\inf \{m_{n}\geq k\geq S_{2j-2}: k\not\in \mathcal{J}_n\},\]
\[S_{2j}\triangleq\inf \{ m_{n} \geq k\geq S_{2j-1}:k\in \mathcal{J}_n \}.\]
Also let $N_{1}\triangleq\sup\{j:2j+1 \leq N^{*}\}$ and $N_{2}\triangleq\sup\{j:2j+2 \leq N^{*}\}$, where $N^{*}\triangleq\inf\{j:S_{j}=m_{n}\}$.

Then, by construction for $k=S_{2j},\ldots,S_{2j+1}-1$, $Z_{k}$ is not contained in a $4\delta$  neighbourhood of $\partial( D\setminus\nu[0,m_{n-1}])$  and for
$k=S_{2j+1},\ldots,S_{2j+2}-1$, $Z_{k}$ is contained in a $4\delta$  neighbourhood of $\partial( D\setminus\nu[0,m_{n-1}])$. Let $\ell_{k}\triangleq\inf\{l:S_{l}\geq k+1\}$.

Then by telescoping the sum, we can write
\begin{eqnarray*}&&\sum_{k=m_{n-1}+1}^{m_{n}}L_{k}V^{(1)}_{t_{k}}(t_{m_{n}})
\\&=& \sum_{k=m_{n-1}+1}^{m_{n}}L_{k}\Big(V^{(1)}_{t_{k}}(t_{S_{\ell_{k}}})+ \sum_{l=\ell_{k}+1}^{N^{*}}[V^{(1)}_{t_{k}}(t_{S_{l}})-V^{(1)}_{t_{k}}(t_{S_{l-1}})]\Big)
\\&=& \Big[\sum_{k=m_{n-1}+1}^{S_{1}-1} L_{k}V^{(1)}_{t_{k}}(t_{S_{1}})\Big]
\\&&+ \Big[\sum_{l=2}^{N^{*}} \sum_{k=S_{l-1}}^{S_{l}-1} L_{k}V^{(1)}_{t_{k}}(t_{S_{l}})  + \sum_{l=1}^{N^{*}} \sum_{k=m_{n-1}+1}^{S_{l}-1} L_{k}\big(V^{(1)}_{t_{k}}(t_{S_{l}})-V^{(1)}_{t_{k}}(t_{S_{l-1}})\big)\Big]
\\
&=& \Big(\sum_{l=1}^{N^{*}} J_{l}\Big) +  \Big(\sum_{l=1}^{N^*}\sum_{k=m_{n-1}+1}^{S_{l}-1} L_{k}\big(V^{(1)}_{t_{k}}(t_{S_{l}})-V^{(1)}_{t_{k}}(t_{S_{l-1}})\big)\Big),
\end{eqnarray*}
where  for $l=1,2,3,\ldots,N^{*}$,
\[ J_{l}\triangleq\sum_{k=S_{l-1}}^{S_{l}-1} L_{k}V^{(1)}_{t_{k}}(t_{S_{l}}). \]
Let
\begin{equation}K_{n}\triangleq\sum_{l=1}^{N^*}\sum_{k=m_{n-1}+1}^{S_{l}-1} L_{k}\big(V^{(1)}_{t_{k}}(t_{S_{l}})-V^{(1)}_{t_{k}}(t_{S_{l-1}})\big).\label{beq2}\end{equation}
\begin{lemma}\label{mel6}
For any sufficiently small $\epsilon>0$,
\begin{equation}\Big|\mathbb{E}\Big[\sum_{l=1}^{N^{*}}J_{l}|\mathcal{F}_{\varsigma_{\delta}(t_{m_{n-1}})},\mathfrak{A}^{+}_{m_{n}}\Big]\Big|\leq C\delta^{\frac{1}{2}-\epsilon}\label{req17}\end{equation}
for some constant $C>0$ not depending on $n$ or $\delta$.
\end{lemma}
\begin{proof}
Note that for $k=S_{2j+1},\ldots, S_{2j+2}-1$,  $Z_{k}$ are contained in a $4\delta$ neighbourhood of $\partial( D\setminus\nu[0,m_{n-1}])$. Hence by Lemma 4.8 in \cite{LMR}, we have
\[|V_{t_{k}}^{(1)}(t_{S_{l}})| \leq C_{2}\delta^{\frac{1}{2}}\]
for some constant $C_{2}>0$ not depending on $n$.
Using this estimate, and applying exactly the same method of proof as in Lemma \ref{mel3}, we get
\begin{equation}
\big{|}\mathbb{E}\big[\sum_{j=0}^{N_{2}}J_{2j+1}|\mathcal{F}_{\varsigma_{\delta}(t_{m_{n-1}})},\mathfrak{A}_{m_{n}}^+\big]
\big{|}\leq C_{3}\delta^{\frac{1}{2}-\epsilon} \label{resul1}\end{equation}
for some constant $C_{3}>0$ not depending on $n$ and for any sufficiently small $\epsilon>0$.

\ \\
We now consider,
\[ \sum_{j=0}^{N_{1}}J_{2j+2}.\]
Suppose that $l=2j+2$ and define
\[\Psi_{k}(t)\triangleq V^{(1)}_{t_{k}}(t)-V^{(1)}_{t_{k+1}}(t)\]
and hence, since $V^{(1)}_{t_{S_{l}}}(t_{S_{l}})=0$,
\begin{equation}V^{(1)}_{t_{j}}(t_{S_{l}})=\sum_{k=j}^{S_{l}-1}\Psi_{k}(t_{S_{l}}).\label{mbm5}\end{equation}
From the Loewner differential equation, we can see that $\Psi_{k}(t)$ is decreasing.

By exchanging the order of summation in (\ref{mbm5}), we get
\begin{equation}\sum_{k=S_{l-1}}^{S_{l}-1}L_{k}V^{(1)}_{t_{k}}(t_{S_{l}})=\sum_{k=S_{l-1}}^{S_{l}-1}\Psi_{k}(t_{S_{l}})W_{k,S_{l-1}-1}.
\label{mbm1}\end{equation}
Hence taking conditional expectations in (\ref{mbm1}) and using the tower property, we get
\begin{eqnarray*}
&&\mathbb{E}\big[\sum_{k=S_{l-1}}^{S_{l}-1}L_{k}V^{(1)}_{t_{k}}(t_{S_{l}}) |\mathcal{F}_{\varsigma_{\delta}(t_{S_{l-1}})}, \mathfrak{A}^{+}_{S_{l}}  \big]
\\
&=& \mathbb{E}\big[ \sum_{k=S_{l-1}}^{S_{l}-1}  W_{k,S_{l-1}-1}\big( \Psi_{k}(t_{k+1})+\sum_{j=k+1}^{S_{l}-1}[\Psi_{k}(t_{j+1})-\Psi_{k}(t_{j})]\big) |\mathcal{F}_{\varsigma_{\delta}(t_{S_{l-1}})}, \mathfrak{A}^{+}_{S_{l}}  \big]
\\&=& \mathbb{E}\big[\sum_{k=S_{l-1}}^{S_{l}-1}   W_{k,S_{l-1}-1}\big(\mathbb{E}\big[\Psi_{k}(t_{k+1})|\mathcal{F}_{\varsigma_{\delta}(t_{k})},A^{+}_{k}\big]
\\&&+\sum_{j=k+1}^{S_{l}-1}\mathbb{E}\big[\Psi_{k}(t_{j+1})-\Psi_{k}(t_{j}) |\mathcal{F}_{\varsigma_{\delta}(t_{j})},A^{+}_{j}\big]\big)  |\mathcal{F}_{\varsigma_{\delta}(t_{S_{l-1}})}, \mathfrak{A}^{+}_{S_{l}}  \big].
\end{eqnarray*}

We can decompose
\begin{eqnarray}&&\Psi_{k}(t_{j+1})-\Psi_{k}(t_{j})\nonumber \\&=& \big(\Psi_{k}(t_{j+1})-\Psi_{k}(t_{j})\big)\mathbb{I}_{\{Z_{j} \text{ is free}\}}+\big(\Psi_{k}(t_{j+1})-\Psi_{k}(t_{j})\big)\mathbb{I}_{\{Z_{j} \text{ is not free}\}}.\label{mbm3}\end{eqnarray}
Since $\Psi_{k}(t_{j+1})-\Psi_{k}(t_{j})$ and $A^{+}_{j+1}$ given that $Z_{j}$ is free are independent of $\mathcal{F}_{\varsigma_{\delta}(t_{j})}$ (by Lemma \ref{lem3} and Corollaries \ref{couple} and \ref{cor1}) we get
\begin{eqnarray*}&&\mathbb{E}\big[\big(\Psi_{k}(t_{j+1})-\Psi_{k}(t_{j})\big)\mathbb{I}_{\{Z_{j} \text{ is free}\}}|\mathcal{F}_{\varsigma_{\delta}(t_{j})},A^{+}_{j}\big]\\&=&\mathbb{I}_{\{Z_{j} \text{ is free}\}}\mathbb{E}\big[\Psi_{k}(t_{j+1})-\Psi_{k}(t_{j})|\mathcal{F}_{\varsigma_{\delta}(t_{j})},A^{+}_{j}\big]
\\&=&\mathbb{I}_{\{Z_{j} \text{ is free}\}}\mathbb{E}\big[\Psi_{k}(t_{j+1})-\Psi_{k}(t_{j})|A^{+}_{j}\big]
\end{eqnarray*}
Given that $Z_{j}$ is non-free, $A^{+}_{j}$ is \emph{not} independent of $\mathcal{F}_{\varsigma_{\delta}(t_{j})}$. However, we have
\begin{eqnarray*}&&\mathbb{E}\big[\big(\Psi_{k}(t_{j+1})-\Psi_{k}(t_{j})\big)\mathbb{I}_{\{Z_{j} \text{ is not free}\}}|\mathcal{F}_{\varsigma_{\delta}(t_{j})},A^{+}_{j}\big]\\
&=&\mathbb{I}_{\{Z_{j} \text{ is not free}\}}\mathbb{E}\big[\Psi_{k}(t_{j+1})-\Psi_{k}(t_{j})|\mathcal{F}_{\varsigma_{\delta}(t_{j})},A_{j}^{+}\big]
\\
&=&\mathbb{I}_{\{Z_{j} \text{ is not free}\}}\mathbb{E}\big[\Psi_{k}(t_{j+1})-\Psi_{k}(t_{j})|\mathcal{F}_{\varsigma_{\delta}(t_{j})}\big]
\\
&=&\mathbb{I}_{\{Z_{j} \text{ is not free}\}}\mathbb{E}\big[\Psi_{k}(t_{j+1})-\Psi_{k}(t_{j})\big]
\\&=&\mathbb{I}_{\{Z_{j} \text{ is not free}\}}\mathbb{E}\big[\Psi_{k}(t_{j+1})-\Psi_{k}(t_{j})|A_{j}^{+}\big].
\end{eqnarray*}
where the fourth equality follows from the fact that, by definition of $A^{+}_{j}$, when $Z_{j}$ is non-free, $A_{j}^{+}$ contains no additional information and hence $(\Psi_{k}(t_{j+1})-\Psi_{k}(t_{j}))|A^{+}_{j}$ is identically distributed to $\Psi_{k}(t_{j+1})-\Psi_{k}(t_{j})$; similarly the second equality follows from the fact that given $\mathcal{F}_{\varsigma_{\delta}(t_{j})}$, $(\Psi_{k}(t_{j+1})-\Psi_{k}(t_{j}))|A^{+}_{j}$ is identically distributed to $\Psi_{k}(t_{j+1})-\Psi_{k}(t_{j})$; and the third equality follows from Corollary \ref{cor1}.

Combining these relations with (\ref{mbm3}), we get
\[ \mathbb{E}\big[\Psi_{k}(t_{j+1})-\Psi_{k}(t_{j})|\mathcal{F}_{\varsigma_{\delta}(t_{j})},A^{+}_{j}\big]= \mathbb{E}\big[\Psi_{k}(t_{j+1})-\Psi_{k}(t_{j})|A^{+}_{j}\big]. \]
Similarly, we can show that
 $\mathbb{E}\big[\Psi_{k}(t_{k+1})|\mathcal{F}_{\varsigma_{\delta}(t_{k})},A^{+}_{k}\big]= \mathbb{E}\big[\Psi_{k}(t_{k+1})|A^{+}_{k}\big]. $
Hence
\begin{eqnarray*}&&\mathbb{E}\big[\sum_{k=S_{l-1}}^{S_{l}-1}L_{k}V^{(1)}_{t_{k}}(t_{S_{l}}) |\mathcal{F}_{\varsigma_{\delta}(t_{S_{l-1}})}, \mathfrak{A}^{+}_{S_{l}}  \big]
\\&=&\mathbb{E}\big[\sum_{k=S_{l-1}}^{S_{l}-1}   W_{k,S_{l-1}-1}\big(\mathbb{E}\big[\Psi_{k}(t_{k+1})|\mathcal{F}_{\varsigma_{\delta}(t_{k})},A^{+}_{k}\big]
\\&&+\sum_{j=k+1}^{S_{l}-1}\mathbb{E}\big[\Psi_{k}(t_{j+1})-\Psi_{k}(t_{j}) |\mathcal{F}_{\varsigma_{\delta}(t_{j})},A^{+}_{j}\big]\big)  |\mathcal{F}_{\varsigma_{\delta}(t_{S_{l-1}})}, \mathfrak{A}^{+}_{S_{l}}  \big]
\\&=& \mathbb{E}\big[\sum_{k=S_{l-1}}^{S_{l}-
1}   W_{k,S_{l-1}-1} \big(\mathbb{E}\big[\Psi_{k}(t_{k+1})|A^{+}_{k}\big]
\\&&+\sum_{j=k+1}^{S_{l}-1}\mathbb{E}\big[\Psi_{k}(t_{j+1})-\Psi_{k}(t_{j}) |A_{j}^{+}\big]\big) |\mathcal{F}_{\varsigma_{\delta}(t_{S_{l-1}})}, \mathfrak{A}^{+}_{S_{l}}  \big]
\\&=& \mathbb{E}\big[\sum_{k=S_{l-1}}^{S_{l}-
1}   W_{k,S_{l-1}-1} \big(\mathbb{E}\big[\Psi_{k}(t_{k+1})|\mathfrak{A}^{+}_{S_{l}} \big]
\\&&+\sum_{j=k+1}^{S_{l}-1}\mathbb{E}\big[\Psi_{k}(t_{j+1})-\Psi_{k}(t_{j}) |\mathfrak{A}^{+}_{S_{l}} \big]\big) |\mathcal{F}_{\varsigma_{\delta}(t_{S_{l-1}})}, \mathfrak{A}^{+}_{S_{l}}  \big]
\\&=& \mathbb{E}\big[\sum_{k=S_{l-1}}^{S_{l}-1}   W_{k,S_{l-1}-1} \mathbb{E}\big[\Psi_{k}(t_{S_{l}}) |\mathfrak{A}_{S_{l}}^{+}\big] |\mathcal{F}_{\varsigma_{\delta}(t_{S_{l-1}})}, \mathfrak{A}^{+}_{S_{l}}  \big]
\\&=&\sum_{k=S_{l-1}+1}^{S_{l}-1}   \mathbb{E}\big[\Psi_{k}(t_{S_{l}}) |\mathfrak{A}_{S_{l}}^{+}\big]\mathbb{E}\big[ W_{k,S_{l-1}-1} |\mathcal{F}_{\varsigma_{\delta}(t_{S_{l-1}})}, \mathfrak{A}^{+}_{S_{l}} \big]
\end{eqnarray*}
where the third equality follows from the definition of $\mathfrak{A}^{+}_{S_{l}}$.

Then note that for $l$ even, between $S_{l-1}$ and $S_l-1$, the path does not lie within a $4\delta$ neighbourhood of the boundary. The locality property of the $+$BP (which is inherited from the HexP) implies that the behaviour of $Z_{j}$ in a neighbourhood around $Z_{j}$ does not depend on the boundary. Hence, we must have by symmetry that
\[\mathbb{E}\big[ W_{k,S_{l-1}-1} |\mathcal{F}_{\varsigma_{\delta}(t_{S_{l-1}})} \big]=0.\]
We can write
\[ W_{k,S_{l-1}-1}=\sum_{j=S_{l-1}}^{k}\big(L_{j}\mathbb{I}_{\{Z_{j} \text{ is free}\}}+L_{j}\mathbb{I}_{\{Z_{j} \text{ is non-free}\}}\big)\]
By symmetry at the free vertices,
\[\mathbb{E}\big[ \sum_{j=S_{l-1}}^{k}L_{j}\mathbb{I}_{\{Z_{j} \text{ is free}\}} |\mathcal{F}_{\varsigma_{\delta}(t_{S_{l-1}})} \big]=0.\]
Hence we have
\[\mathbb{E}\big[\sum _{j=S_{l-1}}^{k}L_{j}\mathbb{I}_{\{Z_{j} \text{ is non-free}\}} |\mathcal{F}_{\varsigma_{\delta}(t_{S_{l-1}})} \big]=0.\]
Since at the non-free vertices, $A_{j}^{+}$ contains no relevant information, this implies that
\[\mathbb{E}\big[ \sum_{j=S_{l-1}}^{k}L_{j}\mathbb{I}_{\{Z_{j} \text{ is non-free}\}} |\mathcal{F}_{\varsigma_{\delta}(t_{S_{l-1}})},\mathfrak{A}^{+}_{S_{l}} \big]=\mathbb{E}\big[ \sum_{j=S_{l-1}}^{k}L_{j}\mathbb{I}_{\{Z_{j} \text{ is non-free}\}} |\mathcal{F}_{\varsigma_{\delta}(t_{S_{l-1}})} \big]=0.\]
Also, by Lemma \ref{lem3},
\[\mathbb{E}\big[ \sum_{j=S_{l-1}}^{k}L_{j}\mathbb{I}_{\{Z_{j} \text{ is free}\}} |\mathcal{F}_{\varsigma_{\delta}(t_{S_{l-1}})} ,\mathfrak{A}^{+}_{S_{l}}\big]=0.\]
Hence, again, we have
\[\mathbb{E}\big[ W_{k,S_{l-1}-1} |\mathcal{F}_{\varsigma_{\delta}(t_{S_{l-1}})}, \mathfrak{A}^{+}_{S_{l}}, \big]=0.\]
This implies that
\[\mathbb{E}\big[\sum_{k=S_{l-1}}^{S_{l}-1}L_{k}V^{(1)}_{t_{k}}(t_{S_{l}}) |\mathcal{F}_{\varsigma_{\delta}(t_{S_{l-1}})}, \mathfrak{A}^{+}_{S_{l}}  \big]=0
.\]
\end{proof}
\section{Driving term convergence of the $+$CBP and $+\partial$CBP}
We first define an \emph{$\epsilon$-semimartingale} to be the sum of a local martingale and a finite $(1+\epsilon)$-variation process for every $0<\epsilon<1$. Note that a continuous martingale with finite $(1+\epsilon)$-variation (for $0<\epsilon<1$) is necessarily constant (using the same proof for finite variation see e.g. \cite{RY}). This implies that the decomposition of $\epsilon$-martingales is unique.

 In this section, we will show that the driving term of the +CBP, $\xi^{\delta}_{t}|\mathfrak{A}_{\infty}^{+}$,  converges subsequentially to an $\epsilon$-semimartingale. From this we deduce the same for the $+\partial$CBP. We can write
\begin{equation}\xi^{\delta} _{t_{m_{n}}}=  \left(\mathcal{H}^{\delta} _{t_{m_{n}}}+\mathcal{R}^{\delta}_{t_{m_{n}}} \right),\label{ppeq2}
\end{equation}
where using (\ref{zeq11}), (\ref{beq1}) and (\ref{beq2}), we define
\[\mathcal{H}^{\delta}_{t_{m_n}} \triangleq\left[\sum_{j=1}^{n}(H_{j}-K_{j})\right]\]
and
\[\mathcal{R}^{\delta}_{t_{m_{n}}} \triangleq\sum_{k=1}^{n}\big(K_{k}+ R_{t_{m_{k-1}}}(t_{m_{k}})-R_{t_{m_{k-1}}}(t_{m_{k-1}})\big),\]
for $n=1,2,\ldots$ and for $t\in(t_{m_{n-1}},t_{m_{n}})$, $\mathcal{H}^{\delta}_{t}$ and $\mathcal{R}^{\delta}_{t}$ are obtained by linear interpolation.

\begin{proposition}\label{prop13}
Suppose that $T>0$, and $N$ satisfies $t_{m_{N-1}}<T\leq t_{m_{N}}$. Then for any sequence $(\delta_{k})$ with $\delta_{k}\searrow 0$ as $k\rightarrow \infty$,
 $\mathcal{R}^{\delta_{k}}_{t}|\mathfrak{A}^{+}_{\infty}$ has a subsequence which converges uniformly in distribution to a finite $(1+\epsilon)$-variation process (for every $0<\epsilon<1$) on $[0,T]$.
\end{proposition}
\begin{proof}
We consider
\[\mathcal{R}^{\delta}_{t_{m_{k}}}-\mathcal{R}^{\delta}_{t_{m_{k-1}}}=K_{k}+ R_{t_{m_{k-1}}}(t_{m_{k}})-R_{t_{m_{k-1}}}(t_{m_{k-1}})\]
Note that from (\ref{zeq11})
\begin{eqnarray*}&&R_{t_{m_{k-1}}}(t_{m_{k}})-R_{t_{m_{k-1}}}(t_{m_{k-1}})=\frac{1}{2}(a_{1}(t_{m_{k}})-a_{1}(t_{m_{k-1}}) )\\&& \qquad
+\frac{1}{2}(b_{1}(t_{m_{k}})-b_{1}(t_{m_{k-1}}) )
+\frac{1}{2}\big(\sum_{j=1}^{M}\rho_{j}(r_{j}(t_{m_{k-1}})-r_{j}(t_{m_{k}}))\big)
\\&&\qquad \qquad +\frac{1}{2}\big(\sum_{j=2}^{m_{k-1}}L_{j}( (a_{j}(t_{m_{k}})-a_{j}(t_{m_{k-1}}) ) - (b_{j}(t_{m_{k}})-b_{j}(t_{m_{k-1}}) ))\big).
\end{eqnarray*}
Let
\begin{eqnarray*}\Phi_{j,k}&=& (a_{j}(t_{m_{k}})-a_{j}(t_{m_{k-1}}) ) - (b_{j}(t_{m_{k}})-b_{j}(t_{m_{k-1}}) )\\&&- (a_{j+1}(t_{m_{k}})-a_{j+1}(t_{m_{k-1}}) ) - (b_{j+1}(t_{m_{k}})-b_{j+1}(t_{m_{k-1}}) ).\end{eqnarray*}
Then from the Loewner differential equation, $\Phi_{j,k}>0$. Then,
\begin{eqnarray*} &&(a_{j}(t_{m_{k}})-a_{j}(t_{m_{k-1}}))  - (b_{j}(t_{m_{k}})-b_{j}(t_{m_{k-1}}) )\\
 &=& (a_{m_{k-1}+1}(t_{m_{k}})-b_{m_{k-1}+1}(t_{m_{k}}))+ \sum_{l=j}^{m_{k-1}}\Phi_{l,k}.\end{eqnarray*}
Hence
\begin{eqnarray*}
&&\Big|\sum_{j=2}^{N(s)}L_{j}\big[ (a_{j}(t_{m_{k}})-a_{j}(t_{m_{k-1}}) ) - (b_{j}(t_{m_{k}})-b_{j}(t_{m_{k-1}}) )\big]\Big|
\\&\leq& \Big|\sum_{j=2}^{m_{k-1}} L_{j}\sum_{l=j}^{m_{k-1}}\Phi_{l,k} \Big| +\Big|(a_{m_{k-1}+1}(t_{m_{k}})-b_{m_{k-1}+1}(t_{m_{k}}))\sum_{j=2}^{m_{k-1}} L_{j} \Big|
\\&=& \Big|\sum_{l=2}^{m_{k-1}} \Phi_{l,k}\sum_{j=2}^{l}L_{j}\Big|+\Big|(a_{m_{k-1}+1}(t_{m_{k}})-b_{m_{k-1}+1}(t_{m_{k}}))\sum_{j=2}^{m_{k-1}} L_{j} \Big|
\\&\leq& \mathcal{W}^{\delta}\big((a_{2}(t_{m_{k}})-b_{2}(t_{m_{k}}) )-(a_{2}(t_{m_{k-1}})-b_{2}(t_{m_{k-1}}) ) \\
&&\qquad \qquad +2(a_{m_{k-1}+1}(t_{m_{k}})-b_{m_{k-1}+1}(t_{m_{k}})) \big),
\end{eqnarray*}
where
\[\mathcal{W}^{\delta}=\max_{l=1,\ldots, m_{N}} \Big|\sum_{j=2}^{l}L_{j}\Big|.\]
Similarly, we can show that
\[
|K_{k}|\leq \mathcal{W}^{\delta}V^{(1)}_{t_{m_{k-1}}}(t_{m_{k}})
\]
Hence,
\begin{eqnarray*}&&|\mathcal{R}^{\delta}_{t_{m_{k}}}-\mathcal{R}^{\delta}_{t_{m_{k-1}}}|\\
&\leq& \mathcal{W}^{\delta}\Big[V^{(1)}_{t_{m_{k-1}}}(t_{m_{k}})+(a_{2}(t_{m_{k}})-b_{2}(t_{m_{k}}) )-(a_{2}(t_{m_{k-1}})-b_{2}(t_{m_{k-1}})  \\
&&+2(a_{m_{k-1}+1}(t_{m_{k}})-b_{m_{k-1}+1}(t_{m_{k}}))\Big]
+\frac{1}{2}(a_{1}(t_{m_{k}})-a_{1}(t_{m_{k-1}}) )\\&&
+\frac{1}{2}(b_{1}(t_{m_{k}})-b_{1}(t_{m_{k-1}}) )
+\frac{1}{2}\big(\sum_{j=1}^{M}\rho_{j}(r_{j}(t_{m_{k-1}})-r_{j}(t_{m_{k}}))\big).
\end{eqnarray*}
Note that
\[\sum_{k=1}^{N} V^{(1)}_{t_{m_{k-1}}}(t_{m_{k}})\]
and
\[\sum_{k=1}^{N} (a_{m_{k-1}+1}(t_{m_{k}})-b_{m_{k-1}+1}(t_{m_{k}}))\]
are uniformly bounded above, for all $\delta$, by a constant multiple of $T$.  Hence, we can write
\begin{equation}\big(\mathcal{R}^{\delta}_{t}|\mathfrak{A}^{+}_{\infty}\big)-\big(\mathcal{R}^{\delta}_{s}|\mathfrak{A}^{+}_{\infty}\big)=(\mathcal{W}^{\delta}|\mathfrak{A}^{+}_{\infty})(X^{\delta}_{t}-X^{\delta}_{s})+(Y^{\delta}_{t}-Y^{\delta}_{s})\label{ppeq1}\end{equation}
for continuous finite variation processes $X^{\delta}_{t}$ and $Y^{\delta}_{t}$ which are the sum of finitely many (uniformly for all $\delta$) monotonic increasing/decreasing functions. Note that by Proposition 3.76 in \cite{lawlerbook},
\[\mathrm{diam} (f_{s}^{-1}(\gamma[s,t]))\leq C_{1}\mathrm{diam}(\gamma[s,t])^{\frac{1}{2}}\Big(\sup_{z\in\gamma[s,t]} \mathrm{Im}[z]\Big)^{\frac{1}{2}}\]
for some universal constant $C_{1}>0$.
Hence,
$(a_{2}(t)-b_{2}(t) )-(a_{2}(s)-b_{2}(s)$, $(a_{1}(t)-a_{1}(s))$, $(b_{1}(t)-b_{1}(s) )$ and $(r_{j}(t)-r_{j}(s))$ are uniformly bounded above by a constant multiple of $\sqrt{t-s}$. Hence we also have
\begin{eqnarray}&&|X^{\delta}_{t}-X^{\delta}_{s}|\leq C_{2}\sqrt{t-s}\nonumber\\
&&|Y^{\delta}_{t}-Y^{\delta}_{s}|\leq C_{2}\sqrt{t-s}\label{kjeq1}\end{eqnarray}
for some constant $C_{2}>0$ independent of $\delta$.
 Then by the Arzel\'{a}-Ascoli theorem, for any sequence $(\delta_{k})$ with $\delta_{k}\searrow 0$, we can find subsequence  $(\delta_{n_{k}})$ such that $X^{\delta_{n_{k}}}\rightarrow X$ and $Y^{\delta_{n_{k}}}\rightarrow Y$ as $k\rightarrow \infty$ where $X$ and $Y$ are also continuous and of finite variation (since they are also the sum of finitely many monotonic increasing/decreasing functions).

By the Helly selection principle, the law of $\mathcal{R}^{\delta}_{\cdot}|\mathfrak{A}^{+}_{\infty}$ has weak convergent subsequence for the sequence $(\delta_{n_{k}})$. Tightness of $\mathcal{R}^{\delta}_{\cdot}|\mathfrak{A}^{+}_{\infty}$ is guaranteed by its uniform integrability which is a consequence of  Lemma \ref{turn1} and  the fact that any moment of $|X^{\delta_{n_{k}}}_{t}-X^{\delta_{n_{k}}}_{s}|$ and $|Y^{\delta_{n_{k}}}_{t}-Y^{\delta_{n_{k}}}_{s}|$ is bounded. Using the Kolmogorov extension theorem, we call the corresponding weak limit $\mathcal{R}^{*}_{\cdot}$.

Then for any $n$,
\[\mathbb{E}\Big[\left|\frac{\left(\mathcal{R}^{\delta_{n_{k}}}_{t}-\mathcal{R}^{\delta_{n_{k}}}_{s}\right)|\mathfrak{A}^{+}_{\infty}-(Y^{\delta_{n_{k}}}_{t}-Y^{\delta_{n_{k}}}_{s})}{X^{\delta_{n_{k}}}_{t}-X^{\delta_{n_{k}}}_{s}}\right|^{n}\Big]\leq\mathbb{E}[ (\mathcal{W}^{\delta_{n_{k}} }|\mathfrak{A}^{+}_{\infty})^{n} ].\]
Since
\[\mathbb{E}[(\mathcal{W^{\delta}}|\mathfrak{A}^{+}_{\infty})^{n}]\]
does not depend on $\delta$ by Lemma \ref{turn1}, using the Skorokhod-Dudley Theorem and Fatou's Lemma, we have
\[\mathbb{E}\Big[\left|\frac{\mathcal{R}^{*}_{t}-\mathcal{R}^{*}_{s}-(Y_{t}-Y_{s})}{X_{t}-X_{s}}\right|^{n}\Big] <\infty.\]
From the above, we also have
\[|X_{t}-X_{s}|<C_{2}\sqrt{t-s}\]
for any $0\leq s<t\leq T$.
Then using the version of the Kolmogorov-Centsov Continuity Theorem in the Appendix (Theorem \ref{KC}), this implies that for any $0<\epsilon<1$,
\begin{equation}|(\mathcal{R}^{*}_{t}-\mathcal{R}_{s}^{*})-(Y_{t}-Y_{s})|\leq B_{\epsilon}(\omega_{t}-\omega_{s})^{1-\epsilon}
\label{loeq1}\end{equation}
for some almost surely finite random variable $B_{\epsilon}$ with $\mathbb{E}[B_{\epsilon}^{q}]<\infty$ for every $q>1$ and finite variation process $\omega_{t}$. Hence
  $R^{*}(t)-Y_{t}$ is of finite $(1+\epsilon)$-variation almost surely. Since $Y_{t}$ is of finite variation, this implies that $\mathcal{R}^{*}_{t}$ is of finite $(1+\epsilon)$-variation almost surely.
\end{proof}

Combining Lemmas \ref{mel3}, \ref{mel5} and \ref{mel6} with Proposition \ref{prop13} allows us to obtain convergence of $\xi^{\delta}_{t}|\mathfrak{A}_{\infty}^{+}$ to an $\epsilon$-semimartingale.
\begin{theorem}\label{mel2}
Suppose that $T>0$, and $N$ satisfies $t_{m_{N-1}}<T\leq t_{m_{N}}$.
 Then for any sequence $(\delta_{k})$ with $\delta_{k}\searrow 0$ as $k\rightarrow \infty$,
 $\xi^{\delta_{k}}_{t}|\mathfrak{A}^{+}_{\infty}$ has a subsequence which converges uniformly in distribution to an $\epsilon$-semimartingale on $[0,T]$.
\end{theorem}
\begin{proof}
Using Proposition \ref{prop13}, we only need to show that $\mathcal{H}_{t}^{\delta}|\mathfrak{A}_{\infty}^{+}$ converges subsequentially to a martingale. We first take $0<\epsilon<\epsilon'<\min\{\frac{\alpha}{2},\frac{1}{2}\}$. We then fix  sufficiently small $\delta>0$ and set $\eta\triangleq\min\{\delta^{\frac{\alpha}{2}-\epsilon'},\delta^{\frac{1}{2}-\epsilon'}\}$. Since $N$ satisfies
\[t_{m_{N-1}}< T\leq t_{m_{N}},\]
by (\ref{xsx1}), we have
\begin{equation}N-1\leq \frac{2\pi}{7}\frac{T}{\eta}. \label{beq7}\end{equation}
Then by Lemmas \ref{mel3}, \ref{mel5} and \ref{mel6}, we have for all $n=1,\ldots N$,
\begin{equation}\eta^{-1}\left|\mathbb{E}\left[\mathcal{H}^{\delta}_{t_{m_{n}}} -\mathcal{H}^{\delta}_{t_{m_{n-1}}}  |\mathcal{F}_{\varsigma_{\delta}(t_{m_{n-1}})},\mathfrak{A}_{t_n}^{+}
 \right]\right|\leq C_{1}\delta^{\epsilon'-\epsilon}, \label{brb10}\end{equation}
for some $C_{1}$ not depending on $n$ or $\delta$.
Then define for each $n=1,2,\ldots$,
\begin{equation}\mathcal{M}^{\delta}_{t_{m_n}}\triangleq
\sum_{j=1}^{n}\Big[\big(\mathcal{H}^{\delta}_{t_{m_j}} -\mathcal{H}^{\delta}_{t_{m_{j-1}}} \big)|\mathfrak{A}_{\infty}^{+}
-\mathbb{E}\big[\mathcal{H}^{\delta}_{t_{m_j}} -\mathcal{H}^{\delta}_{t_{m_{j-1}}}  |\mathcal{F}_{\varsigma_{\delta}(t_{m_{j-1}})},\mathfrak{A}_{\infty}^{+}\big]\Big],\label{nbeq1}\end{equation}
and for $t\in(t_{m_{n-1}},t_{m_n})$, $\mathcal{M}^{\delta} _{t}$ is the linear interpolation between $\mathcal{M}^{\delta} _{t_{m_{n-1}}}$ and $\mathcal{M}^{\delta} _{t_{m_n}}$.
We have
\[\big|\big(\mathcal{H}^{\delta}_{t_{m_n}} |\mathfrak{A}_{\infty}^{+}\big)-\mathcal{M}^{\delta} _{t_{m_n}}\big|\leq N \sup_{j=1,\ldots,N}\big|\mathbb{E}\left[\mathcal{H}^{\delta}_{t_{m_j}} -\mathcal{H}^{\delta}_{t_{m_{j-1}}}
|\mathcal{F}_{\varsigma_{\delta}(t_{m_{j-1}})},\mathfrak{A}_{\infty}^{+}\right]\big|.\]
Hence by (\ref{beq7}) and (\ref{brb10}),
\begin{equation}
\sup_{j=1,\ldots,N}|\big(\mathcal{H}^{\delta}_{t_{m_j}} |\mathfrak{A}_{\infty}^{+}\big)-\mathcal{M}^{\delta}_{t_{m_j}}|\leq C_{1}\delta^{\epsilon'-\epsilon}\searrow 0 \text{ as }\delta\searrow 0, \label{qq2}
\end{equation}
and by construction, $\mathcal{M}^{\delta}_{t_{m_0}},\ldots,\mathcal{M}^{\delta}_{t_{m_N}}$ is a martingale with respect to the filtration $\{\mathcal{F}_{\varsigma_{\delta}(t_{m_{n}})}\}$.
 By the Skorokhod embedding theorem \cite{RY}, there exists a Brownian motion $\mathbb{B}_{t}$ and a sequence of
stopping times $\tau^{\delta}_{0}\leq\tau^{\delta}_{1}\leq\ldots\leq\tau^{\delta}_{N}$ such
that $\mathbb{B}_{\tau^{\delta}_{n}}=\mathcal{M}^{\delta}_{t_{m_n}}$.
We now define for each rational $t$,
\[ \mathcal{N}_{t}\triangleq \liminf _{\delta\searrow 0}\mathbb{B}_{\tau^{\delta}_{n^{\delta}(t)}}\]
where $n^{\delta}(t)$ satisfies $t\in [t_{m_{n^{\delta}(t)}},  t_{m_{n^{\delta}(t)+1}})$. This is well defined since \[t_{m_{n^{\delta}(t)+1}}-t_{m_{n^{\delta}(t)}}\searrow 0 \text{ as }\delta\searrow 0.\] Using a diagonalization argument, we can find subsequence $(\delta_{k})$ with $\delta_{k}\searrow 0$ as $k\rightarrow\infty$ such that for all rational $t$,
\begin{equation}\lim_{k\rightarrow\infty} \mathbb{B}_{\tau^{\delta_{k}}_{n^{\delta_{k}}(t)}}=\mathcal{N}_{t}.\label{ppeq5}\end{equation}
Now, again from (\ref{weq6}) and (\ref{zeq7}), we can write for any rational $t$ and $s$ with $s<t$,
\[
\xi^{\delta}_{t}-\xi^{\delta}_{s}= R_{s}(t)-R_{s}(s) + \sum_{k=N(s)+1}^{N(t)} L_{k}(a_{k}(t_{m_{n}})-b_{k}(t_{m_{n}})).
\]
Similar to the proof of Proposition \ref{prop13}, we can telescope the sum to show that
\begin{equation}(\xi^{\delta}_{t}-\xi^{\delta}_{s})|\mathfrak{A}_{\infty}^{+}=\big(\mathcal{W}^{\delta}|\mathfrak{A}_{\infty}^{+}\big)(\widetilde{X}^{\delta}_{t}-\widetilde{X}^{\delta}_{s})+\widetilde{Y}^{\delta}_{t}-\widetilde{Y}^{\delta}_{s}\label{ppeq6}\end{equation}
where
\begin{eqnarray}&&|\widetilde{X}^{\delta}_{t}-\widetilde{X}^{\delta}_{s}|\leq C_{2}\sqrt{t-s} \nonumber\\&&
|\widetilde{Y}^{\delta}_{t}-\widetilde{Y}^{\delta}_{s}|\leq C_{2}\sqrt{t-s}\label{kjeq2}\end{eqnarray}
for some constant $C_{2}$ not depending on $\delta$.
Lemma \ref{turn1} implies that
\begin{eqnarray}\mathbb{E}[|\xi^{\delta}_{t}-\xi^{\delta}_{s}|^{p}|\mathfrak{A}_{\infty}^{+}]\nonumber&<&C_{2}\big(\mathbb{E}\big[\big(\mathcal{W}^{\delta}\big)^{p}|\mathfrak{A}_{\infty^{+}}\big]+1\big)(t-s)^{\frac{p}{2}}
\\ &<& C_{1}(p)(t-s)^{\frac{p}{2}}\label{nbeq2}
\end{eqnarray}
for some $C_{1}(p)$ depending on $p$ but not $\delta$. From (\ref{ppeq2}),
\[\mathcal{H}^{\delta_{k}}_{t_{m_{n}}}|\mathfrak{A}_{\infty}^{+} =(\xi^{\delta_{k}}_{t_{m_{n}}}|\mathfrak{A}_{\infty}^{+})- (\mathcal{R}^{\delta_{k}}_{t_{m_{n}}}|\mathfrak{A}_{\infty}^{+})\]
Hence from (\ref{ppeq1}), (\ref{nbeq1}) and (\ref{ppeq6}),
\begin{equation}\mathbb{E}[|\mathcal{M}_{t}^{\delta}-\mathcal{M}^{\delta}_{s}|^{p}]\leq C_{2}(p)(t-s)^{\frac{p}{2}}\label{nbeq4}\end{equation}
for constant $C_{2}(p)$ depending only on $p$ but not $\delta$.

Similarly, from (\ref{ppeq2}),
\begin{equation}\xi^{\delta_{k}}_{t_{m_{n}}}|\mathfrak{A}_{\infty}^{+}=\big( (\mathcal{H}^{\delta_{k}}_{t_{m_{n}}}|\mathfrak{A}_{\infty}^{+} )-\mathcal{M}^{\delta_{k}}_{t_{m_{n}}} \big) + (\mathcal{R}^{\delta_{k}}_{t_{m_{n}}}|\mathfrak{A}_{\infty}^{+})+\mathcal{M}^{\delta_{k}}_{t_{m_{n}}}.\label{ppeq4}\end{equation}
 The first term in (\ref{ppeq4}) converges to zero using (\ref{qq2}); passing to a further subsequence, the second term $\mathcal{R}^{\delta_{k}}_{t}|\mathfrak{A}_{\infty}^{+}\rightarrow \mathcal {R}^{*}_{t}$ by Proposition \ref{prop13}; and the third term $\mathcal{M}^{\delta_{k}}_{t_{m_{n}}}=\mathbb{B}_{\tau_{n}^{\delta}}$ converges to $\mathcal{N}_{t}$ by (\ref{ppeq5}). Then (\ref{nbeq2}), Fatou's lemma and the Kolmogorov-Centsov continuity theorem imply that  for any $\widetilde{\epsilon}>0$,
\begin{eqnarray}&&|(\mathcal{N}_{t}-\mathcal{N}_{s})+(\mathcal{R}^{*}_{t}-\mathcal{R}_{s}^{*})|\leq B_{\widetilde{\epsilon}}(t-s)^{\frac{1}{2}-\widetilde{\epsilon}}
\nonumber\\ \Rightarrow && |\mathcal{N}_{t}-\mathcal{N}_{s}|\leq B'_{\widetilde{\epsilon}}( (t-s)^{\frac{1}{2}-\widetilde{\epsilon}} +(\omega_{t}-\omega_{s})^{1-\widetilde{\epsilon}}+(Y_{t}-Y_{s}))
\label{ppeq7}\end{eqnarray}for some almost surely finite random variable $B_{\widetilde{\epsilon}}$ and $B'_{\widetilde{\epsilon}}$ and the same process $\omega_{t}$ of finite variation as  in the proof of Proposition \ref{prop13}.

Then by (\ref{ppeq1}) and (\ref{ppeq6}),
\[\xi^{\delta}_{T}|\mathfrak{A}_{\infty}^{+}=\mathcal{W}^{\delta}\widetilde{X}^{\delta}_{T}+\widetilde{Y}_{T}^{\delta}\]
and
\[\mathcal{R}^{\delta}_{T}|\mathfrak{A}_{\infty}^{+}=\mathcal{W}^{\delta}X^{\delta}_{T}+Y_{T}^{\delta}.\]
In particular, $\xi^{\delta}_{T}|\mathfrak{A}_{\infty}^{+}$ and $\mathcal{R}^{\delta}_{T}|\mathfrak{A}_{\infty}^{+}$ are uniformly $L^{p}$-bounded by Lemma \ref{turn1}, (\ref{kjeq1}) and (\ref{kjeq2}).
 Hence using (\ref{qq2}), Lemma \ref{turn1} and Doob's $L^{p}$ inequality,
\begin{eqnarray*}\sup_{\delta>0}\mathbb{E}\big[\big(\sup_{t\in[0,T]} \mathcal{M}_{t}^{\delta}\big)^{p}\big]&\leq&C_{3}(p)\sup_{\delta>0}\mathbb{E}\big[( \mathcal{M}_{T}^{\delta})^{p}\big]
\\&\leq&C_{4}(p)\big(\sup_{\delta>0}\mathbb{E}\big[( \xi^{\delta}_{T}|\mathfrak{A}_{\infty}^{+})^{p}\big]+\sup_{\delta>0}\mathbb{E}\big[(\mathcal{R}^{\delta}_{T}|\mathfrak{A}_{\infty}^{+})^{p} \big]\big)
<\infty\end{eqnarray*}
for some constants $C_{3}(p)$ and $C_{4}(p)$ depending only on $p$.
This in particular implies that $\{\mathcal{M}^{\delta}_{\cdot}\}_{\delta>0}$ are uniformly integrable and hence $\mathcal{N}_{\cdot}$ is also uniformly integrable on [0,T]. We define a new filtration \[\mathcal{G}_{t}\triangleq\sigma\{\mathcal{N}_{u},u\leq t\}.\]
 Clearly, $\mathcal{N}_{t}$ is adapted to $\{\mathcal{G}_{t}\}$. Hence for any bounded continuous function $f:\mathbb{R}^{j}\rightarrow \mathbb{R}$ and $u_{1}<\ldots u_{j}\leq s<t$,
\begin{eqnarray*}\mathbb{E}[\mathcal{N}_{t}f(\mathcal{N}_{u_{1}},\ldots, \mathcal{N}_{u_{j}})]&=&\mathbb{E}\big[\lim_{k\rightarrow\infty} \mathbb{B}_{\tau^{\delta_{k}}_{n^{\delta_{k}}(t)}}f(\mathbb{B}_{\tau^{\delta_{k}}_{n^{\delta_{k}}(u_{1})}},\ldots, \mathbb{B}_{\tau^{\delta_{k}}_{n^{\delta_{k}}(u_{j})}})\big]
\\&=&\lim_{k\rightarrow\infty} \mathbb{E}\big[ \mathbb{B}_{\tau^{\delta_{k}}_{n^{\delta_{k}}(t)}}f\big(\mathbb{B}_{\tau^{\delta_{k}}_{n^{\delta_{k}}(u_{1})}},\ldots, \mathbb{B}_{\tau^{\delta_{k}}_{n^{\delta_{k}}(u_{j})}}\big)\big]
\\&=&\lim_{k\rightarrow\infty}  \mathbb{E}\big[ \mathcal{M}^{\delta_{k}}_{t_{m_{n^{\delta_{k}}(t)}}}f\big(\mathcal{M}^{\delta_{k}}_{t_{m_{n^{\delta_{k}}(u_{1})}}},\ldots, \mathcal{M}^{\delta_{k}}_{t_{m_{n^{\delta_{k}}(u_{j})}}}\big)\big]
\\&=&\lim_{k\rightarrow\infty}  \mathbb{E}\big[ \mathcal{M}^{\delta_{k}}_{t_{m_{n^{\delta_{k}}(s)}}}f\big(\mathcal{M}^{\delta_{k}}_{t_{m_{n^{\delta_{k}}(u_{1})}}},\ldots, \mathcal{M}^{\delta_{k}}_{t_{m_{n^{\delta_{k}}(u_{j})}}}\big)\big]
\\&=&\lim_{k\rightarrow\infty} \mathbb{E}\big[ \mathbb{B}_{\tau^{\delta_{k}}_{n^{\delta_{k}}(s)}}f\big(\mathbb{B}_{\tau^{\delta_{k}}_{n^{\delta_{k}}(u_{1})}},\ldots, \mathbb{B}_{\tau^{\delta_{k}}_{n^{\delta_{k}}(u_{j})}}\big)\big]
\\&=&\mathbb{E}[\mathcal{N}_{s}f(\mathcal{N}_{u_{1}},\ldots, \mathcal{N}_{u_{j}})],
\end{eqnarray*}
where we have swapped the limit with the expectation by uniform integrability and used the martingale property of $\mathcal{M}^{\delta}_{t_{m_{n}}}$. This is the martingale property for $\mathcal{N}_{t}$. Hence by (\ref{ppeq7}) and this property, there is a continuous modification of $\mathcal{N}_{t}$ which is also a continuous time martingale with respect to the naturally induced extension of the filtration $\mathcal{G}_{t}$. We will use the same symbol to denote this modification.

Then for any $\widehat{\epsilon}>0$,
\begin{eqnarray*}&&\limsup_{k\rightarrow\infty}\mathbb{P}\Big[ \sup_{t\in[0,T]}\big|(\mathcal{H}_{t}^{\delta_{k}}|\mathfrak{A}_{\infty}^{+})-\mathcal{N}_{t}\big|>\widehat{\epsilon}\Big]\\&\leq&
\limsup_{k\rightarrow\infty}\mathbb{P}\Big[ \sup_{j=1,\ldots,N}|\big(\mathcal{H}^{\delta_{k}}_{t_{m_j}} |\mathfrak{A}_{\infty}^{+}\big)-\mathcal{M}^{\delta_{k}}_{t_{m_j}}|+
\sup_{t\in[0,T]}|\mathcal{M}^{\delta_{k}}_{t}-\mathcal{N}_{t}|>\widehat{\epsilon}\Big]
\\&=&
\limsup_{k\rightarrow\infty}\mathbb{P}\Big[ \sup_{t\in[0,T]}|\mathcal{M}^{\delta_{k}}_{t}-\mathcal{N}_{t}|>\widehat{\epsilon}\Big]
\end{eqnarray*}
where the second line follows from (\ref{qq2}). Taking a partition $0=s_{0}<\ldots<s_{M}=T$ such that $|s_{j}-s_{j-1}|<\theta$ for some $\theta>0$.
\begin{eqnarray*}
&&
\limsup_{k\rightarrow\infty}\mathbb{P}\Big[ \sup_{t\in[0,T]}|\mathcal{M}^{\delta_{k}}_{t}-\mathcal{N}_{t}|>\widehat{\epsilon}\Big]
\\&=& \limsup_{k\rightarrow\infty}\mathbb{P}\Big[ \max_{j=0,\ldots,M-1}\sup_{t\in[s_{j},s_{j+1}]}|\mathcal{M}^{\delta_{k}}_{t}-\mathcal{N}_{t}|>\widehat{\epsilon}\Big]
\\&\leq&\limsup_{k\rightarrow\infty}\mathbb{P}\Big[ \max_{j=0,\ldots,M-1}\sup_{t\in[s_{j},s_{j+1}]}\big(|\mathcal{M}^{\delta_{k}}_{t}-\mathcal{M}^{\delta_{k}}_{s_{j}}|+
|\mathcal{M}^{\delta_{k}}_{s_{j}}-
\mathcal{N}_{s_{j}}|+|\mathcal{N}_{s_{j}}-\mathcal{N}_{t}|\big)>\widehat{\epsilon}\Big]
\\&\leq&\limsup_{k\rightarrow\infty}\mathbb{P}\Big[ \max_{j=0,\ldots,M-1}\sup_{t\in[s_{j},s_{j+1}]}\big(|\mathcal{M}^{\delta_{k}}_{t}-\mathcal{M}^{\delta_{k}}_{s_{j}}|+|\mathcal{N}_{s_{j}}-\mathcal{N}_{t}|\big)>\widehat{\epsilon}\Big]
\end{eqnarray*}
where the last equality is by (\ref{ppeq5}).  Then,
\begin{eqnarray*}
&&\limsup_{k\rightarrow\infty}\mathbb{P}\Big[ \max_{j=0,\ldots,M-1}\sup_{t\in[s_{j},s_{j+1}]}|\mathcal{M}^{\delta_{k}}_{t}-\mathcal{M}^{\delta_{k}}_{s_{j}}|+|\mathcal{N}_{s_{j}}-\mathcal{N}_{t}|>\widehat{\epsilon}\Big]
\\&\leq& \limsup_{k\rightarrow\infty}\Big( \mathbb{P}\Big[ \max_{j=0,\ldots,M-1}\sup_{t\in[s_{j},s_{j+1}]}|\mathcal{N}_{s_{j}}-\mathcal{N}_{t}|>\frac{\widehat{\epsilon}}{2}\Big]
\\ &&\qquad \qquad+ \sum_{j=0}^{M-1}\mathbb{P}\Big[\sup_{\{n:t_{m_{n-1}}\text{ or } t_{m_{n}}\in[s_{j},s_{j+1}]\}}|\mathcal{M}^{\delta_{k}}_{t_{m_{n}}}-\mathcal{M}^{\delta_{k}}_{s_{j}}|>\frac{\widehat{\epsilon}}{2}\Big]\Big)
\\&\leq& \limsup_{k\rightarrow\infty}\Big( \mathbb{P}\Big[ \max_{j=0,\ldots,M-1}\sup_{t\in[s_{j},s_{j+1}]}|\mathcal{N}_{s_{j}}-\mathcal{N}_{t}|>\frac{\widehat{\epsilon}}{2}\Big]\\&&\qquad \qquad +\sum_{j=0}^{M-1}\Big(\frac{\widehat{\epsilon}}{2}\Big)^{-p}\mathbb{E}\big[\sup_{\{n:t_{m_{n-1}}\text{ or } t_{m_{n}}\in[s_{j},s_{j+1}]\}}|\mathcal{M}^{\delta_{k}}_{t_{m_{n}}}-\mathcal{M}^{\delta_{k}}_{s_{j}}|^{p}\big]\Big)
\end{eqnarray*}
where we have used the linear interpolation of $\mathcal{M}_{t}^{\delta}$ and Markov's inequality. Let $n^{*}_{j}=\inf\{n:t_{m_{n}}\geq s_{j+1}\}$. Then $|t_{m_{n^{*}_{j}}}-s_{j}|\leq 2\theta$ and by Doob's $L^{p}$ inequality,
\begin{eqnarray*}\mathbb{E}\big[\sup_{\{n:t_{m_{n-1}}\text{ or } t_{m_{n}}\in[s_{j},s_{j+1}]\}}|\mathcal{M}^{\delta_{k}}_{t_{m_{n}}}-\mathcal{M}^{\delta_{k}}_{s_{j}}|^{p}\big]
&\leq& C_{5}(p)\mathbb{E}\big[\big|\mathcal{M}^{\delta_{k}}_{t_{m_{n^{*}_{j}}}}-\mathcal{M}^{\delta_{k}}_{s_{j}}\big|^{p}\big]
\\ &\leq& C_{6}(p)\theta^{\frac{p}{2}},
\end{eqnarray*}
where the last inequality uses  (\ref{nbeq4}) (for some constants $C_{5}(p)$ and $C_{6}(p)$ depending only on $p$). Hence,
\begin{eqnarray*}&&\limsup_{k\rightarrow\infty}\mathbb{P}\Big[ \sup_{t\in[0,T]}\big|(\mathcal{H}_{t}^{\delta_{k}}|\mathfrak{A}_{\infty}^{+})-\mathcal{N}_{t}\big|>\widehat{\epsilon}\Big]\\&\leq&
\limsup_{k\rightarrow\infty}\Big( \mathbb{P}\Big[ \max_{j=0,\ldots,M-1}\sup_{t\in[s_{j},s_{j+1}]}|\mathcal{N}_{s_{j}}-\mathcal{N}_{t}|>\frac{\widehat{\epsilon}}{2}\Big] +C_{7}(p)\frac{\theta^{\frac{p}{2}-1}}{\widehat{\epsilon}^{p}}\Big)
\end{eqnarray*}
Then choosing $p>2$ and using the modulus of continuity of $\mathcal{N}_{t}$ given by (\ref{ppeq7}) we get the result by picking $\theta$ arbitrarily small.
\end{proof}
We now establish the subsequential driving term convergence of the $+\partial$CBP.
\begin{theorem}\label{coro3}
Suppose that $T>0$ and $\mathbf{D}=(D,a,b)\in\mathcal{D}$. Let $\widetilde{\xi}^{\delta}_{t}$ denote the driving function of the $+\partial$CBP from $a$ to $b$ in $D$ on the lattice of mesh size $\delta$. Then for any sequence $(\delta_{k})$ with $\delta_{k}\searrow 0$ as $k\rightarrow \infty$, $\widetilde{\xi}^{\delta_{k}}_{t}$ has a subsequence which converges  uniformly in probability to an $\epsilon$-semimartingale on $[0,T]$.
\end{theorem}
\begin{proof}
We prove this theorem by repeatedly constructing a coupling of the +CBP with the $+\partial$CBP such that their respective driving functions are close. Fix $\delta>0$ sufficiently small and let $\nu_{0}$ be a +CBP from $a$ to $b$ in $D$, let $Z_{0},Z_{1},\ldots$ be the vertices of $\nu_{0}$. Similarly, let $\widetilde{\nu}_{0}$ be a $+\partial$CBP from $a$ to $b$ in $D$, and let $\widetilde{Z}_{0},\widetilde{Z}_{1},\ldots$ be the vertices of $\widetilde{\nu}_{0}$. Then we can couple $\nu_{0}$ and $\widetilde{\nu}_{0}$ until the first step $N$ where the number of possible choices for $Z_{N+1}$ is strictly greater than the number of possible values for $\widetilde{Z}_{N+1}$. We call the vertex $Z_{N}=\widetilde{Z}_{N}$ a \emph{distinguishing} vertex. For $j<k$, let $[Z_{j},Z_{k}]$ denote the subpath of $\nu_{0}$ between $Z_{j}$ and $Z_{k}$ --  in particular, $[Z_{j},Z_{j+1}]$ is the edge from $Z_{j}$ to $Z_{j+1}$.

Then there are two possibilities for $Z_{N}$:
\begin{enumerate}
\item[Case 1:] $Z_{N}$ is a free vertex of the +CBP.
\item[Case 2:] $Z_{N}$ is a non-free vertex of the +CBP.
\end{enumerate}

We consider Case 1. If $Z_{N}$ is a free vertex, then by definition of the +CBP, $[Z_{N},Z_{N+1}]$ is perpendicular to $[Z_{N-1},Z_{N}]$ and there are two possibilities for $[Z_{N},Z_{N+1}]$ which we denote $L$ and $R$. Since $Z_{N}$ is distinguishing, this means that either

\begin{itemize}
\item $[\widetilde{Z}_{N}, \widetilde{Z}_{N+1}]=L$ or;
\item $[\widetilde{Z}_{N}, \widetilde{Z}_{N+1}]=R$.
\end{itemize}
\pagebreak
\begin{figure}[hp]
 \begin{center}
\scalebox{0.5}{\includegraphics{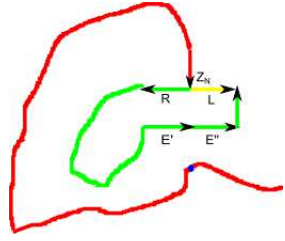}}
 \end{center}
\caption{Case 1. The green path denotes a possible path for the +CBP after $Z_{N}$ up to $Z_{M}$. We replace $[Z_{N},Z_{M}]$ with $R$. } \label{fig8}
\end{figure}

\begin{figure}[hp]
 \begin{center}
\scalebox{0.8}{\includegraphics{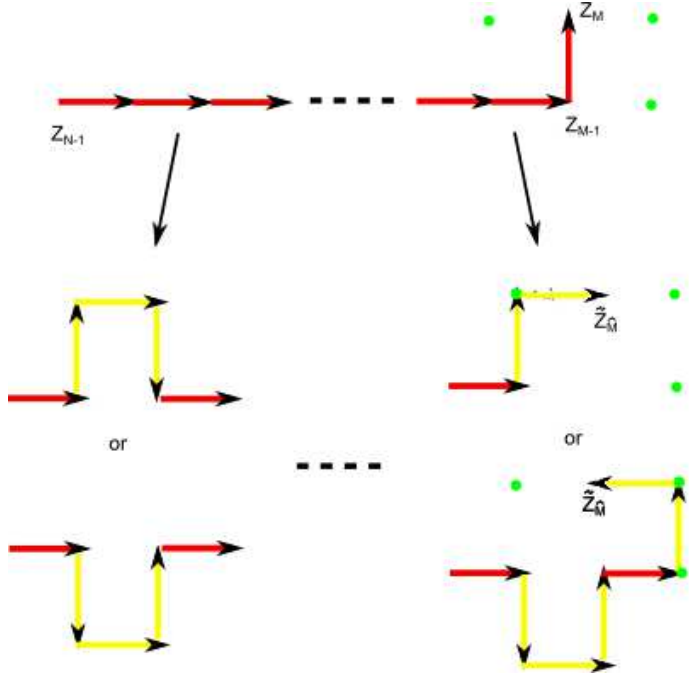}}
 \end{center}
\caption{We couple the paths $[Z_{N}, Z_{M_{0}}]$ with $[\widetilde{Z}_{N}, \widetilde{Z}_{\widehat{M}_{0}}]$ using the diagram. Note that the definition of the +CBP guarantees that only one of the choices at each part is possible and the connectivity of the boundary implies that we can always make such a coupling. } \label{fig9}
\end{figure}

Without loss of generality, we can assume that $[\widetilde{Z}_{N}, \widetilde{Z}_{N+1}]=L$. Let $\omega$ denote $Z_{N}-Z_{N-1}$. In this case, we must have $Z_{N}+\omega$ is not in $\partial D\cap\{Z_{n}\}_{n=0,\ldots,N}$ (otherwise, $Z_{N}$ would be non-free) and $Z_{N}+2\omega$ is in $\partial D\cap\{Z_{n}\}_{n=0,\ldots,N}$ (otherwise there would be two possible choices for $\widetilde{Z}_{N+1}$ so $Z_{N}$ would not be distinguishing). See Figure \ref{fig8}.

This means that any path from $Z_{N}$ to $b$ on the shifted brick-wall lattice whose first edge is $R$, must have two consecutive edges $E'=[Z_{M_{0}-3},Z_{M_{0}-2}]$, $E''=[Z_{M_{0}-2},Z_{M_{0}-1}]$ for some $M_{0}$ such that
\begin{enumerate}
\item $E'$ and $E''$ go in the same direction and;
\item $Z_{M_{0}-2}=Z_{N}+\omega$.
\end{enumerate}
We suppose that $Z_{M_{0}}=Z_{N+1}$. This situation is illustrated in Figure \ref{fig8}. Let $\widehat{M}_{0}=N+1$. Hence we can couple the paths $[\widetilde{Z}_{0}, \widetilde{Z}_{\widehat{M}_{0}}]$ and $[Z_{0}, Z_{M_{0}}]$ such that their respective harmonic measures with respect to a given point $\zeta \in D_{0}$ differ by at most $C(\zeta)\delta$. Here $D_{0}$ is the connected component of $D\setminus\{ z: \mathrm{dist}(z,[Z_{0},Z_{N}])\leq 3\delta\}$ with $b$ on its boundary.

We now consider Case 2 i.e. $Z_{N}$ is a non-free vertex. Let
\begin{eqnarray*}&&M_{0}\triangleq\min\big\{m \geq N+1 :  [Z_{m-2},Z_{m-1}] \text{ and } [Z_{m-1},Z_{m}] \text{ are perpendicular}\big\}.\end{eqnarray*}
We remark that $Z_{N+1},\ldots, Z_{M_{0}-1}$ are non-free vertices of the $+$CBP by definition, and $[Z_{N-1},Z_{M_{0}-1}]$ is a straight line. Note that $M_{0}\geq N+2$, since if $M_{0}=N+1$, then we can couple $[Z_{0},Z_{N+1}]$ with $[\widetilde{Z}_{0},\widetilde{Z}_{N+1}]$ so $Z_{N}$ (which is non-free) would not be distinguishing.
 Let \[\rho\triangleq\max\{k=0,1,2,\ldots:N+3k<M_{0}-1\}.\]
 For $\rho\geq 1$, we will show that we can couple
$[Z_{N-1},Z_{N-1+3\rho}]$ with a $+\partial$CBP path such that these paths are within a $2\delta$ neighbourhood of each other. We perform the coupling procedure inductively: Set $\mu_{0}=N-1$. Suppose that $k=1,\ldots,\rho-1$ and we can find a $\mu_{k-1}>\mu_{k-2}$ such that we can couple $[\widetilde{Z}_{N-1},\widetilde{Z}_{\mu_{k-1}}]$ and $[Z_{N-1},Z_{N-1+3(k-1)}]$ such that the paths are within a 2$\delta$ neighborhood of each other. Note that at the vertex $\widetilde{Z}_{\mu_{k-1}+1}$, the +$\partial$CBP either turns left with probability 1 or right with probability 1 (otherwise, if the $+\partial$CBP can either turn left or right with positive probability, this contradicts the fact that $Z_{N+3k}$ is non-free). If it turns left with probability 1, we couple $[Z_{0},Z_{N-1+3k}]$ with $[\widetilde{Z}_{0},\widetilde{Z}_{\mu_{k-1}}]\cup \mathcal{E}_{k}$ where $\mathcal{E}_{k}$ is the union of the edges passed through by turning left, right, right, then left at $\widetilde{Z}_{\mu_{k-1}}$ (see the the left hand side of the middle row of Figure \ref{fig9}). If the +CBP turns right with probability 1 at $\widetilde{Z}_{\mu_{k-1}+1}$, we couple $[Z_{0},Z_{N-1+3k}]$ with $[\widetilde{Z}_{0},\widetilde{Z}_{\mu_{k-1}}]\cup \mathcal{E}_{k}$ where $\mathcal{E}_{k}$ is now the union of the edges passed through by turning right, left, left, then right at $\widetilde{Z}_{\mu_{k-1}}$ (see the left hand side of the bottom row of Figure \ref{fig9}). We define $\mu_{k}$ such that $\widetilde{Z}_{\mu_{k}}$ is the last vertex of $\mathcal{E}_{k}$. Note that $[\widetilde{Z}_{0},\widetilde{Z}_{\mu_{k}}]$ and $[Z_{0},Z_{N-1+3k}]$ lie in a $2\delta$ neighbourhood of one another. This completes the inductive construction.

Finally, suppose that the +CBP turns left at $Z_{M_{0}-1}$ and suppose that $[Z_{N-1+3\rho},$ $Z_{M_{0}}]$ is not a possible sequence of edges for the $+\partial$CBP at $Z_{\mu_{\rho-1}-1}$ otherwise we automatically obtain a coupling and we stop.  Then we can couple $[Z_{0},Z_{M_{0}}]$ with $[\widetilde{Z}_{0},\widetilde{Z}_{\mu_{\rho-1}}]\cup \mathcal{E}_{\rho}$ where $\mathcal{E}_{\rho}$ is either
\begin{itemize}
\item the edge  passed through by turning left then right at $\widetilde{Z}_{\mu_{\rho-1}}$ (see the right hand side of the middle row of Figure \ref{fig9});
\item the union of the edges passed through by turning right, left, left, right, left, then left at $\widetilde{Z}_{\mu_{\rho-1}}$ (see the right hand side of the bottom row of Figure \ref{fig9}).
\end{itemize}
Similarly for the case where the +CBP turns right at $Z_{M_{0}-1}$ (exchanging ``left'' with ``right'').
 We define $\mu_{\rho}$ such that $\widetilde{Z}_{\mu_{\rho}}$ is the last vertex of $\mathcal{E}_{\rho}$.
We then define $\widehat{M}_{0}=\mu_{\rho}$. Hence we obtain a coupling of $[\widetilde{Z}_{0},\widetilde{Z}_{\widehat{M}_{0}}]$ and $[Z_{0},Z_{M_{0}}]$ such that the paths lie within a $2\delta$ neighbourhood of one another.

Let $\Gamma_{0}=[Z_{0},Z_{M_{0}}]$ and  $\widetilde{\Gamma}_{0}=[\widetilde{Z}_{0}, \widetilde{Z}_{\widehat{M}_{0}}]$. Let $\nu_{1}$ be a +CBP from $\widetilde{Z}_{\widehat{M}_{0}}$ to $b$ in $D\setminus\widetilde{\Gamma}_{0}$ and $\widetilde{\nu}_{1}$ be a $+\partial$CBP from $\widetilde{Z}_{\widehat{M}_{0}}$ to $b$ in $D\setminus\widetilde{\Gamma}_{0}$. We index the vertices of the +CBP and $+\partial $CBP by $M_{0}+1,M_{0}+2,\ldots$ and $\widetilde{M}_{0}+1,\widetilde{M}_{0}+2,\ldots$ respectively. We then apply the same argument to obtain $M_{1}>M_{0}$, $\widehat{M}_{1}>\widetilde{M}_{0}$, $\Gamma_{1}=[Z_{M_{0}},Z_{M_{1}}]$ and $\widetilde{\Gamma}_{1}=[ \widetilde{Z}_{\widehat{M}_{0}},\widetilde{Z}_{\widehat{M}_{1}}]$. Continuing inductively, we obtain two sequences of paths $\{\Gamma_{k}=[ \widetilde{Z}_{\widehat{M}_{k}},Z_{M_{k+1}}]\}$ and $\{\widetilde{\Gamma}_{k}=[ \widetilde{Z}_{\widehat{M}_{k}},\widetilde{Z}_{\widehat{M}_{k+1}}]\}$ such that  $\Gamma_{k}$ is a $+$CBP path starting from $Z_{M_{k-1}}$ in $D\setminus \bigcup_{j=0}^{k-1} \widetilde{\Gamma}_{j}$; $\widetilde{\Gamma}_{k}$ is a $+\partial$CBP path starting from $\widetilde{Z}_{\widehat{M}_{k-1}}$ in $D\setminus \bigcup_{j=0}^{k-1} \widetilde{\Gamma}_{j}$; and $\Gamma_{k}$ is coupled with $\widetilde{\Gamma}_{k}$ such that their respective driving functions are close.

For $k=0,1,2,\ldots$, let $\xi^{k,\delta}_{t}$ for $t\in[0,T^{\delta}_{k}]$ denotes the driving function of $\Gamma_{k}$ in $D\setminus \bigcup_{j=0}^{k-1} \widetilde{\Gamma}_{j}$ and
$\widetilde{\xi}^{k,\delta}_{t}$ for $t\in[0,\widetilde{T}^{\delta}_{k}]$ denotes the driving function of $\widetilde{\Gamma}_{k}$ in $D\setminus \bigcup_{j=0}^{k-1} \widetilde{\Gamma}_{j}$. Then, by the Carath\'{e}odory kernel theorem, for each $k=0,1,2,\ldots$
\begin{equation}\sup_{t\in[0,T_{k}\wedge \widetilde{T}_{k}]} \left|\widetilde{\xi}_{t}^{k,\delta}-\xi_{t}^{k,\delta} \right|\rightarrow 0 \text{ as }\delta\searrow 0. \label{yoeq2}\end{equation}

Now note that $\widetilde{\Gamma}=\widetilde{\Gamma_{0}}\cup\widetilde{ \Gamma}_{1} \cup \ldots$ is a $+\partial$CBP in $D$ from $a$ to $b$. Also, by Theorem \ref{mel2}, for any sequence of mesh-sizes $(\delta_{k})$,
we can find subsequence $(\delta_{n_{k}})$ such that the driving function of the non-crossing path $\Gamma=\Gamma_{0}\cup \Gamma_{1} \cup \ldots$ converges uniformly in distribution to an $\epsilon$-semimartingale $W_{t}$. Combining this fact with (\ref{yoeq2}), we get the result.
\end{proof}

\section{Driving term convergence of the bond percolation exploration process on the square lattice to SLE$_{6}$} \label{section8}
Since the decomposition of $\epsilon$-martingales is unique, we can consider an integral for $\epsilon$-semimartingales as the sum of a It\^{o} integral and a Young integral. In particular the calculus for $\epsilon$-semimartingales is identical to the calculus for usual semimartingales.
\pagebreak
 \begin{figure}[hp]
 \begin{center}
\scalebox{0.5}{\includegraphics{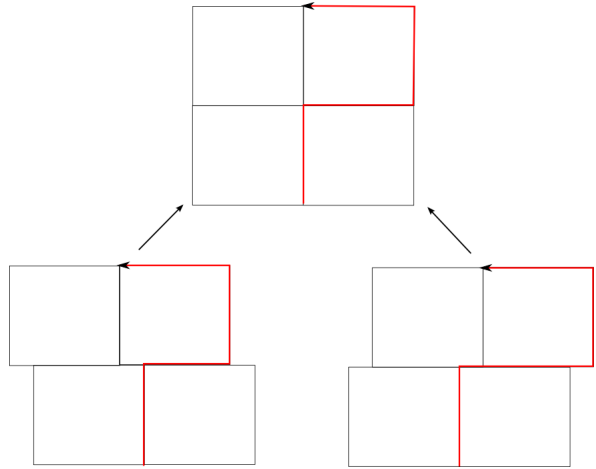}}
 \end{center}
\caption{A $+$BP/$-$BP on the shifted brick-wall lattice as the limit of a path on the $+\epsilon$-brick-wall lattice (on the right) and $-\epsilon$-brick-wall lattice (on the left). Note that the $+$BP cannot turn right for the next step but the $-$BP can turn right.} \label{fig13}
\end{figure}

Now define $Q:\mathbb{C}\rightarrow \mathbb{C}$ by $Q(x+iy)=\frac{\sqrt{3}}{3}x+i\frac{2}{3}y$. Then $Q$ maps the rectangles in the shifted brick-wall lattice of mesh-size $\delta$ to squares in the square lattice of mesh-size $\delta$.

Now consider a $+\partial$CBP, $\widetilde{\nu}$, on the shifted brick-wall lattice of mesh-size $\delta$ and a SqP, $\nu$, on the square lattice of mesh-size $\delta$. $\widetilde{\nu}$ is almost the same path as $Q^{-1}(\nu)$ except that at certain vertices of the lattice, $\widetilde{\nu}$ cannot create a loop at corner of two sites. In other words, some vertices of the path are non-free vertices of $\widetilde{\nu}$ but the corresponding vertices are free vertices of $\nu$ (in the sense defined in Sections \ref{sect2} and \ref{sect3}). This is due to topological restrictions of the $\epsilon$-brick-wall lattice for $\epsilon>0$ from which we constructed the +CBP and $+\partial$CBP. A more formal way to think of this is that, for $\epsilon>0$, the topology on the $\epsilon$-brick-wall lattice induces a fine topology on the shifted brick-wall lattice.

 However the vertices at which a $-\partial$CBP can create a loop are exactly the vertices of the lattice at which a $+\partial$CBP cannot create a loop. Moreover, from the remark at the end of Section \ref{sect3}, the driving function of the $-\partial$CBP also converges subsequentially to an $\epsilon$-semimartingale as we take the scaling-limit. See Figure \ref{fig13}.

We establish the subsequential driving term convergence for $Q^{-1}(\nu)$.
 \begin{proposition}\label{prop14}
For any $\mathbf{D}=(D,a,b)\in\mathcal{D}$ and $T>0$, let $\xi^{\delta} _{t}$ denote the driving function of $Q^{-1}(\nu)$ where $\nu$ is a bond percolation exploration process on the square lattice of mesh size $\delta>0$ in $\mathbf{D}$. Then for any sequence $(\delta_{k})$ with $\delta_{k}\searrow 0$ as $k\rightarrow \infty$,
 $\xi^{\delta_{k}} _{t}$ has a subsequence which converges uniformly in distribution to an $\epsilon$-semimartingale on $[0,T]$.
\end{proposition}

We first need a few lemmas in order to establish this proposition. As in the statement of the proposition, let $\xi^{\delta}_{t}$ denote the driving function of $Q^{-1}(\nu)$ where $\nu$ is a bond percolation exploration process on the square lattice of mesh size $\delta>0$ in $(D,a,b)$.

By (\ref{weq6}), we can write (using the notation in Section \ref{notation}),
\begin{equation}\xi_{t}^{\delta}=\frac{1}{2}\left(a_{1}(t)+b_{1}(t)\right)+\frac{1}{2}\big(\sum_{j=1}^{M}\rho_{j}(r_{j}-r_{j}(t))\big)
+\frac{1}{2}\big(\sum_{k=2}^{N(t)}\widetilde{L}_{k}(a_{k}(t)-b_{k}(t))\big),\label{yoeq1}\end{equation}
where  $(\widetilde{L}_{k})$ is the turning sequence of the $Q^{-1}(\nu)$.
We first need the following two lemmas.
\begin{lemma}\label{boundaryCBP1}
For $T>0$, let $\widetilde{N}^{\delta}\triangleq\inf\{n:t_{n}>t\}$ and
\[\widetilde{\mathcal{W}}^{\delta}\triangleq\max_{l=1,\ldots,\widetilde{N}^{\delta}}\sum_{j=2}^{l} \big|\widetilde{L}_{j}\big|.\]
Then
\[\mathbb{P}[\widetilde{\mathcal{W}}^{\delta}\geq l]<C_{1}e^{-C_{2}l}.\]
For constants $C_{1}$, $C_{2}$ not depending on $\delta$. In particular, for any $p>0$,
\[\mathbb{E}[(\widetilde{\mathcal{W}}^{\delta})^{p}]<\infty.\]
\end{lemma}
\begin{proof}
The proof is identical to the second part of the proof of Lemma \ref{turn1} using the independence of the process in disjoint rectangles.
\end{proof}
\begin{lemma}\label{boundaryCBP2}
We can write
\[\xi_{t}^{\delta}=\widetilde{\mathcal{W}}^{\delta} \widehat{X}_{t}^{\delta}+\widehat{Y}_{t}^{\delta},\]
where for $t>s$,
\begin{eqnarray*}
&&|\widehat{X}_{t}^{\delta}-\widehat{X}_{s}^{\delta}|\leq C\sqrt{t-s}
\\&&|\widehat{Y}_{t}^{\delta}-\widehat{Y}_{s}^{\delta}|\leq C\sqrt{t-s}
\end{eqnarray*}
for some constant $C>0$ not depending on $\delta$.
\end{lemma}
\begin{proof}
This lemma follows from telescoping the second sum in (\ref{yoeq1}) as in the part of the proof of Theorem \ref{mel3} leading up to (\ref{ppeq6}) and (\ref{kjeq2}).
\end{proof}
\begin{proof}[Proof of Proposition \ref{prop14}]
Let $\nu$ be a SqP from $a$ to $b$ in $D$ and let $\overline{\nu}=Q^{-1}(\nu)$ which is a path on the shifted brick-wall lattice. Let $Z_{0},Z_{1},\ldots$ be the vertices of $\overline{\nu}$.  We use an iterative coupling method as in the proof of Theorem \ref{coro3}. Let $\nu_{0}$ be a $+\partial$CBP from $Q^{-1}(a)$ to $Q^{-1}(b)$ in $Q^{-1}(D)$ and denote its vertices by $\widetilde{Z}_{0},\widetilde{Z}_{1}$. Then by construction, we can couple $\overline{\nu}$ and $\nu_{0}$ until the first step $N_{0}$ such that $Z_{N_{0}}$ is free but $\widetilde{Z}_{N_{0}}$ is non-free.

Now let $\nu_{1}$ be a $-\partial$CBP from  $Z_{N_{0}}$ to $Q^{-1}(b)$ in $Q^{-1}(D)$ and for simplicity, we will slightly abuse the notation and write the vertices of $\nu_{1}$ as $\widetilde{Z}_{N_{0}},\widetilde{Z}_{N_{0}+1},\widetilde{Z}_{N_{0}+2},\ldots$ . By the discussion preceding the statement of the proposition, $\widetilde{Z}_{N_{0}}$ is a free vertex of $\nu_{1}^{*}$, a $-\partial$CBP on $D\setminus\gamma[0,t_{N_{0}}]$ from $Z_{N_{0}}$ to $b$ (since it is a non-free vertex of $\nu_{0}$). We then can couple $\nu_{1}$ with the subpath of $\overline{\nu}$ starting from $Z_{N_{0}}$ until the first step $N_{1}\geq N_{0}$ that $Z_{N_{1}}$ is free but $\widetilde{Z}_{N_{1}}$ is non-free.

We now let $\nu_{2}$ be a $+\partial$CBP from  $Z_{N_{1}}$ to $Q^{-1}(b)$ in $Q^{-1}(D)$ and as before, we write the vertices of $\nu_{1}$ as $\widetilde{Z}_{N_{1}},\widetilde{Z}_{N_{1}+1},\widetilde{Z}_{N_{1}+2},\ldots$ . As above we can couple $\nu_{2}$ with the subpath of $\overline{\nu}$ starting from $Z_{N_{1}}$ until the first step $N_{2}\geq N_{1}$ that $Z_{N_{2}}$ is free but $\widetilde{Z}_{N_{2}}$ is non-free. We proceed inductively, alternating the coupling with the $-\partial$CBP and the $+\partial$CBP.

This implies that the Loewner driving function of $\overline{\nu}$ satisfies
\[\xi ^{\delta}_{t}=\widetilde{\xi}^{\rho^{\delta}_{t},\delta}_{t},\]
where $\rho^{\delta}_{t}=\sup\{k:t_{N_{k}}<t\}$ and for $k=0,2,4,\ldots$, $\widetilde{\xi}^{k,\delta}_{t}$ is the driving function of a $+\partial$CBP from time $t_{N_{k}}$ to $t_{N_{k+1}}$; and for  $k=1,3,5,\ldots$, $\widetilde{\xi}^{k,\delta}_{t}$ is the driving function of a $-\partial$CBP from time $t_{N_{k}}$ to $t_{N_{k+1}}$.
We let $\Xi_{t}^{\delta}\triangleq \widetilde{\xi}^{\rho^{\delta}_{t},\delta}_{t}$.

By Lemmas \ref{boundaryCBP1} and \ref{boundaryCBP2}, we apply the same method as in the last part of the proof of Proposition \ref{prop13} to show that, for any sequence $(\delta_{j})$ with $\delta_{j}\searrow 0 $ as $j\rightarrow\infty$, there exists subsequence $(\delta_{n_{j}})$ such that $\xi^{\delta_{n_{j}}}_{\cdot}$ converges weakly. This implies that $\Xi_{\cdot}^{\delta_{n_{j}}}$ also converges weakly.

Now by Theorem \ref{mel2}, for any sequence $(\delta_{j})$ such that $\delta_{j}\searrow 0$ as $j\rightarrow \infty$,
we can find subsequence $(\delta_{n_{j}})$ such that $\widetilde{\xi}^{k,\delta_{n_{j}}}_{t}$ can be decomposed into
\[\widetilde{\xi}^{k,\delta_{n_{j}}}_{t}= \mathcal{H}^{k,\delta_{n_{j}}}_{t}+\mathcal{R}^{k,\delta_{n_{j}}}_{t}\]
where for each $k=0,1,2,\ldots$, $\mathcal{H}_{t}^{k,\delta_{n_{j}}}$  converges in distribution to a martingale and $\mathcal{R}_{t}^{k,\delta_{n_{j}}}$ converges in distribution to a finite $(1+\epsilon)$-variation process as $j\rightarrow \infty$.  Then we can write
\[\Xi_{t}^{\delta_{n_{j}}}=\mathcal{H}^{*,\delta_{n_{j}}}_{t}+\mathcal{R}^{*,\delta_{n_{j}}}_{t}\]
where
\[\mathcal{H}^{*,\delta}_{t}\triangleq \mathcal{H}^{\rho^{\delta}_{t},\delta}_{t}, \]
\[\mathcal{R}^{*,\delta}_{t}\triangleq \mathcal{R}^{\rho^{\delta}_{t},\delta}_{t}. \]
In particular, $\mathcal{H}^{*,\delta}_{t}$ and $\mathcal{R}^{*,\delta}_{t}$ are continuous.

 By Lemma \ref{boundaryCBP1}, the winding of $Q^{-1}(\nu)$ is uniformly $L^{p}$ bounded. By this fact and the version of the Kolmogorov-Centsov continuity theorem in the Appendix (Theorem \ref{KC}), for each $\delta_{n_{j}}$, we can find an almost surely finite random variable $B^{\delta_{n_{j}}}_{\epsilon}$ such that
\[\sup_{D}\sum_{i=0}^{N-1} |\mathcal{R}^{*,\delta_{n_{j}}}_{t_{i+1}}-\mathcal{R}^{*,\delta_{n_{j}}}_{t_{i}}|^{\frac{1}{1-\epsilon}}\leq B^{\delta_{n_{j}}}_{\epsilon}|\omega^{\delta_{n_{j}}}_{t}-\omega^{\delta_{n_{j}}}_{s}| \]
for some finite variation process $\omega^{\delta_{n_{j}}}_{t}$ where the supremum is taken over all partitions $D=\{s=t_{0}<t_{1}<\ldots<t_{N}=t\}$. Moreover, $B^{\delta_{n_{j}}}_{\epsilon}$ is $L^{p}$ bounded for any $p>1$ and hence there exists a subsequential weak limit $B_{\epsilon}$ as $j\rightarrow\infty$. Also, from the proof of Proposition \ref{prop13} and Theorem \ref{KC}, $\omega^{\delta_{n_{j}}}_{t}$ can be taken as the variation of $X^{\delta_{n_{k}}}_{t}$ which is bounded uniformly by the sum of finitely many monotonic increasing and decreasing functions. This implies that
$\omega^{\delta_{n_{j}}}_{t}$ converges almost surely to some finite variation process $\omega_{t}$ as $j\rightarrow\infty$. Hence, by passing to a further subsequence, $\mathcal{R}^{*,\delta_{n_{j}}}_{t}$
converges to a finite $(1+\epsilon)$-variation process as $j\rightarrow \infty$.

This also implies that
$\mathcal{H}^{*,\delta_{n_{j}}}_{t}$
also converges uniformly in distribution as $j\rightarrow\infty$ to the sum of martingale differences which is a martingale.
\end{proof}

For any sequence $\Delta=(\delta_{k})$ with $\delta_{k}\searrow 0$ as $k\rightarrow \infty$ such that $\xi^{\delta_{k}} _{t}$ converges to an $\epsilon$-semimartingale, we denote the limit by $V^{\Delta,\mathbf{D}}_{t}$.

Now consider $Q^{-1}(\nu)$. We approximate $Q^{-1}(\nu)$ by a smooth path $\nu^{\delta}$  by smoothing out the corners of the path in the interior of each site (see Figure \ref{fig14}). Mapping $\nu^{\delta}$ to $\mathbb{H}$, we get a curve $\gamma^{\delta}:[0,\infty)\mapsto\overline{\mathbb{H}}$  parametrized by half-plane capacity with chordal driving function $\xi^{\delta}(t)$ and  associated conformal map $g_{t}^{\delta}$ satisfying the chordal Loewner differential equation. Then let $g^{Q,\delta}_{t}$ be the conformal map of $\mathbb{H}\setminus Q\circ\gamma(0,t]$ onto $\mathbb{H}$ that are normalized hydrodynamically. Let $Q^{\delta}_{t}=g^{Q,\delta}_{t}\circ Q \circ (g^{\delta}_{t})^{-1}$.
Then $g^{Q,\delta}_{t}$ satisfies
\begin{equation}\label{rot1}\dot{g}_{t}^{Q,\delta}(z)=\frac{b_{Q,\delta}(t)}{g_{t}^{Q,\delta}(z)-Q^{\delta}_{t}(\xi^{\delta}(t))}\end{equation}
for some $b_{Q,\delta}(t)>0$.
Note that $Q^{\delta}_{t}$ can be extended to a quasiconformal mapping on a full neighbourhood of $\xi^{\delta}(t)$ in $\mathbb{C}$ by reflection.

\begin{lemma}\label{Qt}
 $Q^{\delta}_{t}(x+iy)$ has smooth partial derivatives with respect to $x$ and $y$ and differentiable with respect to $t$ at $z\in\mathbb{R}$ sufficiently close to $\xi^{\delta}(t)$.
\end{lemma}
\begin{proof}
First note that since $\gamma^{\delta}$ is smooth, $Q^{\delta}_{t}$ can be extended to a smooth function for $x\in\mathbb{R}$ sufficiently close to for $\xi^{\delta}(t)$. Also, the fact that $\gamma^{\delta}$ is a smooth curve implies that $\xi^{\delta}(t)$ and $Q_{t}^{\delta}\circ \xi^{\delta}(t)$ are smooth as functions of $t$ as well. If we define $\overline{Q}_{t}(z)\triangleq Q^{\delta}_{t}(z+\xi^{\delta}(t))-Q^{\delta}_{t}(\xi^{\delta}(t))$, then $\overline{Q}_{t}(0)=0$ and hence
$\dot{\overline{Q}}_{t}(0)=0$. Since $Q^{\delta}_{t}(z)=\overline{Q}_{t}(z-\xi^{\delta}(t))+Q_{t}(\xi^{\delta}(t))$, this implies that $Q^{\delta}_{t}(z)$ is differentiable with respect to $t$ at $z=\xi^{\delta}(t)$.

\pagebreak
 \begin{figure}[hp]
 \begin{center}
\scalebox{0.5}{\includegraphics{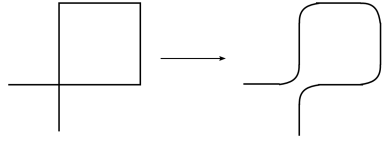}}
 \end{center}
\caption{Approximating a path on the lattice by a smooth path.} \label{fig14}
\end{figure}

Furthermore, from the Loewner differential equation and (\ref{rot1}), we have
\[\dot{Q}^{\delta}_{t}(\xi^{\delta}(t))=\frac{b_{Q,\delta}(t)}{Q^{\delta}_{t}(z)-Q^{\delta}_{t}(\xi^{\delta}(t))}-\frac{2}{z-\xi^{\delta}(t)}\frac{\partial}{\partial z}Q^{\delta}_{t}(z)\]

Hence we can write $\dot{Q}^{\delta}_{t}(\xi^{\delta}(t))$ in terms of the partial derivatives of $Q^{\delta}_{t}(x+iy)$ with respect to $x$ and $y$ evaluated at $z=\xi^{\delta}(t)$.
\end{proof}
By It\^{o}'s formula and Lemma \ref{Qt}, for subsequence $\Delta=(\delta_{k})$, $W^{\Delta,k,\mathbf{D}}_{t}\triangleq Q^{\delta_{k}}_{t}(V^{\Delta,\mathbf{D}}_{t})$ is an $\epsilon$-semimartingale. Also, since $\{Q_{t}^{\delta_{k}}\}$ forms a normal family, by passing to a further subsequence we can assume that $Q^{\delta_{k}}_{t}$ converges locally uniformly to a limit $Q_{t}$. We let $W^{\Delta,\mathbf{D}}_{t}\triangleq Q_{t}(V^{\Delta,\mathbf{D}}_{t})$. We write $W^{\Delta,k}_{t}=W^{\Delta,k,\mathbf{H}}_{t}$ and $W^{\Delta}_{t}=W^{\Delta,\mathbf{H}}_{t}$.
We also let $\xi_{t}^{\mathbf{D},\delta}$ denote the driving function of the SqP in $\mathbf{D}=(D,a,b)$ on the lattice of mesh-size $\delta$. Then note that for 
\begin{equation}\sup_{t\in[0,T]}|W_{t}^{\Delta,k,\mathbf{D}}-\xi_{t}^{\mathbf{D},\delta_{k}}|=o_{p}(1)\label{czeq1}\end{equation}
as $k\rightarrow \infty$. Here $o_{p}(1)$ denotes a random variable that converges in probability to $0$ as $k\rightarrow\infty$. Hence we can couple $W_{t}^{\Delta,k,\mathbf{D}}$ and $\xi_{t}^{\mathbf{D},\delta_{k}}$ such that
\[\sup_{t\in[0,T]}|W_{t}^{\Delta,k,\mathbf{D}}-\xi_{t}^{\mathbf{D},\delta_{k}}|\rightarrow 0 \text{ almost surely as } k\rightarrow\infty.\]

Let $M_{t}^{\Delta,k,\mathbf{D}}$ be the martingale part of $W_{t}^{\Delta,k,\mathbf{D}}$. We now adopt a localization argument since, in spite of (\ref{czeq1}), we cannot guarantee the uniform $L^{p}$ boundedness of $M_{t}^{\Delta,k,\mathbf{D}}$.
For each $N\in\mathbb{N}$, we define stopping times $\varphi^{\Delta,k\mathbf{D}}_{N}=\inf\{t:|W_{t}^{\Delta,k,\mathbf{D}}|\geq N\}$. Then $\varphi^{\Delta,k\mathbf{D}}_{N}$ converges to some stopping time $\varphi^{\Delta,\mathbf{D}}_{N}$ almost surely as $k\rightarrow\infty$. Then we can find a local martingale $M^{\Delta,\mathbf{D}}_{t}$ such that $M_{t\wedge\varphi_{N}^{\Delta,k,\mathbf{D}}}^{\Delta,k,\mathbf{D}}$ converges to $M_{t\wedge\varphi_{N}^{\Delta,\mathbf{D}}}^{\Delta,\mathbf{D}}$ since we trivially have uniform $L^{2}$ boundedness. As above, we write $M_{t}^{\Delta,k}=M_{t}^{\Delta,k,\mathbf{H}}$, $M_{t}^{\Delta}=M_{t}^{\Delta,\mathbf{H}}$ and 
$\varphi_{N}^{\Delta,k}=\varphi_{N}^{\Delta,k,\mathbf{H}}$,  $\varphi_{N}^{\Delta}=\varphi_{N}^{\Delta,\mathbf{H}}$.

We now apply the previous results to obtain convergence of the Loewner driving term of a SqP to $\sqrt{6}B_{t}$.

\begin{theorem}\label{th2}
For any  $T>0$ and $\mathbf{D}=(D,a,b)\in\mathcal{D}$. Let $\xi^{\mathbf{D},\delta}_{t}$ denote the driving function of the SqP in $(D,a,b)$ on the lattice of mesh size $\delta$. Then  for any sequence $(\delta_{k})$ with $\delta_{k}\searrow 0$ as $k\rightarrow\infty$, there is a subsequence $(\delta_{n_{k}})$ such that  $\xi^{\mathbf{D},\delta_{n_{k}}}_{t}$ converges uniformly in distribution to $\sqrt{6}B_{t}$ on $[0,T]$ as $k\rightarrow \infty$.
\end{theorem}
\begin{proof}
Let $(Z_{k} )$ denote the vertices of a SqP $\gamma$ parametrized by the half-plane capacity in $\mathbb{H}$ from $0$ to $\infty$ on the square lattice of mesh size $\delta>0$. Now take any $\mathbf{D}=(D,a,b)\in\mathcal{D}$ such that $\mathrm{diam}(D)>2\delta$ and consider the SqP $\gamma^{\mathbf{D}}$ in $D$ from $a$ to $b$ on the lattice of mesh size $\delta$. By translation, we can assume $a=0$ since the driving function $\xi^{\mathbf{D},\delta}_{t}$ does not change under translation. Similarly, we can rotate $D$ by a multiple of $\pi/2$ radians about 0 without changing the driving function. Hence, without loss of generality, we can assume that $D\cap\mathbb{H}\neq \emptyset$.
 Let $(Z_{k}^{\mathbf{D}})$ denote the vertices of a  SqP in $D$ from $a$ to $b$ on the square lattice of mesh size $\delta>0$. For each $\delta>0$, define a stopping time \[T^{\mathbf{D}}_{\delta}\triangleq\inf\{j: \mathrm{dist}( Z_{j} ,\partial (D\cap \mathbb{H})\setminus\mathbb{R})<2\delta\}.\]
Then $T^{\mathbf{D}}_{\delta}\not\equiv 0$. By the locality property, we can couple the two processes $(Z_{k} )$ and $(Z_{k}^{\mathbf{D}} )$ such that $Z_{k} =Z_{k}^{\mathbf{D}} $ for $k=0,\dots,T^{\mathbf{D}}_{\delta}$. This implies that we can couple  a time-change of the paths (since the time parametrizations of the two curves are different) up to a stopping time $\tau_{\delta}^{\mathbf{D}}$ i.e. for some increasing function $\sigma_{\delta}(t)$,
\[\gamma(\sigma^{\mathbf{D}}_{\delta}(t))=\phi_{\mathbf{D}}\circ\gamma^{\mathbf{D}}(t) \text{ for } t\in[0,\tau_{\delta}^{\mathbf{D}}].\]
Since \[g_{t}(\gamma(t))=\xi^{\delta}_{t}\text{ and } g_{t}^{\mathbf{D}}(\gamma^{\mathbf{D}}(t))=\xi_{t}^{\mathbf{D},\delta},\]
this implies that
\begin{equation}\xi_{t}^{\mathbf{D},\delta}=\Phi^{\mathbf{D},\delta} _{t}(\xi^{\delta}_{\sigma^{\mathbf{D}}_{\delta}(t)}),\label{ppe1}\end{equation}
where \[\Phi^{\delta}_{t}(z)=g^{\mathbf{D}}_{t}\circ \phi_{\mathbf{D}}\circ g_{\sigma^{\mathbf{D}}_{\delta}(t)}(r^{\mathbf{D}}_{\delta}z),\]
and $r^{\mathbf{D}}_{\delta}>0$ is chosen in such a way that
\begin{equation}\left(\Phi^{\mathbf{D},\delta}_{0}\right)'(0)=1.\label{coeq1}\end{equation}
Note that the Schwarz reflection principle implies that $\Phi^{\mathbf{D},\delta}_{t}$ can be extended analytically to a neighbourhood of $\xi_{\sigma_{\delta}(t)}$ in $\mathbb{C}$.
  Also, by (4.15) in \cite{lawlerbook} note that $\sigma^{\mathbf{D}}_{\delta}(t)$ satisfies
\begin{equation}\dot{\sigma}^{\mathbf{D}}_{\delta}(t)=\left(\Phi^{\mathbf{D},\delta}_{t}\right)'(\xi^{\delta}_{\sigma^{\mathbf{D}}_{\delta}(t)})^{2}.\label{ppe2}\end{equation}
so $\dot{\sigma}_{\delta}^{\mathbf{D}}(0)=1$. Moreover, since \[\{\Phi_{t}^{\mathbf{D},\delta}\}_{t\in[0,\tau_{\delta}^{\mathbf{D}}],D\in\mathcal{D},\delta>0}\] forms a normal family (by Montel's theorem), we can assume that
$\sigma_{\delta}^{\mathbf{D}}(u)=u+o(u)$ as $u\searrow 0$ where $o(u)$ does not depend on $\mathbf{D}$ or $\delta$.

Also, by Proposition 4.40 in \cite{lawlerbook}, \begin{equation}\dot{\Phi}_{t}^{\mathbf{D},\delta}(\xi^{\delta}_{\sigma^{\mathbf{D}}_{\delta}(t)})=-3(\Phi_{t}^{\mathbf{D},\delta})''(\xi^{\delta}_{\sigma^{\mathbf{D}}_{\delta}(t)}).\label{ppe3}\end{equation}

For any sequence $\Delta=(\delta_{k})$ such that $W^{\Delta,k}_{t}$ is defined, we consider
$\Phi_{t}^{\mathbf{D},\delta_{k}}(W^{\Delta,k}_{t})$. We can write
$dW^{\Delta,k}_{t}=X^{\Delta,k}_{t}dB_{t}+dY^{\Delta,k}_{t}$ where we can assume that $X^{\Delta,k}_{t}$ is right continuous at $t=0$ and $Y^{\Delta,k}_{t}$ is a finite $(1+\epsilon)$ variation process. Continuity of $\Phi^{\mathbf{D},\delta_{k}}_{t}$ and its derivatives and $Q^{\delta_{k}}_{t}$ along with (\ref{czeq1}) implies that from (\ref{ppe1}), (\ref{ppe2}) and (\ref{ppe3}), we get
\begin{equation}W^{\Delta,k,\mathbf{D}}_{t}=\Phi^{\mathbf{D},\delta_{k}} _{t}(W^{\Delta,k}_{\sigma^{\mathbf{D}}_{\delta_{k}}(t)})+ o_{p}(1),\label{weq3}\end{equation}
\[\dot{\sigma}^{\mathbf{D}}_{\delta_{k}}(t)=e^{ 2o_{p}(1)}\cdot(\Phi^{\mathbf{D},\delta_{k}}_{t})'(W^{\Delta,k}_{\sigma^{\mathbf{D}}_{\delta_{k}}(t)})^{2},\]
\[\dot{\Phi}_{t}^{\mathbf{D},\delta_{k}}(W^{\Delta,k}_{\sigma^{\mathbf{D}}_{\delta_{k}}(t)})=(-3+ o_{p}(1))(\Phi_{t}^{\mathbf{D},\delta_{k}})''(W^{\Delta,k}_{\sigma^{\mathbf{D}}_{\delta_{k}}(t)}).\]
where $o_{p}(1)$ denotes a random variable that converges in probability to $0$ as $k\rightarrow \infty$.
By It\^{o}'s formula,
\begin{eqnarray}&&
\Phi^{\mathbf{D},\delta_{k}} _{t}(W^{\Delta,k}_{\sigma^{\mathbf{D}}_{\delta_{k}}(t)})-\Phi_{0}^{\mathbf{D},\delta_{k}}(W^{\Delta,k}_{0}) \label{weq4}
\\&=&\int_{0}^{t}e^{-o_{p}(1)}X_{t}^{\Delta,k}\sqrt{\dot{\sigma}^{\mathbf{D}}_{\delta_{k}} (s)}dB_{\sigma^{\mathbf{D}}_{\delta_{k}} (s)}+\int_{0}^{t} (\Phi_{s}^{\mathbf{D},\delta_{k}})'(W^{\Delta,k}_{\sigma^{\mathbf{D}}_{\delta_{k}}(s)})dY^{\Delta,k}_{\sigma^{\mathbf{D}}_{\delta_{k}}(s)}\nonumber\\&&+\int_{0}^{t}e^{-2 o_{p}(1)}\left(\frac{(X_{\sigma^{\mathbf{D}}_{\delta_{k}}(s)}^{\Delta,k})^{2}}{2}-3+ o_{p}(1)\right)\frac{(\Phi_{s}^{\mathbf{D},\delta_{k}})''(W^{\Delta,k}_{\sigma^{\mathbf{D}}_{\delta_{k}}(s)})}{(\Phi_{s}^{\mathbf{D},\delta_{k}})'(W^{\Delta,k}_{\sigma^{\mathbf{D}}_{\delta_{k}}(s)})^{2}}ds\nonumber
\\&=&\int_{0}^{t}e^{-o_{p}(1)}X_{\sigma^{\mathbf{D}}_{\delta_{k}}(s)}^{\Delta,k}d\widetilde{B}_{s}+\int_{0}^{t} (\Phi_{s}^{\mathbf{D},\delta_{k}})'(W^{\Delta,k}_{\sigma^{\mathbf{D}}_{\delta_{k}}(s)})dY^{\Delta,k}_{\sigma^{\mathbf{D}}_{\delta_{k}}(s)}\nonumber\\&&+\int_{0}^{t}e^{-2 o_{p}(1)}\left(\frac{(X_{\sigma^{\mathbf{D}}_{\delta_{k}}(s)}^{\Delta,k})^{2}}{2}-3+ o_{p}(1)\right)\frac{(\Phi_{s}^{\mathbf{D},\delta_{k}})''(W^{\Delta,k}_{\sigma^{\mathbf{D}}_{\delta_{k}}(s)})}{(\Phi_{s}^{\mathbf{D},\delta_{k}})'(W^{\Delta,k}_{\sigma^{\mathbf{D}}_{\delta_{k}}(s)})^{2}}ds,\nonumber\end{eqnarray}
where
\[\widetilde{B}_{t}\triangleq\int_{0}^{t}\sqrt{\dot{\sigma}^{\mathbf{D}}_{\delta_{k}} (s)}dB_{\sigma_{\delta_{k}} (s)}\]
is also a standard 1-dimensional Brownian motion. Also, since $\sigma^{\mathbf{D}}_{\delta_{k}} (t)$ is locally Lipschitz,
\[\widetilde{Y}^{\Delta,k}_{t}\triangleq \int_{0}^{t} (\Phi_{s}^{\mathbf{D},\delta_{k}})'(W^{\Delta,k}_{\sigma^{\mathbf{D}}_{\delta_{k}}(s)})dY^{\Delta,k}_{\sigma^{\mathbf{D}}_{\delta_{k}}(s)}\]
is well-defined as a Young integral  (see \cite{LQ}). Since the integrand has modulus of continuity $\alpha<\frac{1}{2}$ given by Theorem \ref{mel2}, $\Phi^{\delta_{k}}_{t}$ is smooth and $Y^{\Delta,k}_{t}$ is of finite $(1+\epsilon)$-variation for sufficiently small $\epsilon>0$. Moreover,  $\widetilde{Y}^{\Delta,k}_{t}$ is of finite $(1+\epsilon)$-variation.

Now, consider $(\mathbb{H},0,\infty)$. Conditioned on $\gamma[0,s]$, the curve $\gamma(s+t)$ is identically distributed to the SqP on \[\mathbf{H}_{s}\triangleq (\mathbb{H}\setminus\gamma(0,s],\gamma(s),\infty).\] By (\ref{weq3}) and (\ref{weq4}), the martingale part of $W^{\Delta,k,\mathbf{H}_{s}}_{t}$ is
\[M^{\Delta,k,\mathbf{H}_{s}}_{t}\triangleq\int_{0}^{t}e^{- o_{p}(1)}X^{\Delta,k}_{\sigma^{\mathbf{H}_{s}}_{\delta_{k}}(u)}d\widetilde{B}_{u}+o_{p}(1)\]
However, we also have  $W^{\Delta,k,\mathbf{H}_{s}}_{t}$ is identically distributed to $W^{\Delta,k}_{t+s}-W^{\Delta,k}_{s}$ which is the driving function of $g_{s}(\gamma(s+t))$.
Hence for all $s,h>0$, $M^{\Delta,k}_{t+s}-M^{\Delta,k}_{s}$ conditioned on $\mathcal{F}_{s}$ has the same distribution as
\[\int_{0}^{t}e^{- o_{p}(1)}X^{\Delta,k}_{\sigma^{\mathbf{H}_{s}}_{\delta_{k}}(u)}d\widetilde{B}_{u}+o_{p}(1).\]
Thus for any partition,
\[\mathcal{P}_{k}=\{0=s_{0}<s_{1}<s_{2}<\ldots<s_{N-1}<s_{N}=t\wedge\varphi_{N}^{\Delta,k}\}\]
the distribution of \[ M^{\Delta,k}_{t}=\int_{0}^{t}X^{\Delta,k}_{u}dB_{u}\] is given by the convolution product of the distributions of
\[\widetilde{M}^{\Delta,k}_{s_{i}}-\widetilde{M}^{\Delta,k}_{s_{i-1}}=\int_{0}^{s_{i}-s_{i-1}}e^{- o_{p}(1)}X^{\Delta,k}_{\sigma^{\mathbf{H}_{s_{i-1}}}_{\delta_{k}}(u)}d\widetilde{B}^{i}_{u}\] conditioned on $\mathcal{F}_{s_{i-1}}$, where $\widetilde{B}_{t}^{i}$ are independent Brownian motions. By (\ref{coeq1}), for each $i=1,\ldots N$, we have $\dot{\sigma}^{\mathbf{H}_{s_{i-1}}}_{\delta_{k}}(0)=1$. Hence for small $u>0$, using the right continuity of $X_{t}^{\Delta,k}$ at $t=0$, we have
\begin{equation}X^{\Delta,k}_{\sigma^{\mathbf{H}_{s_{i-1}}}_{\delta_{k}}(u)}=X^{\Delta,k}_{u+o(u)},\label{coeq2}\end{equation}
where $o(u)$ does not depend on $i$.
We let
\[ Q_{i}^{k}\triangleq \int_{0}^{s_{i}-s_{i-1}} e^{- o_{p}(1)}X^{\Delta,k}_{u}d\widetilde{B}^{i}_{u}.\]
\[\epsilon'_{i}\triangleq \int_{0}^{s_{i}-s_{i-1}} e^{- o_{p}(1)}(X^{\Delta,k}_{\sigma^{\mathbf{H}_{s_{i-1}}}_{\delta_{k}}(u)}-X^{\Delta,k}_{u})d\widetilde{B}^{i}_{u}.\]
Hence, $\widetilde{M}^{\Delta,k}_{s_{i}}-\widetilde{M}^{\Delta,k}_{s_{i-1}}$ conditioned on $\mathcal{F}_{s_{i-1}}$ has distribution $Q^{k}_{i}+\epsilon'_{i}$ where $\{Q^{k}_{i}\}$ are i.i.d.  random variables. Since $\sum_{i=1}^{N}\epsilon'_{i}$ is the sum of martingale differences, we can apply the Burkholder-Davis-Gundy inequality, to show that
\[ \mathbb{E}\Big[\max_{k=1,\ldots,N}\Big|\sum_{i=1}^{k}\epsilon'_{i}\Big|^{2}\Big]\leq C\sum_{i=1}^{N}\mathbb{E}\Big[\int_{0}^{s_{i}-s_{i-1}}e^{-2o_{p}(1)}(X^{\Delta,k}_{\sigma^{\mathbf{H}_{s_{i-1}}}_{\delta_{k}}(u)}-X^{\Delta,k}_{u})^{2}du\Big] .\]
Then, It\^{o} isometry and the Burkholder-Davis-Gundy inequality implies
\[\sum_{i=1}^{N} \mathbb{E}\Big[\int_{0}^{s_{i}-s_{i-1}}(X^{\Delta,k}_{u})^{2}du\Big]=C_{1} \sum_{i=1}^{N} \mathbb{E}[(M^{\Delta,k}_{s_{i}}-M^{\Delta,k}_{s_{i-1}})^{2}]\leq C_{1}\mathbb{E}[(M^{\Delta,k}_{t\wedge\varphi_{N}^{\Delta,k}})^{2}]<\infty.\]
 Similarly, 
\begin{eqnarray*}&&\sum_{i=1}^{N} \mathbb{E}\Big[\int_{0}^{s_{i}-s_{i-1}}\Big(X^{\Delta,k}_{\sigma^{\mathbf{H}_{s_{i-1}}}_{\delta_{k}}(u)}\Big)^{2}du\Big]
\\&=&\sum_{i=1}^{N} \mathbb{E}\Big[\int_{0}^{(\sigma^{\mathbf{H}_{s_{i-1}}}_{\delta_{k}})^{-1}(s_{i}-s_{i-1})}(X^{\Delta,k}_{u})^{2}(\dot{\sigma}^{\mathbf{H}_{s_{i-1}}}_{\delta_{k}}(u))^{-1}du\Big]
\\&\leq& 3C_{2}\sum_{i=1}^{N} \mathbb{E}\Big[\int_{0}^{s_{i}-s_{i-1}}(X^{\Delta,k}_{u})^{2}du\Big]
< 3C_{2}\mathbb{E}[(M^{\Delta,k}_{t\wedge\varphi_{N}^{\Delta,k}})^{2}]<\infty,
\end{eqnarray*}
for some constant $C_{2}>0$ not depending on $N$. Here we have used the fact that $\sigma^{\mathbf{H}_{s_{i-1}}}_{\delta_{k}}(u)=u+o(u)$ with $o(u)$ not depending on $i$ which implies that we can cover, the sum of integrals from $0$ to $(\sigma^{\mathbf{H}_{s_{i-1}}}_{\delta_{k}})^{-1}(s_{i}-s_{i-1})$ with 3 times the integrals from $0$ to $s_{i}-s_{i-1}$ for $|\mathcal{P}_{k}|$ sufficiently small; and (\ref{coeq1}) and (\ref{ppe2}), since $\{\Phi_{t}^{\mathbf{D},\delta}\}$ forms a normal family, the rate of convergence is uniform which implies that 
\[(\dot{\sigma}^{\mathbf{H}_{s_{i-1}}}_{\delta_{k}}(u))^{-1}<C_{2}\]

Thus by (\ref{coeq2}), using the dominated convergence theorem, we get \[\sum_{i=1}^{N} \mathbb{E}\Big[\int_{0}^{s_{i}-s_{i-1}}e^{2o_{p}(1)}(X^{\Delta,k}_{\sigma^{\mathbf{H}_{s_{i-1}}}_{\delta_{k}}(u)}-X^{\Delta,k}_{u})^{2}du\Big]\rightarrow 0 \text{ as } |\mathcal{P}_{k}|\searrow 0.\]
Hence by the Markov inequality, we have
\[ \sum_{i=1}^{N}\epsilon'_{i}\rightarrow 0 \text{ in probability as } |\mathcal{P}_{k}|\searrow 0.\]
In particular, $\{Q_{i}^{k}+\epsilon'_{i} \}$ are uniformly asymptotically negligible.
This means that $M^{\Delta}_{t}$ is a continuous, infinitely divisible process that is also a local martingale. Hence by the L\'{e}vy-Khintchine theorem, we must have $M^{\Delta}_{t}=\sqrt{\kappa_{\Delta}}B_{t}$ for some $\kappa_{\Delta}\in\mathbb{R}$ and for all $t<\varphi_{N}^{\Delta}$. Since this is true for all $N\in\mathbb{N}$, we must have
$M^{\Delta}_{t}=\sqrt{\kappa_{\Delta}}B_{t}$ for all $t$.

Hence, we have
\begin{eqnarray}&&
\Phi^{\mathbf{D},\delta_{k}} _{t}(W^{\Delta,k}_{\sigma^{\mathbf{D}}_{\delta_{k}}(t)})-\Phi_{0}^{\mathbf{D},\delta_{k}}(W^{\Delta,k}_{0}) \label{hoeq1}
\\&=&\int_{0}^{t}e^{-o_{p}(1)}\sqrt{\kappa_{\Delta}}d\widetilde{B}_{s}+\int_{0}^{t} (\Phi_{s}^{\mathbf{D},\delta_{k}})'(W^{\Delta,k}_{\sigma^{\mathbf{D}}_{\delta_{k}}(s)})dY^{\Delta,k}_{\sigma^{\mathbf{D}}_{\delta_{k}}(s)}\nonumber\\&&+\int_{0}^{t}e^{-2 o_{p}(1)}\left(\frac{\kappa_{\Delta}}{2}-3+ o_{p}(1)\right)\frac{(\Phi_{s}^{\mathbf{D},\delta_{k}})''(W^{\Delta,k}_{\sigma^{\mathbf{D}}_{\delta_{k}}(s)})}{(\Phi_{s}^{\mathbf{D},\delta_{k}})'(W^{\Delta,k}_{\sigma^{\mathbf{D}}_{\delta_{k}}(s)})^{2}}ds.\nonumber\end{eqnarray}
Then since $W_{t}^{\Delta,k,\mathbf{H}_{s}}$ is identically distributed to $W_{t+s}^{\Delta,k}-W_{s}^{\Delta,k}$, we have $Y_{t+s}^{\Delta,k}-Y_{s}^{\Delta,k}$ conditioned on $\mathcal{F}_{s}$ has the same distribution as
\begin{eqnarray}\int_{0}^{t} &&(\Phi_{s}^{\mathbf{H}_{s},\delta_{k}})'(W^{\Delta,k}_{\sigma^{\mathbf{H}_{s}}_{\delta_{k}}(s)})dY^{\Delta,k}_{\sigma^{\mathbf{H}_{s}}_{\delta_{k}}(s)} \label{pkeq1} \\
&& \ +\int_{0}^{t}e^{-2 o_{p}(1)}\left(\frac{\kappa_{\Delta}}{2}-3+ o_{p}(1)\right)\frac{(\Phi_{s}^{\mathbf{H}_{s},\delta_{k}})''(W^{\Delta,k}_{\sigma^{\mathbf{H}_{s}}_{\delta_{k}}(s)})}{(\Phi_{s}^{\mathbf{H}_{s},\delta_{k}})'(W^{\Delta,k}_{\sigma^{\mathbf{H}_{s}}_{\delta_{k}}(s)})^{2}}ds. \nonumber
\end{eqnarray}
Thus for any partition,
\[\mathcal{P}=\{0=s_{0}<s_{1}<s_{2}<\ldots<s_{N-1}<s_{N}=t\},\]
the distribution of
\[\mathcal{Y}_{t}^{\mathcal{P}}\triangleq \sum_{i=1}^{N}\int_{0}^{s_{i}-s_{i-1}} (\Phi_{s}^{\mathbf{H}_{s_{i-1}},\delta_{k}})'(W^{\Delta,k}_{\sigma^{\mathbf{H}_{s_{i-1}}}_{\delta_{k}}(s)})dY^{\Delta,k}_{\sigma^{\mathbf{H}_{s_{i-1}}}_{\delta_{k}}(s)}\]
is given by the convolution product of the conditional distributions given $\mathcal{F}_{s_{i-1}}$ of
\[\int_{0}^{s_{i}-s_{i-1}} (\Phi_{s}^{\mathbf{H}_{s_{i-1}},\delta_{k}})'(W^{\Delta,k}_{\sigma^{\mathbf{H}_{s_{i-1}}}_{\delta_{k}}(s)})dY^{\Delta,k}_{\sigma^{\mathbf{H}_{s_{i-1}}}_{\delta_{k}}(s)}.\]
We let
\[R^{k}_{i}\triangleq Y^{\Delta,k}_{s_{i}-s_{i-1}},\]
\begin{eqnarray*}\epsilon''_{i}&\triangleq& \int_{0}^{s_{i}-s_{i-1}} (\Phi_{s}^{\mathbf{H}_{s_{i-1}},\delta_{k}})'(W^{\Delta,k}_{\sigma^{\mathbf{H}_{s_{i-1}}}_{\delta_{k}}(s)})dY^{\Delta,k}_{\sigma^{\mathbf{H}_{s_{i-1}}}_{\delta_{k}}(s)}-Y^{\Delta,k}_{s_{i}-s_{i-1}} \\
&=&\int_{0}^{s_{i}-s_{i-1}} \Big[(\Phi_{s}^{\mathbf{H}_{s_{i-1}},\delta_{k}})'(W^{\Delta,k}_{\sigma^{\mathbf{H}_{s_{i-1}}}_{\delta_{k}}(s)})\dot{\sigma}^{\mathbf{H}_{s_{i-1}}}_{\delta_{k}}(s)-1\Big] dY^{\Delta,k}_{s}
\\&&+\int_{s_{i}-s_{i-1}}^{\sigma_{\delta_{k}}^{\mathbf{H}_{s_{i-1}}}(s_{i}-s_{i-1})} (\Phi_{s}^{\mathbf{H}_{s_{i-1}},\delta_{k}})'(W^{\Delta,k}_{\sigma^{\mathbf{H}_{s_{i-1}}}_{\delta_{k}}(s)})\dot{\sigma}^{\mathbf{H}_{s_{i-1}}}_{\delta_{k}}(s) dY^{\Delta,k}_{s}.\end{eqnarray*}
Hence
\[\int_{0}^{s_{i}-s_{i-1}} (\Phi_{s}^{\mathbf{H}_{s_{i-1}},\delta_{k}})'(W^{\Delta,k}_{\sigma^{\mathbf{H}_{s_{i-1}}}_{\delta_{k}}(s)})dY^{\Delta,k}_{\sigma^{\mathbf{H}_{s_{i-1}}}_{\delta_{k}}(s)}\]
conditioned on $\mathcal{F}_{s_{i-1}}$ has the same as that of $R^{k}_{i}+\epsilon''_{i}$ where $\{R^{k}_{i}\}$ are i.i.d. random variables.

Then Young's inequality (\cite{LQ},\cite{Young}), states that:
\begin{quote} If $f$ is a function of finite $p$-variation, and $g$ is a function of finite $q$-variation with $\frac{1}{p}+\frac{1}{q}>1$, then
\[ \Big|\int_{0}^{1} f dg\Big|\leq C_{p,q}(|f(0)|+V_{p}(f))V_{q}(g)\]
where $V_{p}$ denotes the corresponding $p$-variation.
\end{quote}
We let, 
\[I_{i}=(\Phi_{s}^{\mathbf{H}_{s_{i-1}},\delta_{k}})'(W^{\Delta,k}_{\sigma^{\mathbf{H}_{s_{i-1}}}_{\delta_{k}}(s)})\dot{\sigma}^{\mathbf{H}_{s_{i-1}}}_{\delta_{k}}(s)-1.\]
Hence, applying Young's inequality with $q=1+\epsilon$, since the $1+\epsilon$ variation of $Y^{\Delta,k}_{t}$ from $0$ to $T$ is finite,
\begin{eqnarray*}\Big|\sum_{i=1}^{N}\epsilon''_{i}\Big|&\leq&C_{p,1+\epsilon}V_{1+\epsilon}(Y_{s}^{\Delta,k},0,t)\max_{i=1,\ldots,N}V_{p}(I_{i})
\\&\leq& K\max_{i=1,\ldots,N}V_{p}(I_{i}) 
\\&\leq& K \max_{i=1,\ldots,N} ||I_{i}||_{\infty}
\end{eqnarray*}
for some constant $K>0$. Here $V_{1+\epsilon}(Y_{s}^{\Delta,k},0,t)$ denotes the $1+\epsilon$ variation of $Y_{s}^{\Delta,k}$ from $0$ to $t$.

Then $||I_{i}||_{\infty}$ converges to 0  as $|\mathcal{P}|\searrow 0$ almost surely by (\ref{coeq1}) and (\ref{ppe2}) and since $\{\Phi_{t}^{\mathbf{D},\delta}\}$ forms a normal family, the rate of convergence is uniform; also, $\sigma^{\mathbf{H}_{s_{i-1}}}_{\delta_{k}}(u)=u+o(u)$ with $o(u)$ not depending on $i$.

Hence, $\{R_{k}^{i}+\epsilon''_{i}\}$ are uniformly asymptotically negligible. This implies that as $|\mathcal{P}|\searrow 0$, $\mathcal{Y}^{\mathcal{P}}_{t}$
converges to a continuous infinitely divisible process $\mathcal{Y}_{t}$ almost surely. Hence by the L\'{e}vy-Khintchine theorem and the fact that $\mathcal{Y}_{t}$ is of finite $(1+\epsilon)$-variation, we must have
\[\mathcal{Y}_{t}=b_{\Delta}t\]
for some $b_{\Delta}\in\mathbb{R}$.

In particular, by (\ref{pkeq1}),
\begin{eqnarray} Y^{\Delta,k}_{t}&=&\lim_{|\mathcal{P}|\searrow0}\sum_{i=1}^{N} Y^{\Delta,k}_{s_{i}}-Y^{\Delta,k}_{s_{i-1}}\nonumber\\
&=& b_{\Delta}t+\lim_{|P|\searrow 0}\sum_{i=1}^{N}\int_{0}^{s_{i}-s_{i-1}}e^{-2 o_{p}(1)}\left(\frac{\kappa_{\Delta}}{2}-3+ o_{p}(1)\right)\nonumber\\
&& \quad \quad \quad \quad \quad \quad \quad\times\frac{(\Phi_{s}^{\mathbf{H}_{s_{i-1}},\delta_{k}})''(W^{\Delta,k}_{\sigma^{\mathbf{H}_{s_{i-1}}}_{\delta_{k}}(s)})}{(\Phi_{s}^{\mathbf{H}_{s_{i-1}},\delta_{k}})'(W^{\Delta,k}_{\sigma^{\mathbf{H}_{s_{i-1}}}_{\delta_{k}}(s)})^{2}}ds.\label{bneq1} \end{eqnarray}
and hence is an element in Cameron-Martin space since the integrand is uniformly bounded for all $i$. and so it satisfies the Novikov condition.
Hence, using a Girsanov transformation, we can find a change of measure that makes  $W^{\Delta,k}_{t}=M^{\Delta,k}_{t}$. Since $M^{\Delta}_{t+s}-M^{\Delta}_{s}$, which is identically distributed to $M_{t}^{\Delta,\mathbf{H}_{s}}$ by construction, does not depend on $\gamma[0,s]$, this implies that we must have $\kappa_{\Delta}=6$ by (\ref{hoeq1}). Hence under our new measure, we must have
\[M^{\Delta,k}_{t}=\sqrt{6}B_{t}+o_{p}(1).\]
Hence under our original measure,
\begin{eqnarray*}&&
\Phi^{\mathbf{D},\delta_{k}} _{t}(W^{\Delta,k}_{\sigma^{\mathbf{D}}_{\delta_{k}}(t)})-\Phi_{0}^{\mathbf{D},\delta_{k}}(W^{\Delta,k}_{0})
\\&=&\int_{0}^{t}e^{-o_{p}(1)}\sqrt{\kappa_{\Delta}}d\widetilde{B}_{s}+\int_{0}^{t} (\Phi_{s}^{\mathbf{D},\delta_{k}})'(W^{\Delta,k}_{\sigma^{\mathbf{D}}_{\delta_{k}}(s)})dY^{\Delta,k}_{\sigma^{\mathbf{D}}_{\delta_{k}}(s)}\nonumber\\&&+\int_{0}^{t}e^{-2 o_{p}(1)}\left(o_{p}(1)\right)\frac{(\Phi_{s}^{\mathbf{D},\delta_{k}})''(W^{\Delta,k}_{\sigma^{\mathbf{D}}_{\delta_{k}}(s)})}{(\Phi_{s}^{\mathbf{D},\delta_{k}})'(W^{\Delta,k}_{\sigma^{\mathbf{D}}_{\delta_{k}}(s)})^{2}}ds.\nonumber\end{eqnarray*}

Then, by (\ref{bneq1}), we must have as $k\rightarrow \infty$, $Y_{t}^{\Delta,k}\rightarrow b_{\Delta}t$. By symmetry of the SqP in $\mathbb{H}$, we must have $b_{\Delta}=0$. We deduce that we must have
\[W^{\Delta}_{t}=\sqrt{6}B_{t}.\]
Hence, for any $\mathbf{D}=(D,a,b)$,
\[W^{\Delta,\mathbf{D}}_{t}=\sqrt{6}B_{t}\]
for $t\in[0,\tau^{D}]$.

To identify $W_{t}^{\Delta,\mathbb{D}}$ for $t>\tau^{\mathbf{D}}$, we can condition on $\gamma^{\mathbf{D}}[0,\tau^{\mathbf{D}}]$ and consider
\[\mathbf{D}'=(D\setminus \gamma^{\mathbf{D}}[0,\tau^{\mathbf{D}}],\gamma^{\mathbf{D}}(\tau^{\mathbf{D}}),\infty).\]
By repeating this argument inductively and using a Skorokhod embedding argument, we can deduce that
\[W^{\Delta,\mathbf{D}}_{t}=\sqrt{6}B_{t} \text{ for } t\in[0,\infty).\]
\end{proof}
\begin{corollary}
For any  $T>0$ and $\mathbf{D}=(D,a,b)\in\mathcal{D}$. Let $\xi^{\mathbf{D},\delta}_{t}$ denote the driving function of the SqP in $(D,a,b)$ on the lattice of mesh size $\delta$. Then  $\xi^{\mathbf{D},\delta}_{t}$ converges uniformly in distribution to $\sqrt{6}B_{t}$ on $[0,T]$ as $\delta\searrow0$.
\end{corollary}
\begin{proof}
For each $t\in[0,T]$, suppose that a sequence $(\delta_{k})$ with $\delta_{k}\searrow 0$ as $k\rightarrow\infty$ is such that $\xi_{t}^{\mathbf{D},\delta_{k}}$ converges in distribution to some function $U_{t}$. Then, by Theorem \ref{th2}, there exists subsequence $(\delta_{n_{k}})$ such that $\xi_{t_{n_{k}}}^{\mathbf{D},\delta}$ converges in distribution to $\sqrt{6}B_{t}$. This implies that $U_{t}=\sqrt{6}B_{t}$. Since this is true for every subsequence $(\delta_{k})$, this implies that $\xi_{t}^{\mathbf{D},\delta_{k}}$ converges in distribution to $\sqrt{6}B_{t}$. Hence, via a standard diagonalization argument, for $t\in\mathbb{Q}\cap [0,T]$, $\xi_{t}^{\mathbf{D},\delta}$ converges pointwise in distribution to $\sqrt{6}B_{t}$. Hence, by the Skorokhod representation theorem, 
we can define a probability space such that for $t\in\mathbb{Q}\cap [0,T]$, $\xi^{\mathbf{D},\delta}_{t}$ converges pointwise to $\sqrt{6}B_{t}$ almost surely.

Then on this probability space,  for any sequence $\xi^{\mathbf{D},\delta_{k}}_{t}$ with $\delta_{k}\searrow 0$ as $k\rightarrow\infty$, by Theorem \ref{th2}, we can find subsequence $\xi^{\mathbf{D},\delta_{n_{k}}}_{t}$ with $\delta_{n_{k}}\searrow 0$ as $k\rightarrow\infty$ that converges uniformly to $\sqrt{6}B_{t}$ on $\mathbb{Q}\cap[0,T]$ almost surely.  Suppose for contradiction that $\xi^{\mathbf{D},\delta}_{t}$ does not converge uniformly to $\sqrt{6}B_{t}$ almost surely on $\mathbb{Q}\cap [0,T]$. Then we have some $\epsilon>0$ and points $t_{n_{k}}\in\mathbb{Q}\cap [0,T]$ such that as $k\rightarrow \infty$,
\[|\xi^{\mathbf{D},\delta_{n_{k}}}_{t_{n_{k}}}-\sqrt{6}B_{t_{n_{k}}}|\geq \epsilon.\]
This implies that $\xi^{\mathbf{D},\delta_{n_{k}}}_{t}$ does not converge uniformly to $\sqrt{6}B_{t}$  which is a contradiction. Hence we get the desired result by continuity of $\xi_{t}^{\mathbf{D},\delta}$ and $\sqrt{6}B_{t}$.
\end{proof}
\section{Obtaining curve convergence from driving term convergence}
Let $\gamma $ be the SqP in $\mathbf{D}=(D,a,b)$ on the lattice mesh-size $\delta>0$ and let
$\Gamma(t)$ be the trace of chordal SLE$_{6}$ in $\mathbf{D}$.
Theorem \ref{th2} does not imply strong curve convergence i.e. that the law of $\gamma[0,\infty]$ converges weakly to the law of $\Gamma[0,\infty]$ with respect to the metric $\rho_{\mathbf{D}}$ given in (\ref{metric}). In order to get this convergence and prove Theorem \ref{main}, we can either use a similar method of calculating multi-arm estimates as in \cite{werner07} or apply Corollary 1.6 in \cite{SS10}. We will focus on the latter method.

To this end, it suffices to show that the radial driving function with respect to any internal point of the curve $\gamma $ and show this converges to $\sqrt{6}B_{t}$ (which is the radial driving function of chordal SLE$_{6}$ with respect to any internal point by Proposition 6.22 in  \cite{lawlerbook}).
This can be done by applying a formula similar to (\ref{weq6}) for the radial driving function and applying the same method mutatis mutandis. We obtain this formula as follows: consider $\mathbf{D}=(D,a,b)\in\mathcal{D}^{\mathbb{L}}$ where $\mathbb{L}$ is either the shifted brick-wall lattice or square lattice of mesh size $\delta$. Fix a point $x\in D$ not on the lattice. Then we can find a unique conformal map $\phi_{\mathbf{D}}$ which maps the unit disc $\mathbb{D}=\{z:|z|<1\}$ conformally onto $D$ with $\phi_{\mathbf{D}}(0)=x$ and $\phi_{\mathbf{D}}'(0)=1$. Then by the Schwarz-Christoffel formula, we can write $\phi_{\mathbf{D}}$
\begin{equation}\phi_{\mathbf{D}}'(z)^{2}=R\prod_{j=1}^{M}(z-e^{ir_{j}})^{\rho_{j}}\label{aeq2}\end{equation}
for some $e^{ir_{j}}\in \partial \mathbb{D}$, $\rho_{j}\in\mathbb{R}$, $M\in\mathbb{N}$ and $R\neq 0$.

Now, let $\nu$  be a simple path on the lattice from $a$ to $b$ in $D$, $(Z_{k})$ denote the vertices of $\nu$. Let $\gamma:[0,T_{x}]\rightarrow\mathbb{D}$  be the parametrization of $\nu$ by capacity such that $\phi_{\mathbf{D},x}^{-1}(\gamma[0,T_{x}])=\nu$. Here, parameterizing by capacity means that if we denote by $g_{t}$ the conformal maps of $D_{t}=\mathbb{D}\setminus\gamma(t)$ onto $\mathbb{D}$ normalized such that $g_{t}(0)=0$ and $g_{t}'(0)>0$, then we have
\[g_{t}'(0)=e^{t}.\]
Note that the above $g_{t}$ satisfies the radial Loewner differential equation:
\[\dot{g}_{t}(z)=g_{t}(z)\frac{e^{i\lambda_{t}}+g_{t}(z)}{e^{i\lambda_{t}}-g_{t}(z)},\]
where $g_{t}(\gamma(t))=e^{i\lambda_{t}}$ is the radial driving function.

Let $t_{0}=0<t_{1}<t_{2}<\ldots<t_{N}=T_{x}$  be the times such that $\phi_{\mathbf{D}}^{-1}(\gamma(t_{k}))=Z_{k}$. For any $t\geq 0$, we define $N(t)$ to be the largest $k$ such that $t_{k}<t$. Then for $1\leq k\leq N(t)$, we define
$a_{k}(t)$ and $b_{k}(t)$ such that $e^{ia_{k}(t)}$ and $e^{ib_{k}(t)}$ are the two preimages of $\phi_{\mathbf{D}}^{-1}(Z_{k})$ under $f_{t}$ such that $b_{k}(t)<a_{k}(t)$; $a_{k}(t)$, $b_{k}(t)$ are continuous and moreover, for any $t$, we can find an interval $I_{t}$ of length $2\pi$ such that for any $k=1,\ldots,N(t)$ and
\[a_{k}(t),b_{k}(t)\in I_{t}\]
For $j=1,\ldots,M$, we also define $r_{j}(t)$ to satisfy $e^{ir_{j}(t)}=g_{t}(e^{ir_{j}})$ such that $r_{j}(t)$ is continuous and also we can assume that $r_{j}(t)\in I_{t}$.
 Finally, we define
\[L_{k}=\left\{\begin{array}{ll}+1 & \text{ if }\nu\text{ turns right at }Z_{k}. \\0& \text{ if }\nu\text{ goes straight at }Z_{k}  \\ -1& \text{ if }\nu\text{ turns left at }Z_{k}.\end{array}\right. \]

Now let $f_{t}=g_{t}^{-1}$, $\phi_{\mathbf{D}}\circ f_{t}$ is also a map onto a polygonal domain and hence satisfies the Schwarz-Christoffel formula:
\begin{eqnarray}
&&\phi_{\mathbf{D}}'(f_{t}(z))^{2}f_{t}'(z)^{2}  \nonumber \\ &=&R_{t}\frac{(z-e^{i\lambda_{t}})^{2}}{(z-e^{ia_{1}(t)})(z-e^{ib_{1}(t)})}
\big(\prod_{k=2}^{N(t)}(\frac{z-e^{ib_{k}(t)}}{z-e^{ia_{k}(t)}})^{L_{k}}\big)
 \big(\prod_{j=1}^{M}(z-e^{ir_{j}(t)})^{\rho_{j}}\big).\label{aeq1}\end{eqnarray}
for some continuous function $R_{t}\neq 0$. Note that $R_{0}=R$. By the Schwarz reflection principle, we can extend $f_{t}$ to be analytic at a neighbourhood of $\infty$ such that $f_{t}(\infty)=\infty$ and $f_{t}'(\infty)=e^{t}$. Hence for some $k$,
\begin{eqnarray*}R_{t}&=&\lim_{z\rightarrow\infty} \frac{\phi_{\mathbf{D}}'(f_{t}(z))^{2}f_{t}'(z)^{2}}{z^{k}}
\\ &=& \phi_{\mathbf{D}}'(\infty)e^{t}
\end{eqnarray*}
by \eqref{aeq1}. This implies that we must have $R_{t}=Re^{t}$. Combining this fact with (\ref{aeq2}) and (\ref{aeq1}), we obtain
\begin{eqnarray}
&&f_{t}'(z)^{2}\prod_{j=1}^{N}(f_{t}(z)-e^{ir_{j}})^{\rho_{j}}\nonumber\\
&=&e^{t}\frac{(z-e^{i\lambda_{t}})^{2}}{(z-e^{ia_{1}(t)})(z-e^{ib_{1}(t)})}\big(\prod_{k=2}^{N(t)}(\frac{z-e^{ib_{k}(t)}}{z-e^{ia_{k}(t)}})^{L_{k}}\big)
 \big(\prod_{j=1}^{M}(z-e^{ir_{j}(t)})^{\rho_{j}}\big)\label{aeq3}\end{eqnarray}

By our choice of normalization and parametrization, $f_{t}'(0)=e^{t}$. Then by substituting $z=0$ to both sides we get
\[\mathrm{LHS}=\exp\big( t+i\sum_{j=1}^{M}\rho_{j}r_{j}\big),\]
\[\mathrm{RHS}=\exp\Big[t+i\Big(2\lambda_{t}-a_{1}(t)-b_{1}(t)
+\big(\sum_{k=2}^{n}L_{k}(b_{k}(t)-a_{k}(t))\big)+\big(\sum_{j=1}^{m}\rho_{j}r_{j}(t)\big)\Big)\Big].\]
Then by taking the branch of $\mathrm{arg}$ with principle values in $I_{t}$, we get
\[ \sum_{j=1}^{M}\rho_{j}r_{j} = 2\lambda_{t}-a_{1}(t)-b_{1}(t)+\big(\sum_{k=2}^{n}L_{k}(b_{k}(t)-a_{k}(t))\big)+\big(\sum_{j=1}^{m}\rho_{j}r_{j}(t)\big). \]
 Rearranging this, we get
\[\lambda_{t}=\frac{1}{2}\big[a_{1}(t)+b_{1}(t)+\big(\sum_{k=2}^{n}L_{k}(a_{k}(t)-b_{k}(t))\big)+\big(\sum_{j=1}^{m}\rho_{j}(r_{j}-r_{j}(t))\big)\big].\]
which we can utilize in the same way as the formula in (\ref{weq6}) in order to establish Theorem \ref{main}.
\section*{Subsequent work}
In a subsequent paper \cite{TYZ}, we will prove that the myopic random walk \cite{GS} also converges to SLE$_{6}$. The myopic random walk differs from the SqP by the fact that the myopic random walk can also go straight at every vertex of the path. We do this by constructing a new process from the +CBP and -CBP which can go straight at each free vertex.
\section*{Appendix: A version of the Kolmogorov-Centsov continuity theorem}
We need the following variation of the Kolmogorov-Centsov continuity theorem.
\begin{theorem}\label{KC}
Let $C, r>0$. Suppose that the process $M_{t}$ and a finite variation process $X_{t}$ satisfy, for $s,t\in[0,1]$ and for all sufficiently large $n$,
\[\mathbb{E}\Big[\left|\frac{M_{t}-M_{s}}{X_{t}-X_{s}}\right|^{n}\Big]< \infty\]
and
\[|X_{t}-X_{s}|\leq C|t-s|^{r} \text{ a.s.}\]
Then for $\gamma$ such that $r\epsilon>\gamma$ and $n$ such that $n(r\epsilon-\gamma)>1$, there exists a modification of the process $M_{t}$, which we also denote as $M_{t}$, that is a continuous process that satisfies
\[\sup_{D}\sum_{i=0}^{N-1} |M_{t_{i+1}}-M_{t_{i}}|^{\frac{1}{1-\epsilon}}\leq C_{1}B_{\epsilon}^{\frac{1}{n}}\sum_{i=1}^{\infty}i^{\theta}2^{-\gamma i} \sum_{k=1}^{2^{i}}\Big|X_{\frac{k+1}{2^{i}}}-X_{\frac{k}{2^{i}}}\Big|\]
where the supremum is taken over all finite partitions of $[0,1]$, $D=\{0=t_{0}<t_{1}<\ldots<t_{N}=1\}$. Also, $B_{\epsilon}$ is an almost surely finite random variable with
\[ \mathbb{E}[B_{\epsilon}^{q}]<\infty\]
for $q>1$.
In particular, $M_{t}$ is of finite $(1+\epsilon)$ variation for any $\epsilon>0$.
\end{theorem}
\begin{proof}
For any $\epsilon>0$ and for $m\in\mathbb{N}$, let
\[\mathcal{D}_{m}=\{\frac{k}{2^{m}},k=0,\ldots,2^{m}\}.\]
By the H\"{o}lder inequality, for $p,q>1$ with $\frac{1}{p}+\frac{1}{q}=1$,
\begin{eqnarray}\mathbb{E}\Big[ \left|\frac{M_{t}-M_{s}}{(X_{t}-X_{s})^{1-\epsilon}}\right|^{n}\Big]&\leq& \mathbb{E}\Big[\left| \frac{M_{t}-M_{s}}{X_{t}-X_{s}}\right|^{np}\Big]^{\frac{1}{p}}\mathbb{E}[|X_{t}-X_{s}|^{n\epsilon q}]^{\frac{1}{q}}\label{sdeq1}
\\&\leq& C_{n}|t-s|^{rn\epsilon}.\nonumber
\end{eqnarray}
We let \[B_{\epsilon}\triangleq\sum_{m=1}^{\infty}\sum_{k=1}^{2^{m}} \left(\frac{| M_{\frac{k+1}{2^{m}}}-M_{\frac{k}{2^{m}}}|}{| X_{\frac{k+1}{2^{m}}}-X_{\frac{k}{2^{m}}}|^{1-\epsilon}}\right)^{n}2^{\gamma m n}.\]
Now note that by H\"{o}lder's inequality, for any nonnegative numbers $a_{j}$ for $j\geq1$ and any $\theta >q-1>0$, we have
\begin{equation}\left(\sum_{j=1}^{\infty}a_{j}\right)^{q}\leq \left( \sum_{j=1}^{\infty} \frac{1}{j^{\frac{\theta}{q-1}}}\right)^{q-1}\sum_{j=1}^{\infty}j^{\theta}a_{j}^{q}=C_{q,\theta}\sum_{j=1}^{\infty}j^{\theta}a_{j}^{q}.\label{speq1}\end{equation}
Then for $q>1$, by (\ref{sdeq1}) and (\ref{speq1}),
\[\mathbb{E}[B_{\epsilon}^{q}]\leq C_{n}C_{q,\theta}\Big(\sum_{m=1}^{\infty}\frac{m^{\theta}}{2^{qm(n(r\epsilon-\gamma)-1)}}\Big)\]
We choose $\gamma$ such that $r\epsilon>\gamma$ and $n$ such that $n(r\epsilon-\gamma)>1$. In particular, $B_{\epsilon}$ is an almost surely finite random variable in $L^{q}$ and for any $k$ and $m$,
\[ \max_{1\leq k \leq 2^{m}} \frac{| M_{\frac{k+1}{2^{m}}}-M_{\frac{k}{2^{m}}}|}{| X_{\frac{k+1}{2^{m}}}-X_{\frac{k}{2^{m}}}|^{1-\epsilon}}\leq B_{\epsilon}^{\frac{1}{n}} 2^{-\gamma m}.\]
We now cover any subinterval $[s,t]$ of $[0,1]$ with dyadic intervals as follows: Let $m_{0}$ be the smallest $m$ such that $[s,t]$ contains a dyadic interval $[\frac{k}{2^{m}},\frac{k+1}{2^{m}}]$. Let $[\frac{k_{m_{0}}}{2^{m_{0}}},\frac{k_{m_{0}}+1}{2^{m_{0}}}]$ for some $0\leq k_{m_{0}}\leq 2^{m_{0}}-1$ be that interval. If $[s,t]=[\frac{k_{m_{0}}}{2^{m_{0}}},\frac{k_{m_{0}}+1}{2^{m_{0}}}]$, then the construction stops. Otherwise, we have $\frac{k_{m_{0}}+1}{2^{m_{0}}}<t$ and we carry on the construction. We can find $m_{1}>m_{0}$ and $0\leq k_{m_{1}}\leq 2^{m_{1}}-1$ such that $\frac{k_{m_{0}}+1}{2^{m_{0}}}=\frac{k_{m_{1}}}{2^{m_{1}}}$ and $[\frac{k_{m_{1}}}{2^{m_{1}}},\frac{k_{m_{1}}+1}{2^{m_{1}}}]$ has maximum length among all dyadic intervals $[\frac{k}{2^{m}},\frac{k+1}{2^{m}}]\subset [\frac{k_{m_{0}}+1}{2^{m_{0}}},t]$. Repeating this procedure, we obtain an increasing sequence $\{m_{i}\}$ such that
\[ \frac{k_{m_{0}}}{2^{m_{0}}}<\frac{k_{m_{0}}+1}{2^{m_{0}}}=\frac{k_{m_{1}}}{2^{m_{1}}}<\frac{k_{m_{1}}+1}{2^{m_{1}}}=\frac{k_{m_{2}}}{2^{m_{2}}}<\ldots<\frac{k_{m_{i}}}{2^{m_{i}}}\leq t,\]
with $\frac{k_{m_{i}}}{2^{m_{i}}}=t$ if the procedure ends after a finite number of steps or $\frac{k_{m_{i}}}{2^{m_{i}}}\rightarrow t$ otherwise. The same argument applies to the left-end points and thus we can find another increasing subsequence $\{m_{i}'\}$ such that
\[ \frac{k_{m_{0}}}{2^{m_{0}}}=\frac{k_{m_{1}'}+1}{2^{m_{1}'}}>\frac{k_{m_{1}'}}{2^{m_{1}'}}=\frac{k_{m_{2}'}+1}{2^{m_{2}'}}>\ldots>\frac{k_{m_{i}'}}{2^{m_{i}'}}\geq s,\]
with $\frac{k_{m_{i}'}}{2^{m_{i}'}}=s$ or $\frac{k_{m_{i}'}}{2^{m_{i}'}}\rightarrow s$. Note that $i\leq m_{i},m_{i}'$ for every $i$. For simplicity, we denote $s_{i}=\frac{k_{m_{i}}}{2^{m_{i}}}$ and $s_{-i}=\frac{k_{m_{i}'}}{2^{m_{i}'}}$ for $i=1,2,\ldots$. Then by the above construction, we have
\[[s,t]=\bigcup_{i\in\mathbb{Z}}[s_{i},s_{i+1}],\]
where the intervals $[s_{i},s_{i+1}]$ are dyadic intervals and are disjoint except at common endpoints.

By the triangle inequality and (\ref{speq1}), we have
\begin{eqnarray*}&&|M_{t}-M_{s}|^{q}\\
&\leq& C_{1}\Big[ \Big|M_{\frac{k_{m_{0}}+1}{2^{m_{0}}}}-M_{\frac{k_{m_{0}}}{2^{m_{0}}}}\Big|^{q}+\sum_{i=1}^{\infty}i^{\theta}\Big|M_{\frac{k_{m_{i}}+1}{2^{m_{i}}}}-M_{\frac{k_{m_{i}}}{2^{m_{i}}}}\Big|^{q}
\\ &&\qquad \qquad+ \sum_{i=1}^{\infty}i^{\theta}\Big|M_{\frac{k_{m_{i}'}+1}{2^{m_{i}'}}}-M_{\frac{k_{m_{i}'}}{2^{m_{i}'}}}\Big|^{q}\Big]
\\&\leq& C_{1}B_{\epsilon}^{\frac{1}{n}}\Big[ 2^{-\gamma m_{0}}\Big|X_{\frac{k_{m_{0}}+1}{2^{m_{0}}}}-X_{\frac{k_{m_{0}}}{2^{m_{0}}}}\Big|^{q(1-\epsilon)}+\sum_{i=1}^{\infty}i^{\theta}2^{-\gamma m_{i}}\Big|X_{\frac{k_{m_{i}}+1}{2^{m_{i}}}}-X_{\frac{k_{m_{i}}}{2^{m_{i}}}}\Big|^{q(1-\epsilon)}\\
&&\qquad \qquad + \sum_{i=1}^{\infty}i^{\theta}2^{-\gamma m_{i}'}\Big|X_{\frac{k_{m_{i}'}+1}{2^{m_{i}'}}}-X_{\frac{k_{m_{i}'}}{2^{m_{i}'}}}\Big|^{q(1-\epsilon)}\Big]
\end{eqnarray*}
for some constant $C_{1}$ depending only on $p$ and $\theta$. Hence for any finite partition  of $[0,1]$, $D=\{0=t_{0}<t_{1}<\ldots<t_{N}=1\}$, we apply the above inequality to each interval $[t_{l-1},t_{l}]$ of the partition to get
\[\sup_{D}\sum_{i=0}^{N-1} |M_{t_{i+1}}-M_{t_{i}}|^{q}\leq C_{1}B_{\epsilon}^{\frac{1}{n}}\sum_{i=1}^{\infty}i^{\theta}2^{-\gamma i} \sum_{k=1}^{2^{i}}\Big|X_{\frac{k+1}{2^{i}}}-X_{\frac{k}{2^{i}}}\Big|^{q(1-\epsilon)}\]
 Picking $q=1/(1-\epsilon)$ and using the fact that $X_{t}$ is of finite variation, we find that $M_{t}$ is a finite $(1+\epsilon)$-variation process almost surely.

\end{proof}

\end{document}